\documentclass[a4paper, oneside, reqno]{amsart}

\usepackage{amsmath, bbm, mathtools}
\usepackage{amsthm, thmtools}
\usepackage{braket}
\usepackage[dvipsnames]{xcolor}
\usepackage[a4paper]{geometry}
\usepackage{caption}
\usepackage{subcaption}
\usepackage{algorithm, algpseudocodex}
\usepackage{hyperref}
\usepackage[autostyle=false, style=english]{csquotes}
\MakeOuterQuote{"}

\newcommand{\bR}{\mathbb{R}}
\newcommand{\bN}{\mathbb{N}}

\newcommand{\cF}{\mathcal{F}}
\newcommand{\cA}{\mathcal{A}}

\newcommand{\cB}{\mathcal{B}}
\newcommand{\cO}{\mathcal{O}}

\newcommand{\cC}{\mathcal{C}}
\newcommand{\cV}{\mathcal{V}}

\newcommand{\hatC}[1]{\widehat{C}_{#1}}
\newcommand{\hatL}{\widehat{L}}
\newcommand{\hatI}{\widehat{I}}
\newcommand{\hatJ}{\widehat{J}}
\newcommand{\tildeJ}{\widetilde{J}}
\newcommand{\hatgamma}{\widehat{\gamma}}
\newcommand{\hatlambda}{\widehat{\lambda}}
\newcommand{\hatphi}{\widehat{\varPhi}}
\newcommand{\tildealpha}{\widetilde{\alpha}}
\newcommand{\tildebeta}{\widetilde{\beta}}
\newcommand{\tildegamma}{\widetilde{\gamma}}

\newcommand{\Prob}{\mathbb{P}}
\newcommand{\PMeas}{\mathbb{P}}
\newcommand{\QMeas}{\mathbb{Q}}
\newcommand{\E}{\mathbb{E}}
\newcommand{\hatY}{\widehat{Y}}

\newcommand{\hatSf}{\widehat{S}^{f}}
\newcommand{\hatSc}{\widehat{S}^{c}}
\newcommand{\barX}{\overline{X}}
\newcommand{\Var}{\mathbb{V}}
\newcommand{\ind}{\mathbbm{1}}
\newcommand{\comp}{\mathsf{c}}

\DeclareMathOperator{\di}{d \!}

\DeclareMathOperator{\e}{e}
\DeclareMathOperator{\Id}{I_{d}}

\DeclareMathOperator{\Cost}{Cost}

\DeclareMathOperator{\vol}{vol}
\DeclareMathOperator{\sech}{sech}

\newcommand{\abs}[1]{\left \vert #1 \right \vert}
\newcommand{\inner}[1]{\left \langle #1 \right \rangle}
\newcommand{\norm}[1]{\left \Vert #1 \right \Vert}
\newcommand{\floor}[1]{\left \lfloor #1 \right \rfloor}
\newcommand{\ceil}[1]{\left \lceil #1 \right \rceil}

\declaretheorem[name=Theorem]{thm}

\declaretheorem[name=Assumption, style=definition]{asm}

\declaretheorem[name=Proposition]{prop}
\declaretheorem[name=Lemma]{lemma}

\theoremstyle{remark}
\newtheorem*{rem}{Remark}

\title[Higher-order MLMC method for approximating invariant measures]{Higher-order spring-coupled multilevel Monte Carlo method for invariant measures}

\author{H{\aa}kon Hoel}
\address{Department of Mathematics, University of Oslo, Norway}
\email{haakonah@math.uio.no}

\author{Sankarasubramanian Ragunathan$^\star$}
\address{Chair of Mathematics for Uncertainty Quantification, RWTH Aachen University, Germany}
\email{ragunathan@uq.rwth-aachen.de}

\thanks{$^\star$Corresponding author: S. Ragunathan (ragunathan@uq.rwth-aachen.de)}

\keywords{change of measure, ergodic measure, SDE, order 1.5 strong Itô--Taylor method, Girsanov's theorem, MLMC method}

\subjclass[2020]{60H10, 60H35, 65C05, 37M25}

\begin{document}
    \begin{abstract}
        A higher-order change-of-measure multilevel Monte Carlo (MLMC) method is developed for computing weak approximations of the invariant measures of SDE with drift coefficients that do not satisfy the contractivity condition. This is achieved by introducing a spring term in the pairwise coupling of the MLMC trajectories employing the order 1.5 strong It\^o--Taylor method. Through this, we can recover the contractivity property of the drift coefficient while still retaining the telescoping sum property needed for implementing the MLMC method.

        We show that the variance of the change-of-measure MLMC method grows linearly in time $T$ for all $T > 0$, and for all sufficiently small timestep size $h > 0$. For a given error tolerance $\epsilon > 0$, we prove that the method achieves a mean-square-error accuracy of $\cO(\epsilon^2)$ with a computational cost of $\cO(\epsilon^{-2} \abs{\log \epsilon}^{3/2} (\log \abs{\log \epsilon})^{1/2})$ for uniformly Lipschitz continuous payoff functions and $\cO \big( \epsilon^{-2} \abs{\log \epsilon}^{5/3 + \xi} \big)$ for discontinuous payoffs, respectively, where $\xi > 0$. We also observe an improvement in the constant associated with the computational cost of the higher-order change-of-measure MLMC method, marking an improvement over the Milstein change-of-measure method in the aforementioned seminal work by M. Giles and W. Fang. Several numerical tests were performed to verify the theoretical results and assess the robustness of the method.
    \end{abstract}
    
    \maketitle
    
    \bibliographystyle{amsplain}
    \section{Introduction}
    In this paper, we construct a higher-order multilevel Monte Carlo (MLMC) method for weak approximations of the stationary distribution of $d$-dimensional stochastic differential equations of the form 
    \begin{equation}
        \begin{split}
            \di X(t) &= a(X(t)) \di t + \di W(t) \\
            X(0) &= x_{0} \in \bR^{d}.
        \end{split}
        \label{eq:SDE}
    \end{equation}
    Here, $a \colon \bR^{d} \to \bR^{d}$ denotes the drift coefficient, $W(t)$ is a $d$-dimensional standard Wiener process on a filtered, complete probability space $\big( \Omega, \cF, (\cF_{t})_{t \geq 0}, \Prob \big)$, and the deterministic initial condition $x_{0}$ is given. More precisely, we develop a higher-order change-of-measure numerical method for computing expected values of the form
    \begin{equation*}
        \E \big[ \varPhi(X_{\infty}) \big] \coloneqq \int_{\bR^{d}} \varPhi(x) \,  \pi( \di x) = \lim\limits_{t \to \infty} \E \big[ \varPhi (X_{t}) \big],
    \end{equation*}
    where $\varPhi \colon \bR^{d} \to \bR$ represents a quantity of interest (QoI) and $\pi$ denotes the invariant distribution of~\eqref{eq:SDE}. Further details on the regularity and contractivity of the drift coefficient, ensuring that the invariant distribution is well-defined, are provided in Section~\ref{sec:notation_theorem}.

    The main contribution of this work is to show that, under sufficient regularity conditions, the higher-order change-of-measure MLMC method weakly approximates $\E[\varPhi(X_{\infty})]$ with a mean-square-error accuracy of $\cO(\epsilon^2)$ at the computational cost of $\cO (\epsilon^{-2} \abs{\log \epsilon}^{3/2} (\log \abs{\log \epsilon})^{1/2})$, and $\cO \big( \epsilon^{-2} \abs{\log \epsilon}^{5/3 + \xi} \big)$ for all $\xi > 0$, for a payoff function $\varPhi$ that is uniformly Lipschitz continuous, and discontinuous, respectively. The higher-order change-of-measure scheme also has a better constant associated with the computational cost than the Milstein scheme. This is an improvement from $\cO (\epsilon^{-2} \abs{\log \epsilon}^{2})$ obtained in \cite[Theorem~5]{fang2019multilevel}.

    Computing weak approximations of the invariant measure have various applications in physics, biology, statistical mechanics and other fields alike. In molecular dynamics, ergodicity plays an important role in determining the properties of the system \cite{frenkel2001understanding}. In quantum mechanics, the ergodic measure provides a statistical description of the spacing between the energy levels of many-body quantum systems \cite{berry1977regular, leitner2015quantum}. In molecular biology, the ergodic measure relates to the motion of bio-molecules and dictates the essential functions of a cell \cite{li2022role}.

    The invariant measure can be computed either using Monte Carlo techniques, by the numerical simulation of SDE paths, or by solving the associated partial differential equation; the Fokker--Planck equation \cite[equation~(8.6)]{weinan2021applied}. In high-dimensional state space, computing the closed-form analytical solution to the Fokker--Planck equation is oftentimes not tractable, if not impossible, due to the curse of dimensionality. Since we are only interested in computing weak approximations of the invariant measure, we use Monte Carlo simulations because it is numerically tractable and scales well to higher dimensions.

    Note that the naive Monte Carlo method can be computationally expensive as one might need many Monte Carlo samples to achieve $\cO(\epsilon^{2})$ mean-squared error of the estimated QoI. In \cite{giles2008multilevel, giles2015multilevel}, the MLMC method is formulated to compute weak approximations of a QoI w.r.t the measure of an SDE. For the case of a uniformly Lipschitz continuous payoff function, and using the Euler--Maruyama scheme, it was shown that a mean-squared error of $\cO ( \epsilon^{-2} )$ can be achieved using a computational cost of $\cO ( \epsilon^{-2} \abs{\log \epsilon}^{2} )$, which is an improvement from $\cO ( \epsilon^{-3} )$ using the naive Monte Carlo method. The study was later extended to the adaptive Euler-Maruyama scheme with non-uniform and path-dependent time-step sizes in \cite{Raul2012, Szepessy2014}. The results from these works form the basis for the development of an MLMC method to compute weak approximations of invariant measures.
    
    Following the idea above, the MLMC method was extended to compute weak approximations of invariant measures in \cite{zygalakis2020} for SDE with drift coefficients that satisfy the contractivity condition, i.e. the drift coefficient is one-sided Lipschitz continuous with a negative one-sided Lipschitz constant. The authors show that under appropriate pairwise contracting coupling of the trajectories of the MLMC estimator, one can obtain a uniform-in-time estimate of the MLMC variance. The Milstein scheme is proposed in \cite{weng2019invariant} to compute weak approximations of invariant measures of SDE. Assuming that the diffusion coefficient satisfies the commutativity condition \cite[equation~(3.13)]{Kloeden1992} and the drift coefficient satisfies the contractivity condition, the authors prove a strong convergence rate of order 1. The authors also show that the probability density function associated with the numerical solution converges exponentially to the true invariant distribution.

    Higher-order numerical methods for approximating the ergodic measure of SDE is developed in \cite{abdulle2014high}. Sufficient conditions necessary to weakly approximate ergodic measure of SDE with a high order of accuracy, irrespective of the weak order of the numerical scheme, are developed. The convergence properties of a wide class of adaptive numerical approximations of SDE are studied in \cite{foster2023convergence}. The authors show that as long as the adaptive numerical method does not introduce any bias in the computation of the "Lévy area", the numerical approximation converges to the actual SDE dynamics. The seminal works provide a framework for computing weak approximations of invariant measures of SDE and to establish convergence results for methods employing adaptive numerical integrators. This is of interest to our current work and for future extensions.

    Explicit adaptive time-stepping schemes employing the Euler--Maruyama scheme were developed in \cite{Lord2018, lemaire2007} to compute the expected value of QoIs that depend on the solution of SDE with non-uniformly Lipschitz continuous drift coefficients. The fundamental idea behind adaptive time-stepping is to use a time-step size that depends on the drift term. Using a smaller time-step size as the magnitude of the drift term becomes large ensures that the numerical solution is stable and bounded. The idea was extended in \cite{wang2016_EM1, wang2017_EM2, wang2018} for the MLMC method. The authors prove theoretically that using adaptive time-stepping, under suitable pairwise contracting coupling of the MLMC trajectories, the $p$-th moment of the difference between the true SDE solution and the numerical solution is uniformly bounded in time $T$ for all $T > 0$, provided the drift coefficient satisfies the contractivity condition.

    In all the above works, it is assumed that the drift coefficient satisfies the contractivity condition. In practice, only a small class of SDE have drift coefficients that satisfy the contractivity condition. This impedes the effective implementation of the above methods for solving SDE that do not satisfy the assumption. A change of measure technique has hence been developed in \cite{fang2019multilevel} to extend the invariant measure computation to drift coefficients that satisfy the one-sided Lipschitz continuity condition. The contractivity property of the drift coefficient is recovered by constructing a numerical scheme with a "spring" term. The authors also show a linear growth in the variance of the change-of-measure MLMC estimator with respect to time $T$ for all $T > 0$. A similar approach is utilized in \cite{Ballesio2023} and \cite{giles2021} to develop coupled particle filters that work for It\^o diffusion that have unstable dynamics, and to compute pathwise sensitivities of SDE, respectively.

    The rest of the paper is organized as follows: Sections~\ref{sec:notation_theorem} and \ref{sec:theoretical_results} introduce the necessary notation, regularity assumptions, and the main theoretical results on convergence and computational cost associated with the higher-order change-of-measure MLMC scheme, respectively. In Section~\ref{sec:numerical_results}, we will look at how the higher-order change-of-measure algorithm performs by considering different numerical test cases. In Section~\ref{sec:conclusion}, we will discuss the conclusions that can be drawn from the current work and the possible future extensions. Finally, Section~\ref{sec:theoretical_proofs} contains the proofs for the theorems stated in Section~\ref{sec:theoretical_results}.
    
    \section{Notation and the higher-order change-of-measure method}
    \label{sec:notation_theorem}
    In this section, we state the necessary notation, and assumptions on the SDE drift coefficient and its derivatives to ensure the existence and uniqueness of a solution to the SDE~\eqref{eq:SDE}, and to ensure that it has an invariant measure. We also state the higher-order change-of-measure scheme and the corresponding theoretical results for the convergence and the computational cost.

    \subsection{Notation and assumptions}
    Let $\inner{v, w}$ denote the inner product between two vectors $v, w \in \bR^{d}$ defined as
    \[ \inner{v, w} \coloneqq v_{1} w_{1} + v_{2} w_{2} + \cdots + v_{d} w_{d}, \]
    where $v_{i}$ represents the $i$-th component of the vector $v \in \bR^{d}$ for all $i = 1, \ldots, d$. The standard Euclidean norm is induced from the vector inner product and is defined as
    \[ \norm{v} \coloneqq \sqrt{\sum_{i = 1}^{d} v_{i}^{2}}, \quad \forall v \in \bR^{d}. \]
    The definition of the standard Euclidean norm can be extended similarly to a $k$-th order tensor $A \in (\bR^{d})^{\otimes k}$ \[ \norm{A} \coloneqq \sqrt{\sum_{i_{1}, i_{2}, \cdots, i_{k} = 1}^{d} \abs{A_{i_{1} i_{2} \cdots i_{k}}}^{2}}, \qquad \forall k \in \bN \setminus \{ 1 \}, \] where $\otimes$ denotes the tensor outer product \cite{qi2017tensor}.

    For any integer $k\geq 1$, let $C^{k}_b(\bR^d)$ denote the set of scalar-valued functions with continuous and uniformly bounded partial derivatives up to and including order $k$. We make the following assumption on the regularity of the drift coefficient of SDE~\eqref{eq:SDE},
    \begin{asm}[Regularity of $a(x)$]
        \label{asmp:lipschitz_a}
        The components of the drift coefficient $a_{i} \in C_{b}^{3} \big( \bR^{d} \big)$, for all $i = 1, \ldots, d$. We also assume that there exists a constant $K_{b} > 0$ such that
        \begin{equation*}
            \sup_{x \in \bR^{d}} \max \Bigg( \abs{\frac{\partial a_{i}}{\partial x_{j}} (x)}, \abs{\frac{\partial^{2} a_{i}}{\partial x_{j} \partial x_{k}} (x)}, \abs{\frac{\partial^{3} a_{i}}{\partial x_{j} \partial x_{k} \partial x_{l}} (x)} \Bigg) \eqqcolon K_{b} < \infty, \quad \forall i, j, k = 1, \ldots, d.
        \end{equation*}
    \end{asm}
    From Assumption~\ref{asmp:lipschitz_a}, using the Mean value theorem,
    the following implications hold for all $i = 1, \ldots, d$
    \begin{equation*}
        \begin{split}
            \norm{a(x) - a(y)} &= \sqrt{\sum_{i = 1}^{d} \abs{a_{i}(x) - a_{i}(y)}^{2}} \leq \sqrt{d} K_{b} \norm{x - y}, \\
            ~
            \norm{Da(x) - Da(y)} &\leq d K_{b} \norm{x - y}, \\
            ~
            \norm{D^{2}a(x) - D^{2}a(y)} &\leq d^{3/2} K_{b} \norm{x - y},
        \end{split}
    \end{equation*}
    where
    \begin{align*}
        Da(x) &\coloneqq \frac{\partial a_{i}}{\partial x_{j}}(x) \; e_{i} \otimes e_{j}, & Da[u, v] &= \sum_{i,j = 1}^{d}\frac{\partial a_{i}}{\partial x_{j}} u_{i} v_{j}, \\
        D^{2}a(x) &\coloneqq \frac{\partial^{2} a_{i}}{\partial x_{j} \partial x_{k}}(x) \; e_{i} \otimes e_{j} \otimes e_{k}, & D^{2}a[u, v, w] &= \sum_{i,j,k = 1}^{d} \frac{\partial^{2} a_{i}}{\partial x_{j} \partial x_{k}} u_{i} v_{j} w_{k}, \\
        D^{3}a(x) &\coloneqq \frac{\partial^{3} a_{i}}{\partial x_{j} \partial x_{k} \partial x_{l}}(x) \; e_{i} \otimes e_{j} \otimes e_{k} \otimes e_{l}, & D^{3}a[u, v, w, z] &= \sum_{i,j,k,l = 1}^{d} \frac{\partial^{3} a_{i}}{\partial x_{j} \partial x_{k} \partial x_{l}} u_{i} v_{j} w_{k} z_{l}
    \end{align*}
    for all $u, v, w, z \in \bR^{d}$, where $e_{k}$ denotes the standard d-dimensional basis vector, i.e. a column vector with 1 in the $k$-th entry and 0 everywhere else, and $\otimes$ denotes the tensor product. To keep things concise, we also introduce the short-hand notation $Da(x) u \coloneqq Da(x)[\cdot, u]$ and $(Da(x))^{\top} u \coloneqq Da(x)[u, \cdot]$ for all $u \in \bR^{d}$. Similarly, we also introduce the notation $D^{2}a(x) u v \coloneqq D^{2}a(x) [., u, v]$ and $D^{2}a(x) u^{2} \coloneqq D^{2}a(x) [., u, u]$ for all $u, v \in \bR^{d}$.
    
    Assumption~\ref{asmp:lipschitz_a} ensures that there exists a unique solution to the SDE~\eqref{eq:SDE}, \cite[Theorem 5.2.1]{oksendal2013stochastic}. Assumption~\ref{asmp:lipschitz_a} also implies that there exists a constant $\lambda > 0$ such that the following inequality holds
    \begin{equation*}
        \begin{aligned}[c]
            \inner{x - y, a(x) - a(y)} &\leq \lambda \norm{x - y}^{2}, \quad \forall x, y \in \bR^{d},
        \end{aligned}
    \end{equation*}
    where $\lambda$ denotes the one-sided Lipschitz constant.
    \begin{asm}[uniform Lipschitz continuity of $D a(x) a(x)$]
        \label{asmp:aD_lipschitz}
        There exists a constants $L > 0$, such that
        \[ \abs{ \Big( a_{i} \frac{\partial a_{k}}{\partial x_{i}} \Big) (x) - \Big( a_{i} \frac{\partial a_{k}}{\partial x_{i}} \Big) (y) } \leq L \norm{x - y}, \qquad \forall x, y \in \bR^{d} \] for all $i, k = 1, \ldots, d$.
    \end{asm}
    Let $\cA$ denote the infinitesimal generator of the It\^o process~\eqref{eq:SDE} defined as
    \begin{align}
        \label{eq:A_operator}
        \cA \coloneqq \sum_{i = 1}^{d} a_{i}(x) \frac{\partial}{\partial x_{i}} + \frac{1}{2} \sum_{i = 1}^{d} \frac{\partial^{2}}{\partial x_{i}^{2}} = a(x) \cdot D + \frac{1}{2} \Delta,
    \end{align}
    where $\Delta$ represents the Laplacian operator. The Laplacian applied to the vector-valued drift coefficient $a$ is defined as the following component-wise operation
    \[ \displaystyle \Delta a(x) \coloneqq \Bigg( \sum_{i = 1}^{d} \frac{\partial^{2} a_{1}}{\partial x_{i}^{2}}(x), \cdots, \sum_{i = 1}^{d} \frac{\partial^{2} a_{d}}{\partial x_{i}^{2}}(x) \Bigg)^{\top}. \]
    The infinitesimal generator is needed to construct the order 1.5 strong It\^o--Taylor scheme and occurs naturally due to the It\^o--Taylor expansion.
    \begin{prop}[uniformly Lipschitz continuity of $\cA$]
        Let Assumptions~\ref{asmp:lipschitz_a} and \ref{asmp:aD_lipschitz} hold. Then, there exists a constant $\hatL > 0$ such that
        \[ \norm{\cA a(x) - \cA a(y)} \leq \hatL \norm{x - y}, \quad \forall x, y \in \bR^{d}. \]
    \end{prop}
    \begin{proof}
        Consider the operator $\cA$ applied to the $k$-th component of the drift coefficient $a$. We then obtain
        \begin{equation*}
            \begin{split}
                \abs{\cA a_{k}(x) - \cA a_{k}(y)} &= \abs{\sum_{i = 1}^{d} \bigg( a_{i}(x) \frac{\partial a_{k}(x)}{\partial x_{i}} - a_{i}(y) \frac{\partial a_{k}(y)}{\partial x_{i}} \bigg) + \frac{1}{2} \sum_{i = 1}^{d} \bigg( \frac{\partial^{2} a_{k}(x)}{\partial x_{i}^{2}} - \frac{\partial^{2} a_{k}(y)}{\partial x_{i}^{2}} \bigg)} \\
                &\leq \sum_{i = 1}^{d} \abs{ a_{i}(x) \frac{\partial a_{k}(x)}{\partial x_{i}} - a_{i}(y) \frac{\partial a_{k}(y)}{\partial x_{i}} } + \frac{1}{2} \sum_{i = 1}^{d} \abs{ \frac{\partial^{2} a_{k}(x)}{\partial x_{i}^{2}} - \frac{\partial^{2} a_{k}(y)}{\partial x_{i}^{2}} } \\
                &\leq \sum_{i = 1}^{d} L \norm{x - y} + \frac{1}{2} \sum_{i = 1}^{d} K_{b} \norm{x - y} \\
                &= \hatL_{k} \norm{x - y},
            \end{split}
        \end{equation*}
        where $\hatL_{k} \coloneqq d \Big( L + \frac{1}{2} K_{b} \Big)$. We then have
        \begin{equation*}
            \begin{split}
                \norm{\cA a(x) - \cA a(y)} = \sqrt{\sum_{k = 1}^{d} \abs{\cA a_{k}(x) - \cA a_{k}(y)}^{2} } \leq \underbrace{d^{3/2} \Big( L + \frac{1}{2} K_{b} \Big)}_{\eqqcolon \hatL} \norm{x - y}
            \end{split}
        \end{equation*}
    \end{proof}
    The one-sided Lipschitz property of the operator $\cA$ acting on the drift coefficient $a$ follows from the uniform Lipschitz continuity of $\cA$, i.e. there exists a constant $\hatlambda > 0$ such that
    \begin{equation*}
        \inner{x - y, \cA a(x) - \cA a(y)} \leq \hatlambda \norm{x - y}^{2}.
    \end{equation*}
    
    The regularity assumptions on the drift coefficient and its derivatives ensure that the order 1.5 strong It\^o--Taylor scheme is well defined and is required for the implementation of the higher-order change-of-measure scheme \cite[Theorem 10.6.3]{Kloeden1992}.
    
    Note that the regularity assumption on the drift coefficient alone does not guarantee the existence of an invariant measure. To this end, we assume that the drift coefficient also satisfies the dissipativity condition given below
    \begin{asm}[Dissipativity condition]
        \label{asmp:dissipativity_a}
        The drift coefficient satisfies the dissipativity condition, i.e. there exists constants $\alpha, \beta > 0$ such that
        \begin{equation*}
             \inner{x, a(x)} \leq - \alpha \norm{x}^{2} + \beta, \qquad \forall x \in \bR^{d}
        \end{equation*}
    \end{asm}
    Assumptions \ref{asmp:lipschitz_a} and \ref{asmp:dissipativity_a} are sufficient to guarantee that the SDE \eqref{eq:SDE} has an invariant measure $\pi(x)$ and that the solution to the SDE converges exponentially to $\pi(x)$ \cite{mattingly2002ergodicity, meyn1992stability, meyn1993stability_2, meyn1993stability_3}, i.e.
    \begin{equation}
        \label{eq:ergodic_geomtric}
        \abs{\E [ \varPhi(X_{T}) ] - \E [ \varPhi(X_{\infty}) ]} \leq \mu^{\star} \e^{- \lambda^{\star} T},
    \end{equation}
    for some constants $\mu^{\star}, \lambda^{\star} > 0$. More details regarding the sufficient conditions for an It\^o process to be ergodic can be found in \cite{khasminski2018, Kloeden1992}.

    \subsection{The higher-order change-of-measure method}
    Having made the necessary assumptions needed to ensure the well-posedness of a solution to the SDE~\eqref{eq:SDE}, and the existence of an invariant measure $\pi(x)$, the expected value of a QoI under the $\QMeas$ measure, given by $\E^{\QMeas} [\varPhi(X_{T})]$, can be approximated using the standard MLMC method as
    \begin{equation*}
        \E^{\QMeas} [ \varPhi(\barX_{T}^{L}) ] = \E^{\QMeas} [ \varPhi(\barX_{T}^{0}) ] + \sum_{\ell = 1}^{L} \E^{\QMeas} [ \varPhi(\barX_{T}^{f, \ell})  - \varPhi(\barX_{T}^{c, \ell - 1}) ]
    \end{equation*}
    where $L$ denotes the total number of levels in the MLMC estimator, $T > 0$ represents the terminal time, and $\barX^{f}$ and $\barX^{c}$ denote the fine and the coarse trajectories, respectively. Both the fine and the coarse trajectories are numerically simulated using the order 1.5 strong It\^o--Taylor scheme given by
    \\[0.3cm]
    {\scshape Fine trajectory, $\QMeas$-measure:}
    \begin{align}
        \label{eq:fine_odd_wo_spring}
        \begin{split}
            \barX^{f}(t_{2n + 1}) &= \barX^{f}(t_{2n}) + h a(\barX^{f}(t_{2n})) + \Delta W_{2n}^{\QMeas} + Da(\barX^{f}(t_{2n})) \Delta Z_{2n}^{\QMeas} + \frac{h^{2}}{2} \cA a (\barX^{f}(t_{2n})),
        \end{split}
        \\
        \label{eq:fine_even_wo_spring}
        \begin{split}
            \barX^{f}(t_{2n + 2}) &= \barX^{f}(t_{2n + 1}) + h a(\barX^{f}(t_{2n + 1})) + \Delta W_{2n + 1}^{\QMeas} + Da(\barX^{f}(t_{2n + 1})) \Delta Z_{2n + 1}^{\QMeas} \\
            &\quad + \frac{h^{2}}{2} \cA a (\barX^{f}(t_{2n + 1})),
        \end{split}
    \end{align}
    {\scshape Coarse trajectory, $\QMeas$-measure:}
    \begin{align}
        \label{eq:coarse_even_wo_spring}
        \begin{split}
            \barX^{c}(t_{2n + 2}) &= \barX^{c}(t_{2n}) + 2h a(\barX^{c}(t_{2n})) + \Big( \Delta W_{2n + 1}^{\QMeas} + \Delta W_{2n}^{\QMeas} \Big) \\
            &\quad + Da(\barX^{c}(t_{2n})) \Big( \Delta Z_{2n + 1}^{\QMeas} + \Delta Z_{2n}^{\QMeas} + h \Delta W_{2n}^{\QMeas} \Big) + 2 h^{2} \cA a (\barX^{c}(t_{2n})),
        \end{split}
    \end{align}
    for $n = 0, 1, \ldots, N/2 - 1$ where $t_{n} = n h$, $N = T / h \in \bN$, $h = 2^{-\ell} h_{0}$ denotes the uniform time-step size at level $\ell$ of the multilevel estimator, $h_{0}$ denotes the time-step size at level $\ell = 0$, $\barX^{f}(0) = \barX^{c}(0) = x_{0}$, and $\Delta Z_{n}$ is a random variable defined as
    \[ \Delta Z_{n}^{k} \coloneqq \int_{t_{n}}^{t_{n + 1}} \int_{t_{n}}^{s} \di W_{s}^{k} \di s, \qquad \forall k = 1, \ldots, d.  \]
    The tuple of correlated random variables $(\Delta W_{n}^{\QMeas}, \Delta Z_{n}^{\QMeas})$, under the $\QMeas$ measure, can be generated by
    \begin{equation*}
        \Delta W_{n}^{\QMeas} \coloneqq \Delta U_{1, n}, \qquad \Delta Z_{n}^{\QMeas} \coloneqq \frac{1}{2} h \Big( \Delta U_{1, n} + \frac{1}{\sqrt{3}} \Delta U_{2, n} \Big),
    \end{equation*}
    where $\Delta U_{1, n}, \Delta U_{2, n}$ are independent identically distributed (i.i.d.) $\mathcal{N}^{\QMeas}(0, h \Id)$ samples, and $\Id$ denotes the $d \times d$ identity matrix. Note that
    \begin{equation*}
        \E \big[ \inner{\Delta W_{i}^{\QMeas}, \Delta W_{j}^{\QMeas}} \big] = d h \delta_{ij}, \qquad \E \big[ \inner{\Delta Z_{i}^{\QMeas}, \Delta Z_{j}^{\QMeas}} \big] = \frac{d h^{3}}{3} \delta_{ij}, \qquad \E \big[ \inner{\Delta W_{i}^{\QMeas}, \Delta Z_{j}^{\QMeas}} \big] =  \frac{dh^{2}}{2} \delta_{ij},
    \end{equation*}
    where \[ \delta_{ij} = \begin{cases}
        1, & i = j \\
        0, & i \neq j
    \end{cases} \]
    represents the Kronecker delta function.
    
    We cannot use the standard MLMC algorithm for computing the weak approximations of the ergodic measure as the fine and the coarse trajectories of the MLMC estimator may drift away from each other. This can lead to an exponential growth in the multilevel variance w.r.t. the terminal time $T$, which can lead to an increase in the computational cost needed to achieve good variance estimates \cite{wang2018, fang2019multilevel}. To circumvent this issue, we introduce a spring term in both the fine and the coarse trajectories.

    Let $S > 0$ denote the spring constant. The order 1.5 strong It\^o--Taylor scheme, after the introduction of the spring term, for both the fine and the coarse trajectories can be written as
    \\[0.3cm]
    {\scshape Fine trajectory, $\PMeas$-measure:}
    \begin{align}
        \label{eq:fine_odd}
        \begin{split}
            \hatY^{f}(t_{2n + 1}) &= \hatY^{f}(t_{2n}) + Sh \big( \hatY^{c}(t_{2n}) - \hatY^{f}(t_{2n}) \big) + h a(\hatY^{f} (t_{2n})) + \Delta W_{2n}^{\PMeas}, \\
            &\quad + Da(\hatY^{f} (t_{2n})) \Delta Z_{2n}^{\PMeas} + \frac{h^{2}}{2} \cA a(\hatY^{f}(t_{2n})),
        \end{split} \\
        \label{eq:fine_even}
        \begin{split}
            \hatY^{f} (t_{2n + 2}) &= \hatY^{f} (t_{2n + 1}) + Sh \big( \hatY^{c} (t_{2n + 1}) - \hatY^{f} (t_{2n + 1}) \big) + h a(\hatY^{f} (t_{2n + 1})) + \Delta W_{2n + 1}^{\PMeas}, \\
            &\qquad + Da(\hatY^{f} (t_{2n + 1})) \Delta Z_{2n + 1}^{\PMeas} + \frac{h^{2}}{2} \cA a(\hatY^{f} (t_{2n + 1})),
        \end{split}
    \end{align}
    {\scshape Coarse trajectory, $\PMeas$-measure:}
    \begin{align}
        \label{eq:coarse_odd}
        \begin{split}
            \hatY^{c} (t_{2n + 1}) &= \hatY^{c} (t_{2n}) + Sh \big( \hatY^{f}(t_{2n}) - \hatY^{c}(t_{2n}) \big) + h a(\hatY^{c} (t_{2n})) + \Delta W_{2n}^{\PMeas} \\
            &\quad + Da(\hatY^{c}(t_{2n})) \Delta Z_{2n}^{\PMeas} + \frac{h^{2}}{2} \cA a(\hatY^{c}(t_{2n})),
        \end{split} \\
        \label{eq:coarse_even}
        \begin{split}
            \hatY^{c} (t_{2n + 2}) &= \hatY^{c} (t_{2n}) + 2Sh \big( \hatY^{f} (t_{2n}) - \hatY^{f} (t_{2n}) \big) + 2h a(\hatY^{c} (t_{2n})) + (\Delta W_{2n}^{\PMeas} + \Delta W_{2n + 1}^{\PMeas}) \\
            &\qquad + Da(\hatY^{c} (t_{2n})) (\Delta Z_{2n + 1}^{\PMeas} + \Delta Z_{2n}^{\PMeas} + h \Delta W_{2n}^{\PMeas}) + 2 h^{2} \cA a(\hatY^{c} (t_{2n})),
        \end{split}
    \end{align}
    for $n = 0, 1, \ldots, N/2 - 1$. The tuple of random variables $(\Delta W_{n}^{\PMeas}, \Delta Z_{n}^{\PMeas})$ under the $\PMeas$ measure can be generated as
    \begin{equation}
        \label{eq:random_increments}
        \Delta W_{n}^{\PMeas} \coloneqq \Delta V_{1, n}, \qquad \Delta Z_{n}^{\PMeas} \coloneqq \frac{1}{2} h \Big( \Delta V_{1, n} + \frac{1}{\sqrt{3}} \Delta V_{2, n} \Big),
    \end{equation}
    where $\Delta V_{1, n}, \Delta V_{2, n} \sim \mathcal{N}^{\PMeas} (0, h \Id)$ are i.i.d. samples. Note that both the fine and the coarse trajectories have the same initial value, i.e. $\hatY^{f} (0) = \hatY^{c} (0) = x_{0}$.

    Note that the order 1.5 strong It\^o--Taylor scheme with the spring term included, defined by the equations~\eqref{eq:fine_odd}-\eqref{eq:coarse_even}, does not solve the original SDE~\eqref{eq:SDE}. Using Girsanov's theorem, we can compute weak approximations of the invariant measure of the original SDE with numerical schemes that include the "spring" term. Girsanov's theorem states that there indeed exists a transformation from the $\PMeas$ measure, associated with the numerical process with the "spring" term, to the $\QMeas$ measure, associated with the original SDE~\eqref{eq:SDE} \cite[Theorem 9.2]{weinan2021applied}.
    
    Consider the fine trajectory of the numerical scheme with and without the "spring" term, sampled from the measures $\PMeas$ and $\QMeas^{f}$, respectively, where $\QMeas^{f}$ denotes the measure associated with the fine trajectory of the numerical scheme without the "spring" term. Under the transformation 
    \begin{equation}
        \label{eq:transform_fine}
        \begin{split}
            \Delta U_{1, n}^{\QMeas^{f}} &= \hatSf_{n} h + \Delta V_{1, n}^{\PMeas} \\
            \Delta U_{2, n}^{\QMeas^{f}} &= -\sqrt{3} \hatSf_{n} h + \Delta V_{2, n}^{\PMeas},
        \end{split}
    \end{equation}
    we have that \[ \hatY_{n}^{f} (\omega) = \barX_{n}^{f}(\omega), \quad \forall \omega \in \Omega \] for all $n = 0, 1, \ldots, N$, where
    \begin{equation}
        \label{eq:hatSf}
        \begin{split}
            \hatSf_{n} \coloneqq S \Big( \hatY_{n}^{c} - \hatY_{n}^{f} \Big).
        \end{split}
    \end{equation}
    The Radon--Nikodym derivative for the fine trajectory, according to Girsanov's theorem, is given by
    \begin{equation}
        \label{eq:RfT_derivative}
        R^{f}_{T} = \frac{\di \QMeas^{f}}{\di \PMeas} = \prod_{n = 0}^{N - 1} \exp \Bigg( - \inner{\hatSf_{n}, \Delta V_{1, n}} + \sqrt{3} \inner{\hatSf_{n}, \Delta V_{2, n}} - 2 h \norm{\hatSf_{n}}^{2} \Bigg).
    \end{equation}
    Similarly, consider the coarse trajectory of the numerical scheme with and without the "spring" term, sampled from the measures $\PMeas$ and $\QMeas^{c}$, respectively, where $\QMeas^{c}$ denotes the measure associated with the coarse trajectory of the numerical scheme without the "spring" term. Under the transformation
    \begin{equation}
        \label{eq:transform_coarse}
        \begin{split}
            \Delta U_{1, 2n}^{\QMeas^{c}} &= \hatSc_{2n} h + \Delta V_{1, 2n}^{\PMeas} \\
            \Delta U_{2, 2n}^{\QMeas^{c}} &= -2 \sqrt{3} \hatSc_{2n} h + \Delta V_{2, 2n}^{\PMeas},
        \end{split}
    \end{equation}
    we obtain \[ \hatY_{2n}^{c} (\omega) = \barX_{2n}^{c} (\omega), \quad \forall \omega \in \Omega \] for all $n = 0, \ldots, N/2$, where
    \begin{equation}
        \label{eq:hatSc}
        \begin{split}
            \hatSc_{2n} \coloneqq S \Big( \hatY_{2n}^{f} - \hatY_{2n}^{c} \Big).
        \end{split}
    \end{equation}
    The Radon--Nikodym derivative for the coarse trajectory is then given by
    \begin{equation}
        \label{eq:RcT_derivative}
        \begin{split}
            R^{c}_{T} = \frac{\di \QMeas^{c}}{\di \PMeas} &= \prod_{n = 0}^{N/2 - 1} \exp \Bigg( - \inner{\hatSc_{2n}, \Delta V_{1, 2n} + \Delta V_{1, 2n + 1}} + 2 \sqrt{3} \inner{\hatSc_{2n}, \Delta V_{2, 2n} + \Delta V_{2, 2n + 1}} \\
            &\qquad - 13 h \norm{\hatSc_{2n}}^{2} \Bigg).
        \end{split}
    \end{equation}
    
    Using the Radon--Nikodym derivatives given by the equations~\eqref{eq:RfT_derivative} and \eqref{eq:RcT_derivative}, the MLMC estimator under the change-of-measure can be written as
    \begin{equation}
        \label{eqn:mlmc_spring}
        \E^{\QMeas} \Big[ \varPhi(\barX_{T}^{L}) \Big] = \E^{\QMeas} \Big[ \varPhi(\barX_{T}^{0}) \Big] + \sum_{\ell = 1}^{L} \E^{\Prob} \Big[ \varPhi(\hatY_{T}^{f, \ell}) R_{T}^{f, \ell} - \varPhi(\hatY_{T}^{c, \ell}) R_{T}^{c, \ell} \Big].
    \end{equation}
    Theorems~\ref{thm:convergence_estimator} and \ref{thm:convergence_estimator_discont} show that $\varPhi$ is $L^{p}(\Omega)$-stable under the $\PMeas$ measure. We also make the assumption that $\varPhi$ is atleast $L_{1}(\Omega)$-stable under the $\QMeas$ measure, i.e. \[ \E^{\QMeas} \Big[ \abs{\varPhi(\barX_{T})} \Big] < \infty, \quad \forall T > 0. \]
    From the above assumption, and the results from Theorems~\ref{thm:convergence_estimator} and \ref{thm:convergence_estimator_discont}, we have that equation~\ref{eqn:mlmc_spring} is well-defined for all $T > 0$. 
    
    Having defined the higher-order change-of-measure method, we will discuss the corresponding main theoretical results in Section~\ref{sec:theoretical_results}.

    \section{Theoretical results}
    \label{sec:theoretical_results}
    In this section, we state the key theoretical results that establish the stability of the fine and the coarse trajectories of the MLMC estimator. We also prove $L^{p}(\Omega)$-norm convergence of the difference between the pairwise coupled numerical paths, the boundedness of the Radon--Nikodym derivative, and the resulting computational complexity of the higher-order change-of-measure MLMC method.

    \begin{thm}[Stability of numerical solution]
        \label{thm:stability}
            \allowdisplaybreaks
            Consider the SDE~\eqref{eq:SDE} and let Assumptions~\ref{asmp:lipschitz_a}, \ref{asmp:aD_lipschitz} and \ref{asmp:dissipativity_a} hold. Then for any spring constant $S > 0$, there exist constants $\kappa_{1}, \kappa_{2} > 0$, independent of $T$ and $p$, such that for any $T > 0$, $p \geq 1$, and for all $0 < h < \kappa_{1}$, it holds that
            \begin{align*}
                \sup_{0 \leq n \leq N} \E \Big[ \norm{\hatY_{t_{n}}^{f}}^{p}  \Big]^{1 / p} \leq \kappa_{2} p^{1/2}, \qquad \sup_{0 \leq n \leq N} \E \Big[ \norm{\hatY_{t_{n}}^{c}}^{p}  \Big]^{1 / p} \leq \kappa_{2} p^{1/2}.
            \end{align*}
    \end{thm}
    We defer the proof to Section~\ref{sec:theoretical_proofs}. Theorem~\ref{thm:stability} states that the trajectories of the higher-order change-of-measure MLMC method are $L^{p}(\Omega)$ stable for all sufficiently small $h > 0$, and that the moments are uniformly bounded in $T$. Note that the stability of the trajectories is not affected by the spring term. A similar theoretical result is obtained for the numerical scheme without the spring term
    \begin{thm}[Stability of numerical solution without the spring term]
        \label{thm:stability_wo_spring}
        Consider the SDE~\eqref{eq:SDE} and let Assumptions~\ref{asmp:lipschitz_a}, \ref{asmp:aD_lipschitz}, and \ref{asmp:dissipativity_a} hold. Then there exist constants $C_{1}, C_{2} > 0$, independent of $T$ and $p$, such that for any $T > 0$, $p \geq 1$, and for all $0 < h < C_{1}$, it holds that
        \[ \sup_{0 \leq n \leq N} \E \Big[ \norm{\barX_{t_{n}}^{f}}^{p} \Big]^{1/p} \leq C_{2} p^{1/2}, \qquad \sup_{0 \leq n \leq N/2} \E \Big[ \norm{\barX_{t_{2n}}^{c}}^{p} \Big]^{1/p} \leq C_{2} p^{1/2}. \]
    \end{thm}
    Using the results from Theorem~\ref{thm:stability}, we obtain the $L_{p}(\Omega)$-norm bounds for the pairwise coupled trajectories difference.
    \begin{thm}[$L^{p}(\Omega)$-norm pairwise coupled trajectories difference]
        \label{thm:convergence}
        \begingroup
            \allowdisplaybreaks
            Consider the SDE~\eqref{eq:SDE} and let Assumptions~\ref{asmp:lipschitz_a}, \ref{asmp:aD_lipschitz} and \ref{asmp:dissipativity_a} hold. Then for any $S > \lambda / 2$, there exist constants $\kappa_{3}, \kappa_{4}, \kappa_{5} > 0$, independent of $T$ and $p$, such that for any $T > 0$, $p \geq 1$, and for all $0 < h < \min(\kappa_{3}, \kappa_{4} / \sqrt[3]{p})$, it holds that
            \[ \sup_{0 \leq n \leq N} \Bigg( \E \Bigg[ \norm{\hatY_{t_{n}}^{f} - \hatY_{t_{n}}^{c}}^{p} \Bigg] \Bigg)^{1/p} \leq \kappa_{5} \min \Big( p^{1/2} h^{1/2}, \; p h, \; p^{3/2} h^{3/2} \Big). \]
        \endgroup
    \end{thm}
    We defer the proof to Section~\ref{sec:theoretical_proofs}. Theorem~\ref{thm:convergence} states that the $L^{p}(\Omega)$-norm of the pairwise coupled difference of the MLMC trajectories converges uniformly in time $T$ for all sufficiently small $h > 0$. By choosing $S > \lambda / 2$, we can achieve uniform-in-time bounds obtained in Theorem~\ref{thm:convergence}. \\
    \begin{rem}
        Note that for all $p \geq 1$ and $h \in (0,1)$, the term $ph \geq \min (p^{1/2} h^{1/2}, p^{3/2} h^{3/2})$ and hence can be removed from Theorem~\ref{thm:convergence}. We have included the intermediate term as we use it recursively to obtain the higher-order convergence rate of $\cO (p^{3/2} h^{3/2})$.
    \end{rem}
    Making use of Theorems~\ref{thm:stability} and \ref{thm:convergence}, we can bound the $L^{p}(\Omega)$-norm of the Radon--Nikodym derivatives.
    \begin{thm}[$L^{p}(\Omega)$-norm Radon--Nikodym derivative]
     \label{thm:radon_nikodym}
        Consider the SDE~\eqref{eq:SDE} and let Assumptions~\ref{asmp:lipschitz_a}, \ref{asmp:aD_lipschitz} and \ref{asmp:dissipativity_a} hold. Then for all $S > \lambda / 2$, $T \geq 1$, $p \geq 1$, and for all $0 < h \leq h_{\max}(T)$, it holds that
        \[ \E \Big[ \abs{R^{f}_{T}}^{p} \Big] \leq \eta_{1}, \qquad \E \Big[ \abs{R^{c}_{T}}^{p} \Big] \leq \eta_{2}, \]
        where $h_{\max}(T)$ denotes the maximum permissible time-step size as a function of the terminal time $T$ given by \[ h_{\max}(T) \coloneqq \min \big( 1, \; C_{0} / \sqrt{T \log T} \big), \] 
        $C_{0}, \eta_{1}, \eta_{2} > 0$ are constants that are independent of $T$ and $p$.
    \end{thm}
    We defer the proof to Section~\ref{sec:theoretical_proofs}. Theorem~\ref{thm:radon_nikodym} shows that the Radon--Nikodym derivative is uniformly bounded in time $T$ for all sufficiently small $h > 0$. Using Theorems~\ref{thm:stability}, \ref{thm:convergence} and \ref{thm:radon_nikodym}, we can obtain the bounds for the $L^{p}(\Omega)$-norm of the higher-order change-of-measure MLMC estimator. Before proceeding with the analysis, we consider two cases for the regularity of the payoff function $\varPhi \colon \bR^{d} \to \bR$. The first case we consider is that of a uniformly Lipschitz continuous payoff function
    \begin{asm}[uniform Lipschitz continuity of $\varPhi(x)$]
        \label{asmp:lipschitz_payoff}
        There exists a constant $K > 0$, such that \[ \abs{\varPhi(x) - \varPhi(y)} \leq K \norm{x - y}, \qquad \forall x, y \in \bR^{d}, \]
        where $K$ is the uniform Lipschitz constant.
    \end{asm}
    Using Assumption~\ref{asmp:lipschitz_payoff}, we can bound the moments of the MLMC estimator.
    \begin{thm}[Multilevel estimator moments - uniformly Lipschitz continuous payoff]
        \label{thm:convergence_estimator}
        Consider the SDE~\eqref{eq:SDE} and let Assumptions~\ref{asmp:lipschitz_a}, \ref{asmp:aD_lipschitz}, and \ref{asmp:dissipativity_a} hold. Let $\varPhi \colon \bR^{d} \to \bR$ denote a function that satisfies Assumption~\ref{asmp:lipschitz_payoff}. Then for any $S > \lambda / 2$, for all $T > 1$, $p \geq 1$, and for all $h \in \big( 0, h_{\max}(T) \big]$, it holds that
        \[ \E \Bigg[ \abs{\varPhi(\hatY_{T}^{f}) R_{T}^{f} - \varPhi(\hatY_{T}^{c}) R_{T}^{c}}^{p} \Bigg] ^ {1/p} \leq \kappa_{7} \sqrt{T} p^{5/2} h^{3/2}, \]
        where $\kappa_{7}> 0$ is a constant that is independent of $T$ and $p$.
    \end{thm}

    We defer the proof to Section~\ref{sec:theoretical_proofs}. The second case we consider is a class of discontinuous payoff functions of the following form $\varPhi(x) = \ind_{x \in G}$,
    for a Borel set $G \in \cB \big( \bR^{d} \big)$ and
    \begin{equation}
        \label{eq:indicator function}
        \begin{split}
            \ind_{x \in G} 
            \coloneqq \begin{cases}
                1, & \text{if } x \in G \\
                0, & \text{otherwise.} 
            \end{cases}
        \end{split}
    \end{equation}
    Let $\overline{G}$ denote the closure of the set $G$, and $\partial G$ denote the boundary of the set $G$. Note that $\partial G = \overline{G} \setminus \text{int} \big( G \big)$, where $\text{int} ( A )$ denotes the interior of the set $A \subset \bR^{d}$. We define the distance measure from a point $x \in \bR^{d}$ to the boundary $\partial G$ as
    \begin{equation}
        d_{\partial G}(x) \coloneqq \inf_{y \in \partial G} \norm{y - x}, \qquad \forall x \in \bR^{d}.
    \end{equation}
    Using the distance measure, we define the set
    \begin{equation}
        \label{eq:G_delta}
        \begin{split}
            G_{\delta} \coloneqq \Big\{ x \in \bR^{d} \; \big\vert \; d_{\partial G}(x) \leq \delta \Big\}
        \end{split}
    \end{equation}
    that denotes the set of points $x \in \bR^{d}$ that are distance $\delta$ away from the boundary $\partial G$.
    We make the following assumption on the probability density function of the numerical process
    \begin{asm}[Existence and boundedness of density]
        \label{asmp:bounded_density}
        For any $T > 1$, $\xi > 0$, and for all sufficiently small $h \in (0, h_{\max}(T)]$, there exists a probability density function $\rho_{\hatY_{T}^{f}}(x)$ associated to the numerical process $\hatY_{T}^{f}$, such that
        \[ \sup_{T > 1} \sup_{h \in (0, h_{\max}(T)]} \sup_{x \in \bR^{d}} \rho_{\hatY_{T}^{f}}(x) \leq C_{\rho} < \infty, \] where $C_{\rho} > 0$ is a constant that is independent of $h$ and $T$.
    \end{asm}
    Having made the necessary assumptions, we now state the following lemma that ensures that the distance measure is bounded
    \begin{lemma}[Probability of distance measure]
        \label{lem:bounded_distance}
        For all $\delta > 0$, for any $T > 1$, and for all sufficiently small $h \in \big( 0, h_{\max}(T) \big]$, assuming that there exists a constant $C_{v} > 0$ such that \[ \vol(G_{\delta}) \coloneqq \int_{G_{\delta}} \di x \leq C_{v} \delta \] holds, then there exists another constant $C_{b} > 0$ that is independent of $\delta$, $T$, and $h$ such that \[ \Prob \Big( d_{\partial G}(\hatY_{T}^{f}) \leq \delta \Big) \leq C_{b} \delta. \]
    \end{lemma}
    \begin{proof}
        Note that the event
        \[ \Big\{ \omega \in \Omega \; \big\vert \; d_{\partial G} ( \hatY_{T}^{f}(\omega) ) \leq \delta \Big\} = \Big\{ \omega \in \Omega \; \big\vert \; \hatY_{T}^{f}(\omega) \in G_{\delta} \Big\} \]
        Using the assumption that the volume of the set $G_{\delta}$ is bounded, we obtain
        \begin{equation*}
            \begin{split}
                \Prob \Big( \hatY_{T}^{f} \in G_{\delta} \Big) = \E \Big[ \ind_{\hatY_{T}^{f} \in G_{\delta}} \Big] &= \int_{\bR^{d}} \ind_{x \in G_{\delta}} \rho_{\hatY_{T}^{f}}(x) \di x \\
                &= \int_{G_{\delta}} \rho_{\hatY_{T}^{f}}(x) \di x \\
                &\leq \int_{G_{\delta}} \sup_{T > 1} \sup_{h \in (0, h_{\max}(T)]} \sup_{x \in \bR^{d}} \rho_{\hatY_{T}^{f}}(x) \di x \\
                &\leq C_{\rho} \Big( \underbrace{\int_{G_{\delta}} \di x}_{\eqqcolon \vol(G_{\delta})}\Big) \leq \underbrace{C_{\rho} C_{v}}_{\eqqcolon C_{b}} \delta.
            \end{split}
        \end{equation*}
    \end{proof}
    We now extend Theorem~\ref{thm:convergence_estimator} to the case of a discontinuous payoff function defined by equation~\eqref{eq:indicator function}.
    \begin{thm}[Multilevel estimator moments - discontinuous payoff]
        \label{thm:convergence_estimator_discont}
        Consider the SDE~\eqref{eq:SDE} and let Assumptions~\ref{asmp:lipschitz_a}, \ref{asmp:aD_lipschitz}, and \ref{asmp:dissipativity_a} hold. Let $\varPhi \colon \bR^{d} \to \bR$ denote the payoff function given by equation~\eqref{eq:indicator function}. Then for all $S > \lambda / 2$, $T > 1$, $p, q \geq 1$, $h \in \big( 0, h_{\max}(T) \big]$, and $\xi > 0$, it holds that
        \[ \E \Bigg[ \abs{\varPhi(\hatY_{T}^{f}) R_{T}^{f} - \varPhi(\hatY_{T}^{c}) R_{T}^{c}}^{p} \Bigg] \leq \kappa_{8} T^{p/2} p^{2p} (1 + q) q^{\frac{1}{2} \frac{q}{q + 1}} h^{\frac{3}{2} - \xi}, \]
        where $\kappa_{8}> 0$ is a constant that is independent of $T$, $p$, and $q$.
    \end{thm}
    We defer the proof to Section~\ref{sec:theoretical_proofs}. From Theorems~\ref{thm:convergence_estimator} and \ref{thm:convergence_estimator_discont} we have that the variance of the MLMC estimator is $\cO(T h^{3})$ and $\cO (T h^{3/2 - \xi})$ for uniformly Lipschitz continuous and discontinuous payoff functions, respectively. Having established all the necessary theoretical results, we now determine the computational cost of implementing the higher-order change-of-measure MLMC method.
    \begin{thm}[MLMC computational complexity]
        \label{thm:mlmc_cost}
        Let $\hatphi$ denote the MLMC estimator corresponding to the higher-order change-of-measure scheme
        \begin{equation*}
            \begin{split}
                \hatphi &\coloneqq \frac{1}{N_{0}} \sum_{k = 1}^{N_{0}} \big( \varPhi(\barX_{T}^{0}) \big)_{k} + \sum_{\ell = 1}^{L} \frac{1}{N_{\ell}} \sum_{k = 1}^{N_{\ell}} \big( \varPhi(\hatY_{T}^{f, \ell}) R_{T}^{f, \ell} - \varPhi(\hatY_{T}^{c, \ell}) R_{T}^{c, \ell} \big)_{k},
            \end{split}
        \end{equation*}
        where $N_{\ell}$ denotes the number of Monte Carlo samples on each level of the MLMC estimator.
        Let Assumptions~\ref{asmp:lipschitz_a}, \ref{asmp:aD_lipschitz}, and \ref{asmp:dissipativity_a} hold. Then for any sufficiently small error tolerance $\epsilon > 0$, and for a suitable choice of $L, T, N_{\ell}$, and $h_{0}$, the mean square error (MSE) of the estimator $\hatphi$ is bounded as
        \begin{equation*}
            \E \Big[ \big( \hatphi - \E[\varPhi(X_{\infty})] \big)^{2} \Big] \leq \epsilon^{2}
        \end{equation*}
        and the expected total computational cost of the MLMC estimator is
        \begin{equation*}
            \Cost = \cO \big( \epsilon^{-2} \abs{\log \epsilon}^{3/2} (\log \abs{\log \epsilon})^{1/2} \big),
        \end{equation*}
        for all QoI that satisfy Assumption~\ref{asmp:lipschitz_payoff}. The expected total computational cost is
        \begin{equation*}
            \Cost = \cO \big( \epsilon^{-2} \abs{\log \epsilon}^{5/3 + \xi} \big),
        \end{equation*}
        for all $\xi > 0$, QoI of the form given by equation~\eqref{eq:indicator function}, and satisfying Assumption~\ref{asmp:bounded_density}.
    \end{thm}
    We defer the proof to Section~\ref{sec:theoretical_proofs}. The computational cost argument similarly applies to the Milstein scheme. Since the Radon--Nikodym term for the Milstein scheme \cite[equation~(13)]{fang2019multilevel} shares a similar structure with the higher-order scheme given by equations~\eqref{eq:RfT_derivative} and \eqref{eq:RcT_derivative}, and considering that we have not utilized the $\cO ( h^{3/2} )$ convergence rate in deriving Theorem~\ref{thm:radon_nikodym}, we believe the results can be extended to the Milstein scheme, yielding the condition $h_{0} = \cO (1 / \sqrt{T \log T})$. Assuming that this is the case, for a uniformly Lipschitz continuous payoff function, recall that
    \begin{equation*}
        \begin{split}
            \abs{\E \big[ \varPhi(\hatY_{T}^{f, \ell}) R_{T}^{f, \ell} - \varPhi(\hatY_{T}^{c, \ell}) R_{T}^{c, \ell} \big]} &\leq C_{1} \sqrt{T} ( 2^{-\ell} h_{0} ) \\
            \Var \Big( \varPhi(\hatY_{T}^{f, \ell}) R_{T}^{f, \ell} - \varPhi(\hatY_{T}^{c, \ell}) R_{T}^{c, \ell} \Big) &\leq c_{1} T (2^{-\ell} h_{0})^{2}
        \end{split}, \qquad \forall \ell \geq 1,
    \end{equation*}
    where $c_{1}, C_{1} > 0$ are constants.
    Then for the choice of $h_{0} = \cO (1 / \sqrt{T \log T})$ and
    \[ L = \ceil{ \log_{2}  \Big( \epsilon^{-1} \abs{\log T}^{-1/2} \Big) + \log_{2} (\sqrt{6} \widetilde{C}_{1} ) }, \]
    where $\widetilde{C}_{1} > 0$ is a constant, the computational cost is
    \[ \Cost = \cO \Big( \epsilon^{-2} \abs{\log \epsilon}^{3/2} (\log \abs{\log \epsilon})^{1/2} \Big). \]
    For a discontinuous payoff function of the form given by equation~\eqref{eq:indicator function} and for all $\xi > 0$, we can show that the strong error rate and the variance decay rate for the Milstein scheme are bounded as
    \begin{equation*}
        \begin{split}
            \abs{\E \big[ \varPhi(\hatY_{T}^{f, \ell}) R_{T}^{f, \ell} - \varPhi(\hatY_{T}^{c, \ell}) R_{T}^{c, \ell} \big]} &\leq C_{1} \sqrt{T} ( 2^{-\ell} h_{0} )^{1 - \xi} \\
            \Var \Big( \varPhi(\hatY_{T}^{f, \ell}) R_{T}^{f, \ell} - \varPhi(\hatY_{T}^{c, \ell}) R_{T}^{c, \ell} \Big) &\leq c_{1} T (2^{-\ell} h_{0})^{1 - \xi}
        \end{split}, \qquad \forall \ell \geq 1,
    \end{equation*}
    where $c_{1}, C_{1} > 0$ are constants. Then for the choice of $h_{0} = \cO \Big( 1 / \sqrt{T \log T} \Big)$ and
    \[ L = \ceil{\frac{1}{1 - \xi} \log_{2} \Big( \epsilon^{-1} \sqrt{T} h_{0}^{1 - \xi} \Big) + \frac{1}{1 - \xi} \log_{2} (\sqrt{6} \widetilde{C}_{1}) }, \]
    where $\widetilde{C}_{1} > 0$ is a constant, all $\xi > 0$, and $\hat\xi = \frac{\xi}{1 - \xi}$, we obtain
    \[ \Cost = \cO \Big( \epsilon^{-(2 + \hat\xi)} \abs{\log \epsilon}^{2 + \hat\xi/2} \Big). \]

    \section{Numerical results}
    \label{sec:numerical_results}

    In this section, we run several numerical simulations to verify the results obtained in Section~\ref{sec:notation_theorem}. Algorithm~\ref{alg:higher_order} describes the implementation of the higher-order change-of-measure scheme to compute the weak approximations of the invariant measure.

    \begin{algorithm}[ht]
        \begin{algorithmic}[1]
            \Require Initial condition $x_{0}$, Terminal time $T$, \# of multilevel estimator levels $L$, \# of Monte Carlo samples $M$, Time-step size on level--0 $h_{0}$.
            \For{$i = 1,\ldots,M$}
                \State Initialize the fine and the coarse trajectories at time $t = 0$, i.e. $\displaystyle \hatY_{0}^{f} = \hatY_{0}^{c} = x_{0}$.
                \State Set the initial values for the Radon--Nikodym derivatives, i.e. $R_{0}^{f} = R_{0}^{c} = 1$.
                \For{$n = 0, \ldots, N/2 - 1$}
                    \LComment{Updating the state of the process from $t_{2n}$ to $t_{2n + 1}$}
                    \State Using equation~\eqref{eq:random_increments}, generate the tuple of correlated random increments $(\Delta W_{2n}, \Delta Z_{2n})$.
                    \State Update the fine and the coarse trajectories using equations~\eqref{eq:fine_odd} and \eqref{eq:coarse_odd}.
                    \State Compute the Radon--Nikodym derivative $R^{f}_{2n}$ for the fine trajectory using equation~\eqref{eq:RfT_derivative}, i.e.
                    \begin{equation*}
                        \begin{multlined}
                            R_{2n + 1}^{f} = R_{2n}^{f} \cdot \exp \Bigg( - \inner{\hatSf_{2n}, \Delta V_{1, 2n}} + \sqrt{3} \inner{\hatSf_{2n}, \Delta V_{2, 2n}} - 2 h \norm{\hatSf_{2n}}^{2} \Bigg)
                        \end{multlined}
                    \end{equation*}
                    \LComment{Updating the state of the process from $t_{2n + 1}$ to $t_{2n + 2}$}
                    \State Using equation~\eqref{eq:random_increments}, generate the tuple of correlated random increments $(\Delta W_{2n + 1}, \Delta Z_{2n + 1})$.
                    \State Update the fine and the coarse trajectories using equations~\eqref{eq:fine_even} and \eqref{eq:coarse_even}.
                    \State Compute the Radon--Nikodym derivatives, $R_{2n + 1}^{f}$ and $R_{2n + 1}^{c}$, for the fine and the coarse trajectories, respectively, using equations~\eqref{eq:RfT_derivative} and \eqref{eq:RcT_derivative}, i.e.
                    \begin{equation*}
                        \begin{split}
                            R_{2n + 2}^{f} &= R_{2n + 1}^{f} \cdot \exp \bigg( - \inner{\hatSf_{2n + 1}, \Delta V_{1, 2n + 1}} \\
                            &\qquad \qquad \qquad \qquad + \sqrt{3} \inner{\hatSf_{2n + 1}, \Delta V_{2, 2n + 1}} - 2 h \norm{\hatSf_{2n + 1}}^{2} \bigg) \\
                            ~
                            R_{2n + 2}^{c} &= R_{2n}^{c} \cdot \exp \bigg( - \inner{\hatSc_{2n}, \Delta V_{1, 2n} + \Delta V_{1, 2n + 1}} + 2 \sqrt{3} \inner{\hatSc_{2n}, \Delta V_{2, 2n} + \Delta V_{2, 2n + 1}} \\
                            &\qquad \qquad \qquad \qquad - 13 h \norm{\hatSc_{2n}}^{2} \bigg)
                        \end{split}
                    \end{equation*}
                \EndFor
            \EndFor
            \Return For a given payoff function, $\varPhi \colon \bR^{d} \to \bR$, the expected value of the payoff w.r.t. the invariant measure, i.e.
            \begin{equation*}
                \E \Bigg[ \abs{ \varPhi \big( \hatY_{T}^{f} \big) R_{T}^{f} - \varPhi \big( \hatY_{T}^{c} \big) R_{T}^{c} }^{p} \Bigg]^{1/p}
            \end{equation*}
        \end{algorithmic}
        \caption{Higher-order change-of-measure algorithm}
        \label{alg:higher_order}
    \end{algorithm}

    \subsection{1D Triple-well potential -- Discontinuous payoff}
    \label{subsec:1d_triplewell}
    We start with the numerical tests by considering the 1D SDE
    \begin{equation*}
        \begin{split}
            \di X &= \frac{X^3 (2 - X^2) (X^8 + 2 X^6 + 4 X^2 - 4)}{2 (X^6 + 1)^2} \di t + \di W \\
            X(0) &= 1,
        \end{split}
    \end{equation*}
    corresponding to the potential function $f(X) \coloneqq ( X^4- 2 X^2 )^2 / (4 ( X^6 + 1 ))$. The invariant probability density of the SDE is \[ \rho(X) \propto \exp(-2 f(x)) = \exp \Big( - \frac{( X^4- 2 X^2 )^2}{2 ( X^6 + 1 )} \Big), \] and is as depicted in Figure~\ref{fig:triple_well_ergodic rate}. Note that the drift coefficient satisfies Assumptions~\ref{asmp:lipschitz_a}, \ref{asmp:aD_lipschitz}, and \ref{asmp:dissipativity_a}. We are interested in computing the weak approximation of the payoff function $\varPhi(x) \coloneqq \ind_{x \in [0, 2]}$ w.r.t. the invariant measure of the SDE for the Double-well potential. The reference solution to the weak approximation of the payoff function w.r.t the invariant measure was computed by solving the elliptic Fokker--Planck PDE \cite[equation~(8.24)]{weinan2021applied}, using the continuous, piecewise quadratic finite elements for the spatial discretization. For this purpose, we use the \texttt{Fenics} library in \texttt{Python} programming language. The pseudo-reference solution is $\E [ \varPhi(X_{\infty}) ] = 0.42863$, rounded to 5 significant digits. We run the simulation for $M = 10^{8}$ Monte Carlo samples, until the terminal time $T = 10$, to compute the convergence rate for the higher-order numerical scheme. We use an initial time-step size $h_{0} = 2^{-5}$.

    \begin{figure}[ht]
        \centering
        \includegraphics[width=0.49\linewidth]{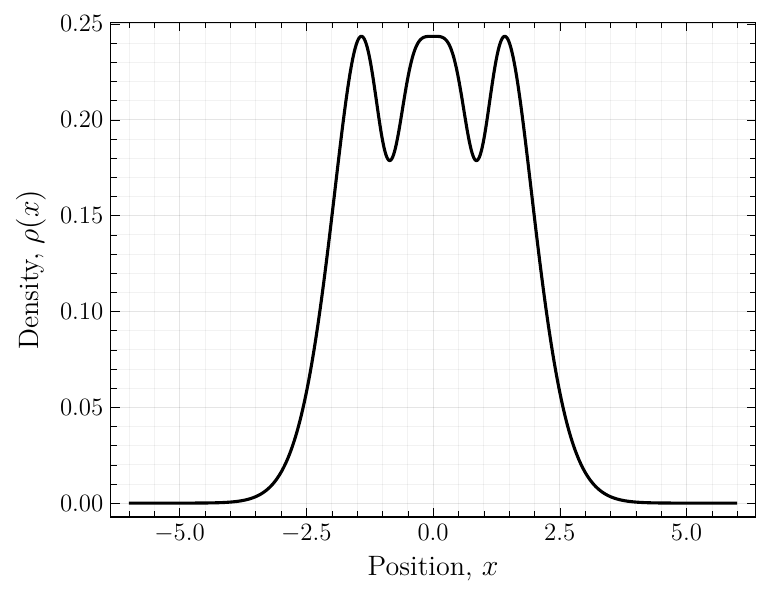}
        \includegraphics[width=0.49\linewidth]{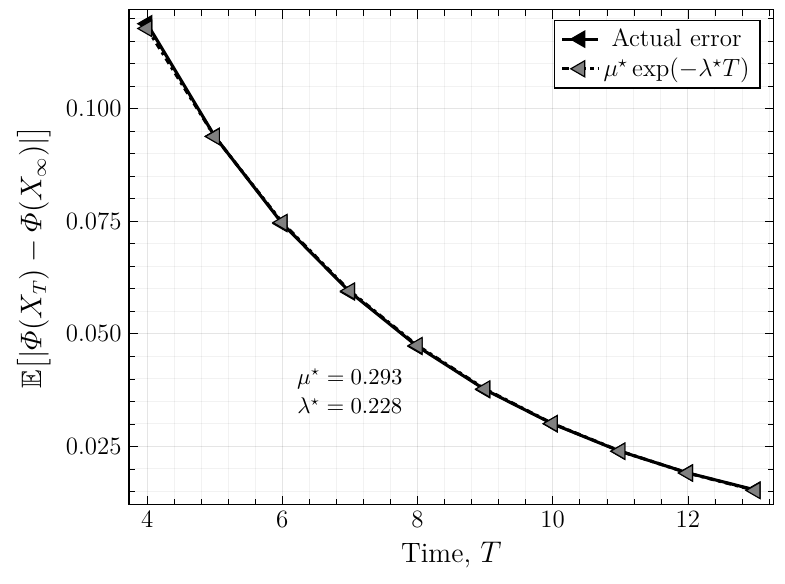}
        \includegraphics[width=0.49\linewidth]{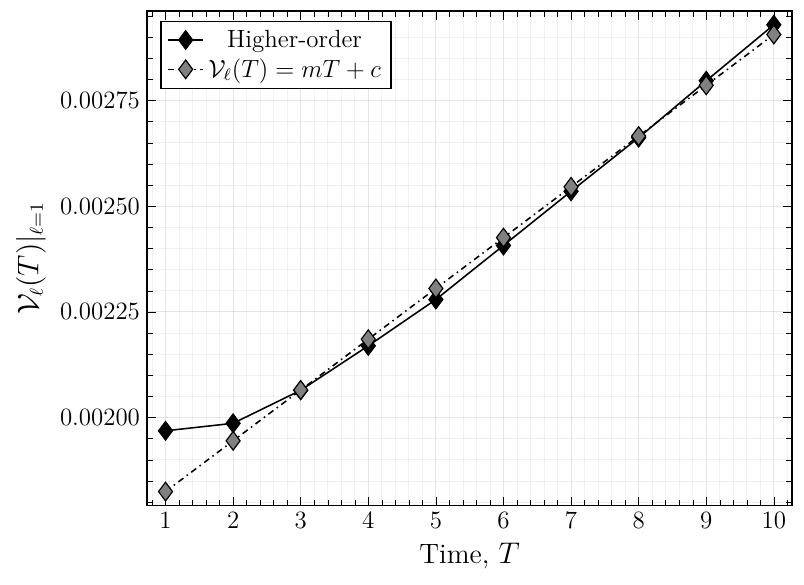}
        \caption{Plots associated to the SDE considered in Section~\ref{subsec:1d_triplewell}: Probability density function of the triple-well potential (top left), Curve fit to compute the parameters $\mu^{\star}$ and $\lambda^{\star}$ (top right), Variance growth rate of the MLMC estimator w.r.t terminal time on level $\ell = 1$ (bottom).}
        \label{fig:triple_well_ergodic rate}
    \end{figure}

    Figure~\ref{fig:triple_well_ergodic rate} shows that the variance of the MLMC estimator increases linearly with time $T$ for a fixed time-step size $h$, as shown in Theorem~\ref{thm:convergence_estimator_discont}. Figure~\ref{fig:triple_well_1D_discontinuous} shows that the convergence rates obtained from the numerical simulation are consistent with theory. We observe that both the strong error and the variance rate for the higher-order change-of-measure scheme are $\cO(h^{3/2})$ as proved in Theorem~\ref{thm:convergence_estimator_discont}. Recall that the set $\Theta_{T} \coloneqq  \big\{ \omega \in \Omega \; \big\vert \; \norm{\hatY_{T}^{f} - \hatY_{T}^{c}} \geq \nu_{1} \abs{\log h} \big\}$ is as defined in Theorem~\ref{thm:radon_nikodym}. From Figure~\ref{fig:triple_well_1D_discontinuous}, we also observe that the probability of the fine and the coarse trajectories diverging $\Prob (\Theta_{T}) = 0$ for $\nu_{1} = 1$. We observe that the kurtosis increases with the level of the multilevel estimator. This can be attributed to our computing the weak approximation of a discontinuous payoff function.
    
    \begin{figure}[ht]
        \centering
        \includegraphics[width=0.49\linewidth]{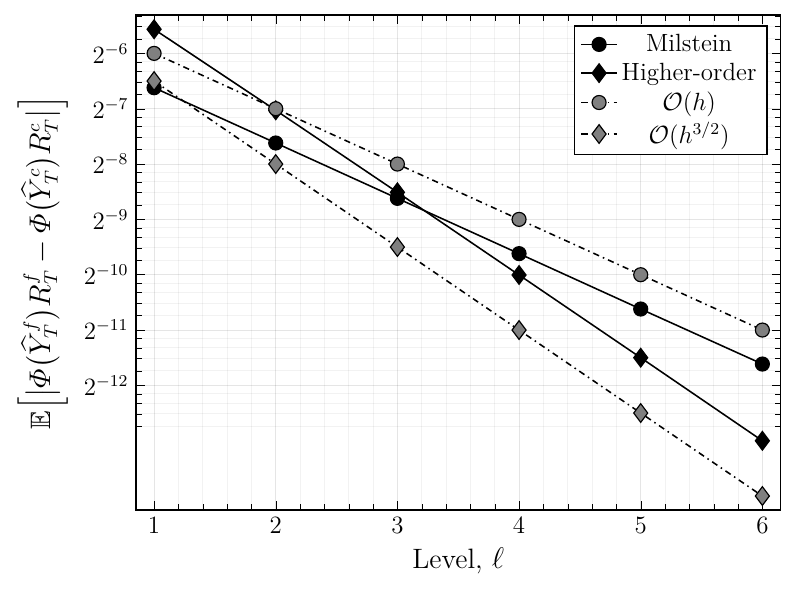}
        \includegraphics[width=0.49\linewidth]{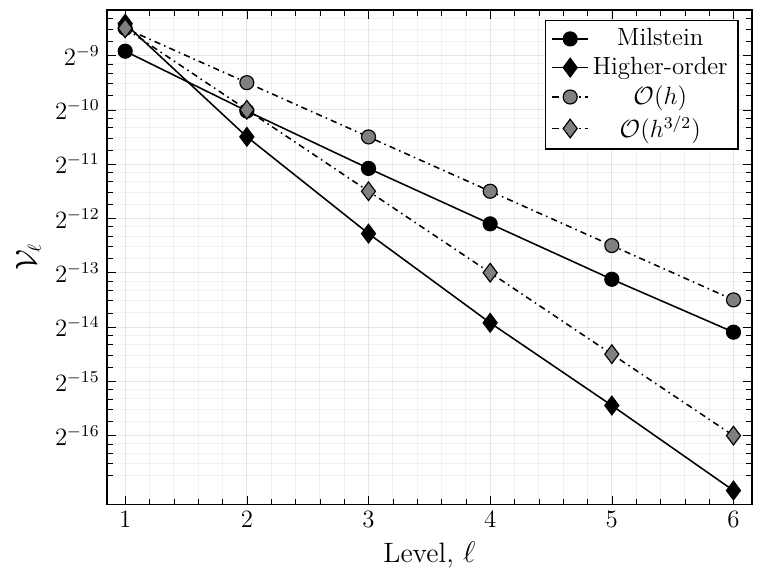}
        \includegraphics[width=0.49\linewidth]{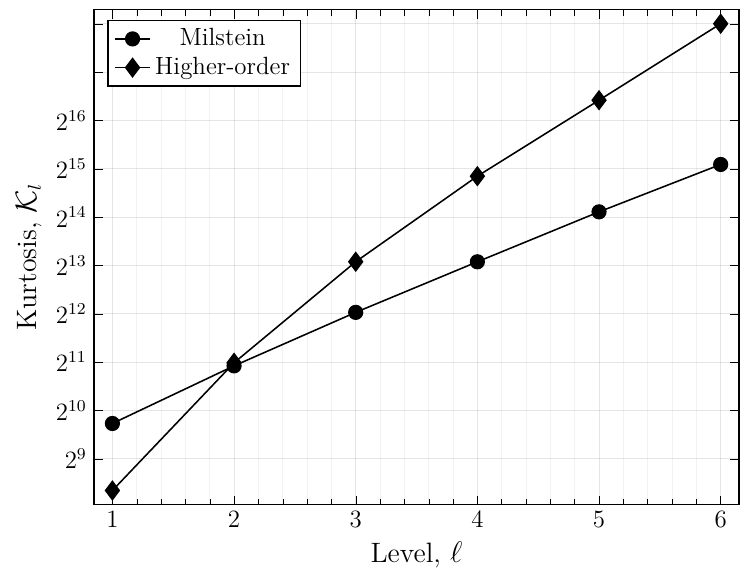}
        \includegraphics[width=0.49\linewidth]{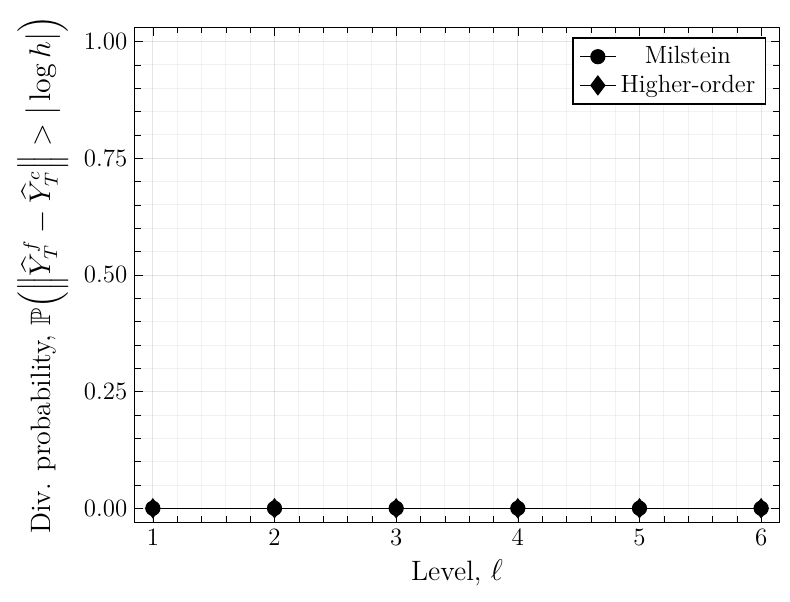}
        \caption{Numerical results for the problem considered in Section~\ref{subsec:1d_triplewell} at the terminal time, $T = 10$: strong error (top left), variance rate (top right), kurtosis (bottom left), divergence probability (bottom right)}
        \label{fig:triple_well_1D_discontinuous}
    \end{figure}

    The numerical experiment matches the theoretical results as shown in Figure~\ref{fig:triple_well_1D_discontinuous_mlmc}. For the choice of $h_{0} = \cO ( T^{-1/(3/2 - \xi)} )$, for a discontinuous QoI, and using the higher-order change-of-measure scheme, we can achieve a computational cost of $\cO \big( \abs{\log \epsilon}^{5/3 + \xi_{1}} \big)$, for all $\xi_{1} > 0$, while still achieving an MSE of $\cO ( \epsilon^{2} )$. As for the Milstein scheme developed in \cite{fang2019multilevel}, we observe that to achieve an MSE of $\cO ( \epsilon^{2} )$, the computational cost is $\cO ( \epsilon^{-\xi_{2}} \abs{\log \epsilon}^{2 + \xi_{2} / 2} )$ for the choice of $h_{0} = \cO ( 1 / \sqrt{T \log T} )$ and all $\xi_{2} > 0$. Although we use a larger initial time-step size for the Milstein scheme, the associated computational cost is lower for the higher-order scheme compared to the Milstein scheme.

    \begin{figure}[ht]
        \centering
        \includegraphics[width=0.49\linewidth]{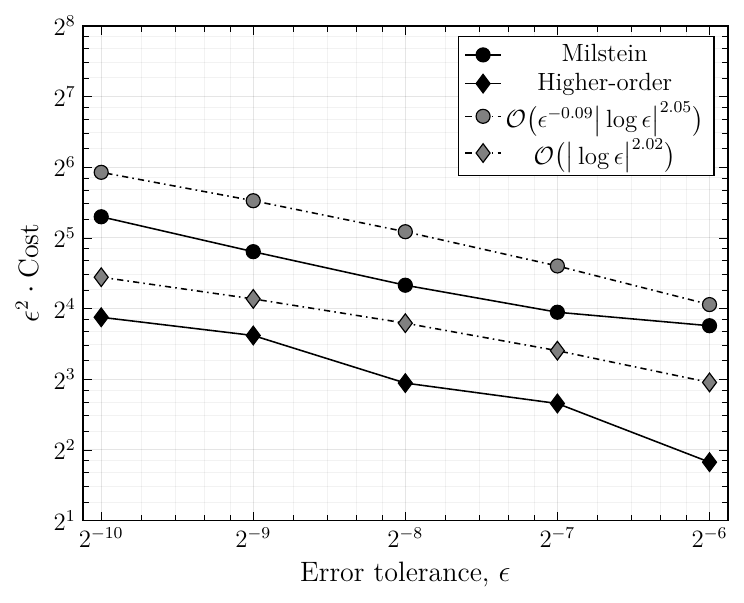}
        \includegraphics[width=0.49\linewidth]{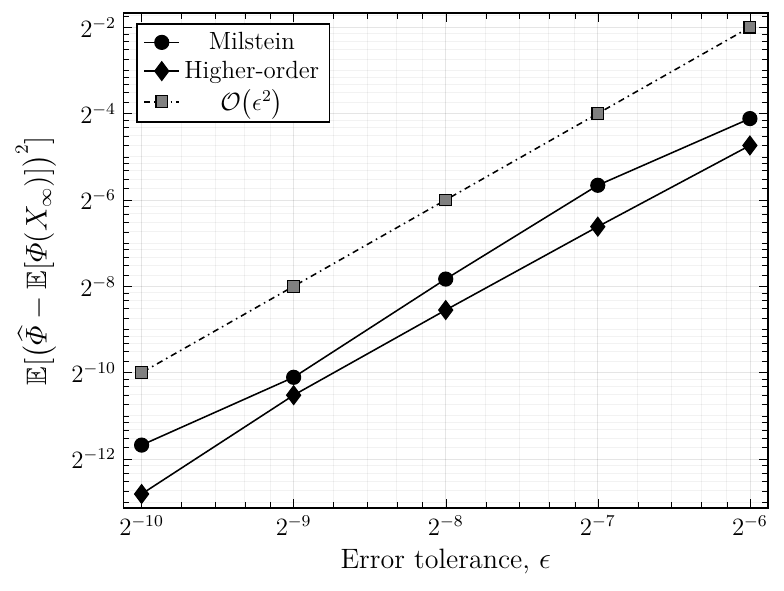}
        \includegraphics[width=0.49\linewidth]{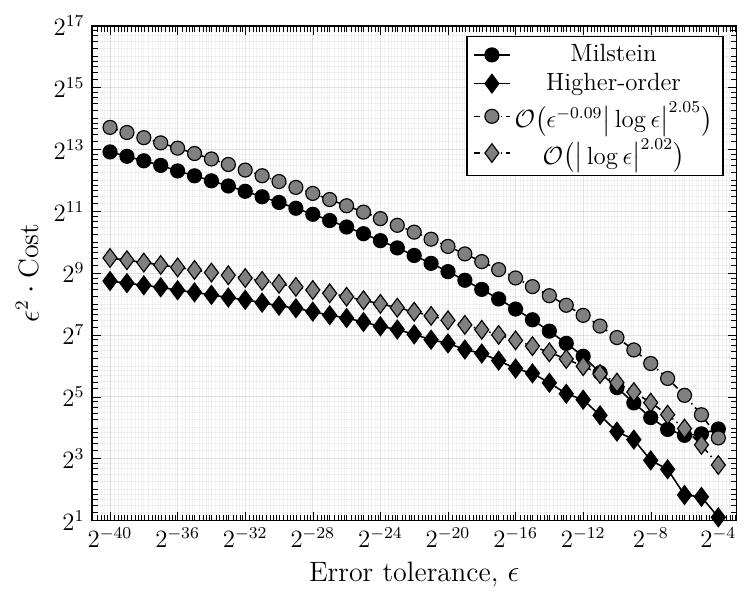}
        \caption{Change-of-measure MLMC scheme results for the problem considered in Section~\ref{subsec:1d_triplewell}: computational cost (top left), mean squared error (top right), asymptotic computational cost (bottom)}
        \label{fig:triple_well_1D_discontinuous_mlmc}
    \end{figure}
    
    \subsection{2D potential well -- Discontinuous payoff}
    \label{subsec:doublewell_2d_mixing}
    We will now consider the SDE
    \begin{equation*}
        \begin{split}
            \di X_{1} &= \Bigg( X_{1} \bigg( \frac{4}{(1 + X_{1}^{2})^{2}} - 2 \bigg) + \frac{1}{2} \sech^{2}(X_{1}) \tanh(X_{2}) \Bigg) \di t + \di W^{1}\\
            \di X_{2} &= \Bigg( X_{2} \bigg( \frac{4}{(1 + X_{2}^{2})^{2}} - 2 \bigg) + \frac{1}{2} \sech^{2}(X_{2}) \tanh(X_{1}) \Bigg) \di t + \di W^{2} \\
            X(0) &= (0, 0)^{\top}
        \end{split}
    \end{equation*}
    corresponding to the potential function
    \[ f(X_{1}, X_{2}) \coloneqq \frac{X_{1}^{4} + 1}{X_{1}^{2} + 1} + \frac{X_{2}^{4} + 1}{X_{2}^{2} + 1} - \frac{1}{2} \tanh(X_{1}) \tanh(X_{2}) - 2 \]
    that models the dynamics of a particle under the influence of a potential well. The drift coefficient satisfies Assumptions~\ref{asmp:lipschitz_a}, \ref{asmp:aD_lipschitz}, and \ref{asmp:dissipativity_a}. The drift coefficient exhibits non-linearity close to the origin. We are interested in computing the weak approximation of the payoff function
    \[ \varPhi(X_{1}, X_{2}) \coloneqq \begin{cases}
        1, & \text{if } 0 \leq X_{1} + X_{2} \leq 1.4 \text{ and } -0.75 \leq X_{2} - X_{1} \leq 0.75 \\
        0, & \text{otherwise}
    \end{cases}, \]
    w.r.t. the invariant measure of the SDE. The pseudo-reference solution is $\E [ \varPhi(X_{\infty}) ] = 0.1674$, rounded to 4 significant digits. We run the simulation until the terminal time $T = 10$, using an initial time-step size $h_{0} = 2^{-4}$, and $M = 10^{7}$ Monte Carlo samples.

    \begin{figure}[ht]
        \centering
        \includegraphics[width=0.5\linewidth]{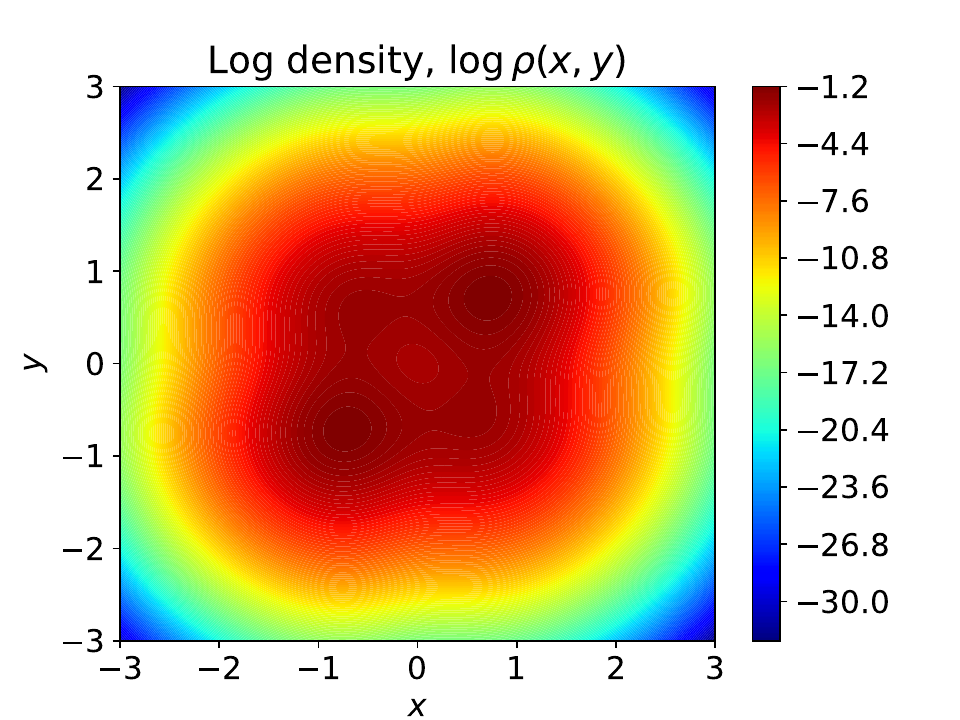}
        \includegraphics[width=0.45\linewidth]{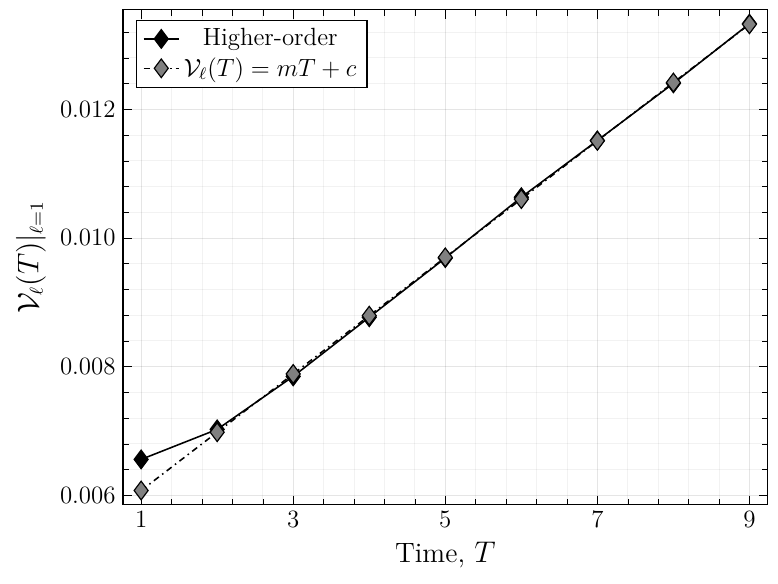}
        \caption{The logarithm of the probability density function of the SDE considered in Section~\ref{subsec:doublewell_2d_mixing}.}
        \label{fig:doublewell_2d_ergodic_rate}
    \end{figure}

    Figure~\ref{fig:doublewell_2d_ergodic_rate} depicts the logarithmic probability density of the 2D potential well and shows that the variance of the MLMC estimator increases linearly with time $T$ for a fixed time-step size $h$, as shown in Theorem~\ref{thm:convergence_estimator_discont}. Figure~\ref{fig:doublewell_2D_discontinuous} shows that the convergence rates obtained from the numerical simulation are consistent with theory. We observe that both the strong error and the variance rate for the higher-order change-of-measure scheme are $\cO(h^{3/2})$ as proved in Theorem~\ref{thm:convergence_estimator_discont}. From Figure~\ref{fig:triple_well_1D_discontinuous}, we also observe that the probability of the fine and the coarse trajectories diverging $\Prob (\Theta_{T}) = 0$ for $\nu_{1} = 1$. We observe that the kurtosis increases with the level of the multilevel estimator. This can be attributed to computing the weak approximation of a discontinuous payoff function.

    \begin{figure}[ht]
        \centering
        \includegraphics[width=0.49\linewidth]{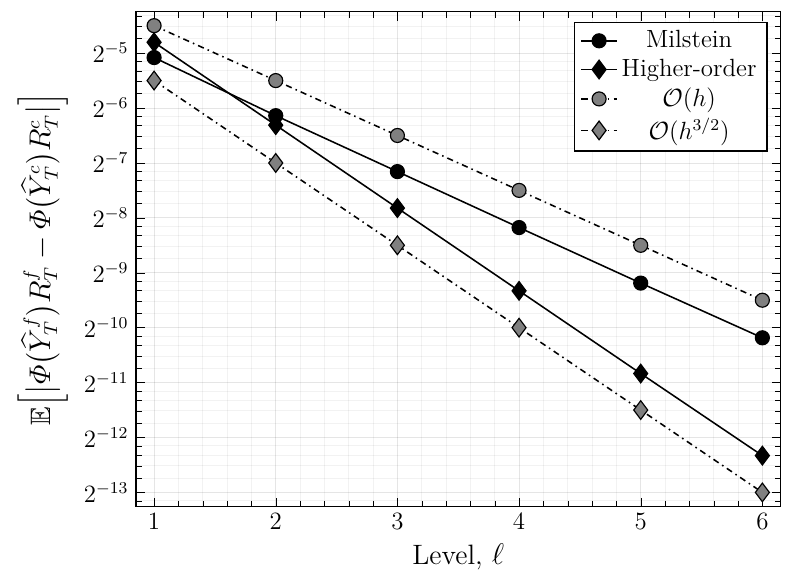}
        \includegraphics[width=0.49\linewidth]{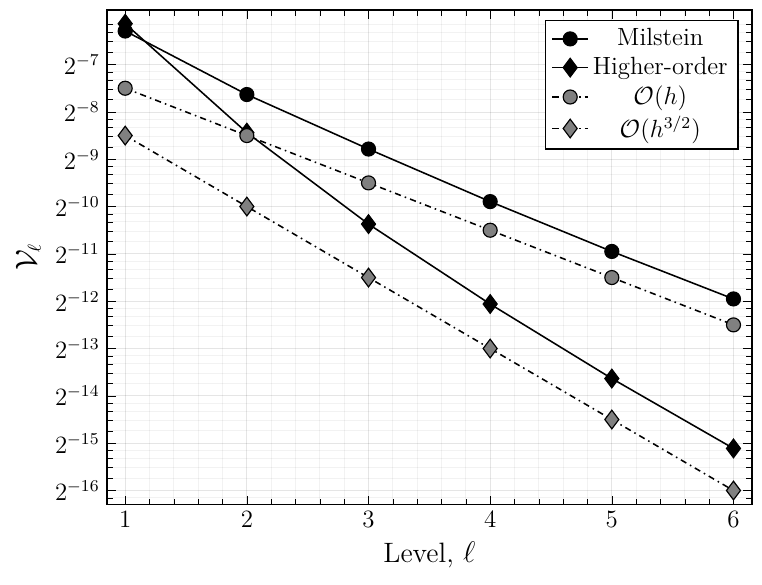}
        \includegraphics[width=0.49\linewidth]{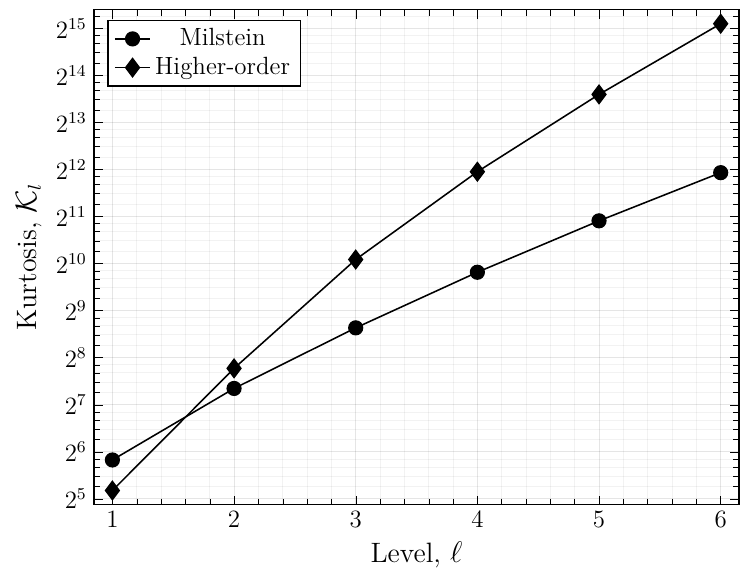}
        \includegraphics[width=0.49\linewidth]{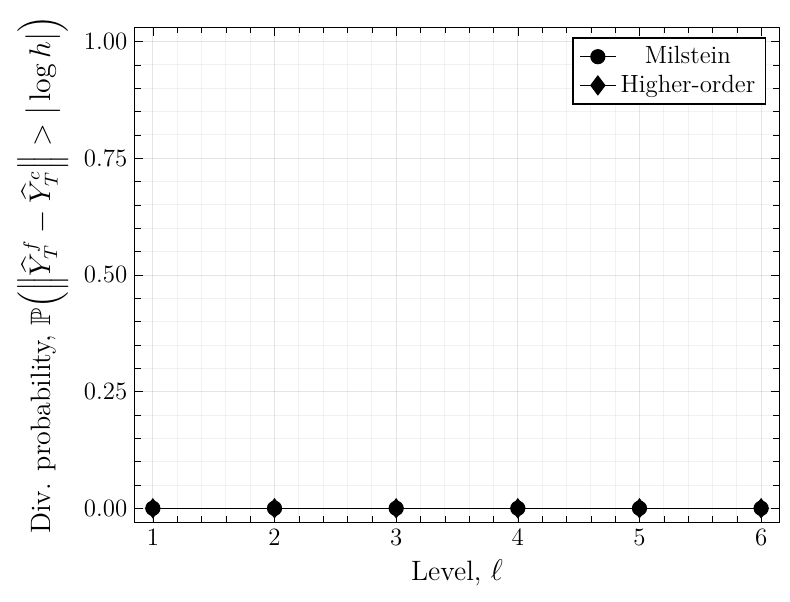}
        \caption{Numerical results for the problem considered in Section~\ref{subsec:doublewell_2d_mixing} at the terminal time, $T = 10$: strong error (top left), variance rate (top right), kurtosis (bottom left), divergence probability (bottom right)}
        \label{fig:doublewell_2D_discontinuous}
    \end{figure}

    From Figure~\ref{fig:doublewell_2D_discontinuous_mlmc}, it is observed that the numerical results match the theoretical results. By choosing $h_{0} = \mathcal{O} ( 1 / \sqrt{T \log T} )$ and $h_{0} = \mathcal{O} \big( T ^ {-1 / (3/2 - \xi_{1})} \big)$ for the Milstein and the higher-order schemes, respectively, for all $\xi_{1}, \xi_{2} > 0$, it is noted that the numerical results match the theoretical results. Furthermore, the computational cost of the higher-order scheme is lower than the Milstein scheme in terms of both the rate and the associated constant. We also note that the results extend to 2D SDE as shown in theory.

    \begin{figure}[ht]
        \centering
        \includegraphics[width=0.49\linewidth]{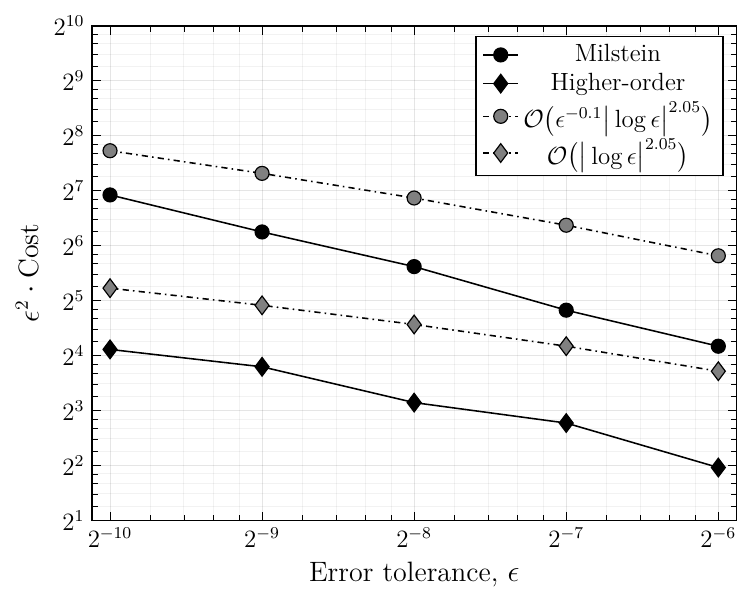}
        \includegraphics[width=0.49\linewidth]{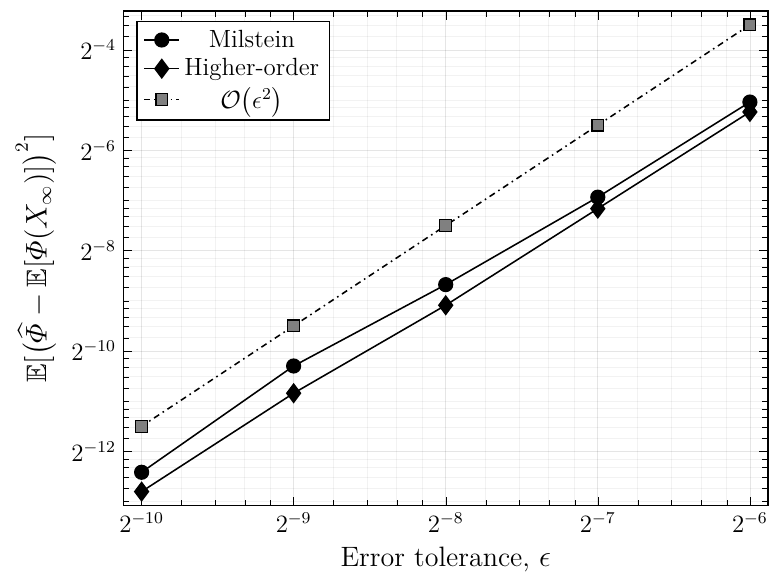}
        \includegraphics[width=0.49\linewidth]{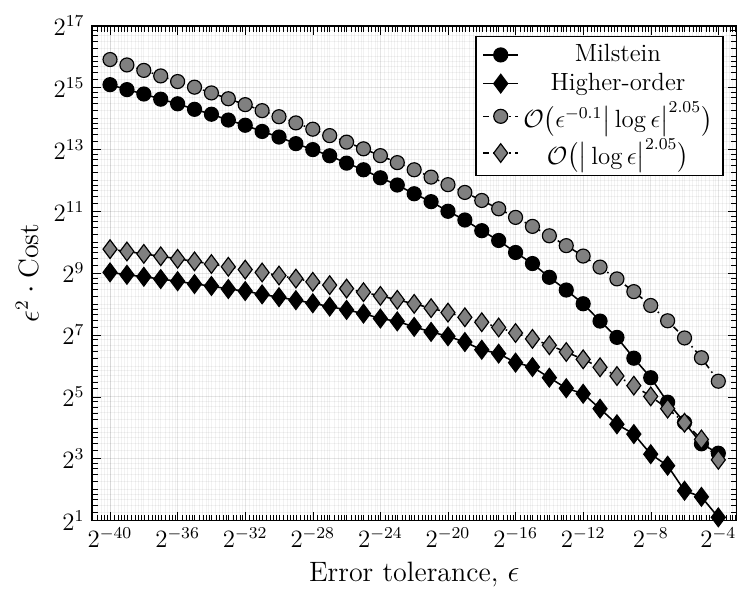}
        \caption{Change-of-measure MLMC scheme results for the problem considered in Section~\ref{subsec:doublewell_2d_mixing}: computational cost (top left), mean squared error (top right), asymptotic computational cost (bottom)}
        \label{fig:doublewell_2D_discontinuous_mlmc}
    \end{figure}

    \subsection{Thomas' cyclically symmetric attractor -- Lipschitz payoff}
    \label{subsec:3d_thomas_attractor_lipschitz}
    In \cite{thomas1999deterministic}, a cyclically symmetric attractor system is formulated, which like the Lorenz system, is a chaotic attractor system. To this end, we consider the SDE
    \begin{equation*}
        \begin{split}
            \di X_{1} &= \big( \sin(X_{2}) - 0.18 X_{1} \big) \di t + \di W^{1} \\
            \di X_{2} &= \big( \sin(X_{3}) - 0.18 X_{2} \big) \di t + \di W^{2} \\
            \di X_{3} &= \big( \sin(X_{1}) - 0.18 X_{3} \big) \di t + \di W^{3} \\
            X(0) &= (1, 2, 2)^{\top}.
        \end{split}
    \end{equation*}
    Note that the drift term given above satisfies Assumptions~\ref{asmp:lipschitz_a}, \ref{asmp:dissipativity_a}, and \ref{asmp:aD_lipschitz}. We are interested in computing the weak approximation of the uniformly Lipschitz continuous payoff function $\varPhi(x) \coloneqq \norm{x}$ w.r.t. the invariant measure of the SDE. The pseudo-reference solution is $\E [ \varPhi(X_{\infty}) ] = 3.9664$, rounded to 5 significant digits. We run $M = 10^{6}$ Monte Carlo simulations until the terminal time $T = 40$ using an initial time-step size $h_{0} = 2^{-4}$.

    \begin{figure}[ht]
        \centering
        \includegraphics[width=0.45\linewidth]{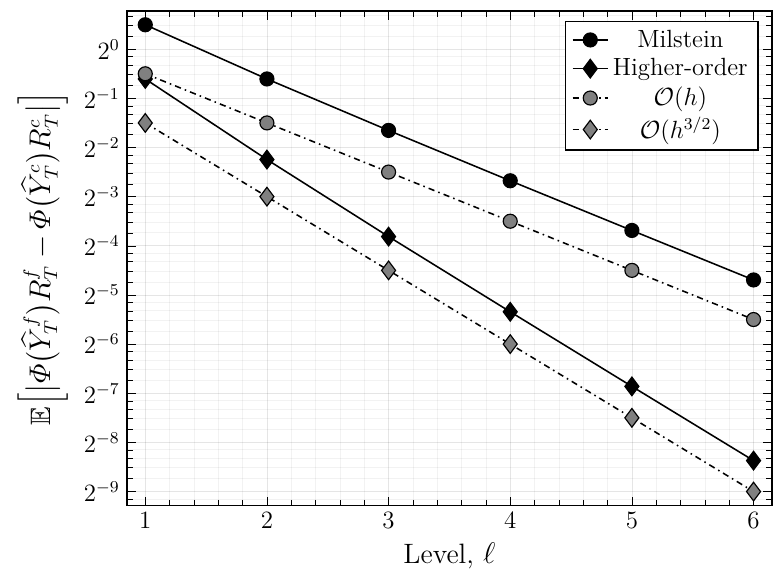}
        \includegraphics[width=0.45\linewidth]{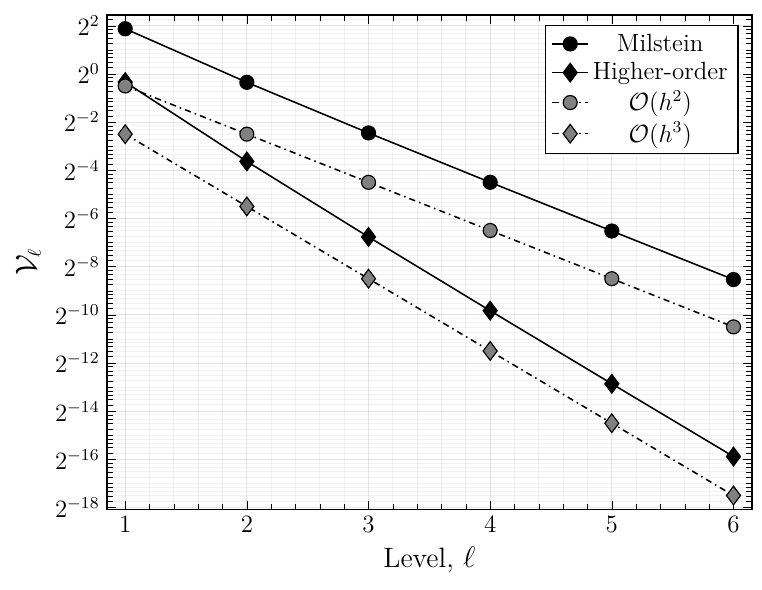}
        \includegraphics[width=0.45\linewidth]{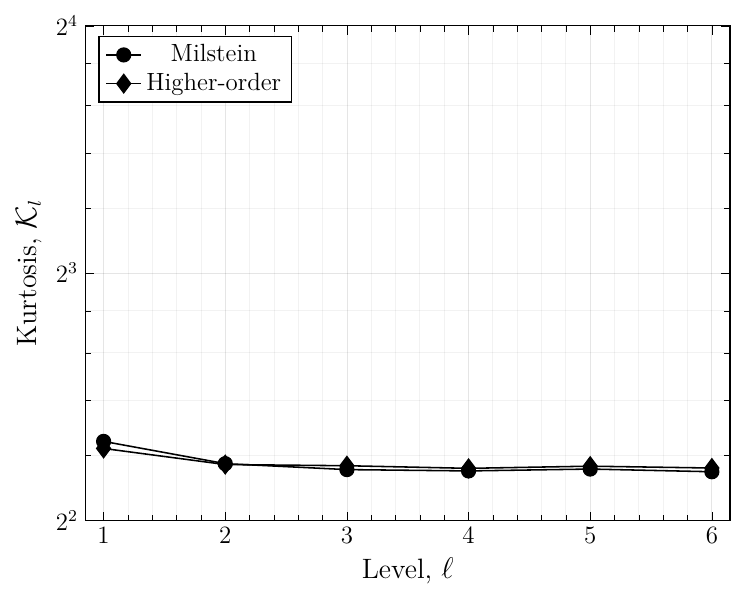}
        \includegraphics[width=0.45\linewidth]{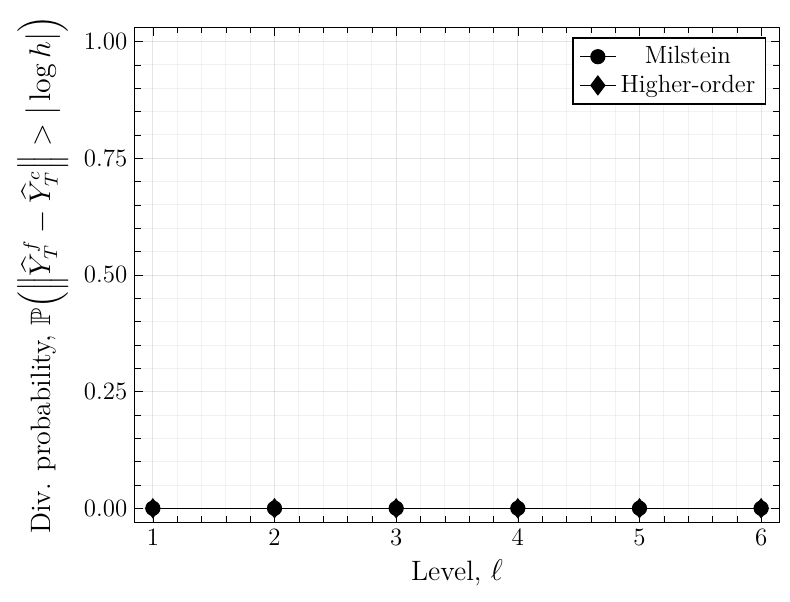}
        \caption{Numerical results for the problem considered in Section~\ref{subsec:3d_thomas_attractor_lipschitz} at the terminal time, $T = 40$: Strong error (top left), Variance rate (top right), Kurtosis (bottom left), and the expected total computational cost (bottom right)}
        \label{fig:thomas_attractor}
    \end{figure}

    From Figure~\ref{fig:thomas_attractor}, we observe that the numerical rates are consistent with the theoretical rates for a uniformly Lipschitz continuous payoff function. Unlike for the case of a discontinuous payoff, we observe that the kurtosis of both the numerical schemes is bounded when considering a uniformly Lipschitz continuous payoff. This leads to a better estimation of the sample variance on all levels of the multilevel estimator. We also observe the probability that the fine and the coarse trajectories diverge $\Prob ( \Theta_{T} ) = 0$ for $\nu_{1} = 1$.

    \begin{figure}[ht]
        \centering
        \includegraphics[width=0.48\linewidth]{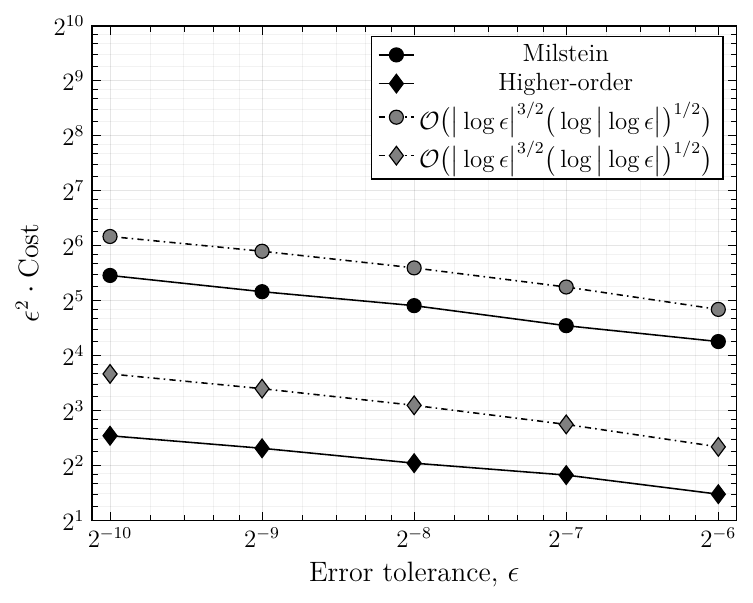}
        \includegraphics[width=0.49\linewidth]{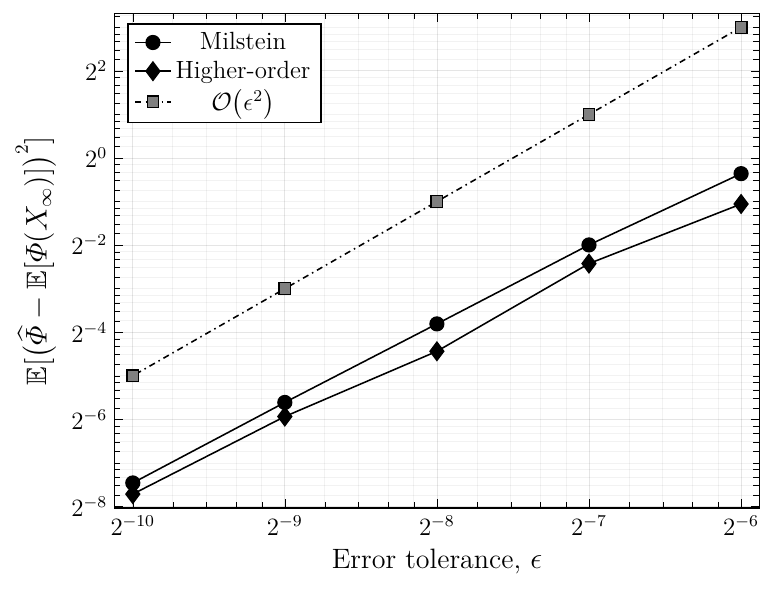}
        \includegraphics[width=0.49\linewidth]{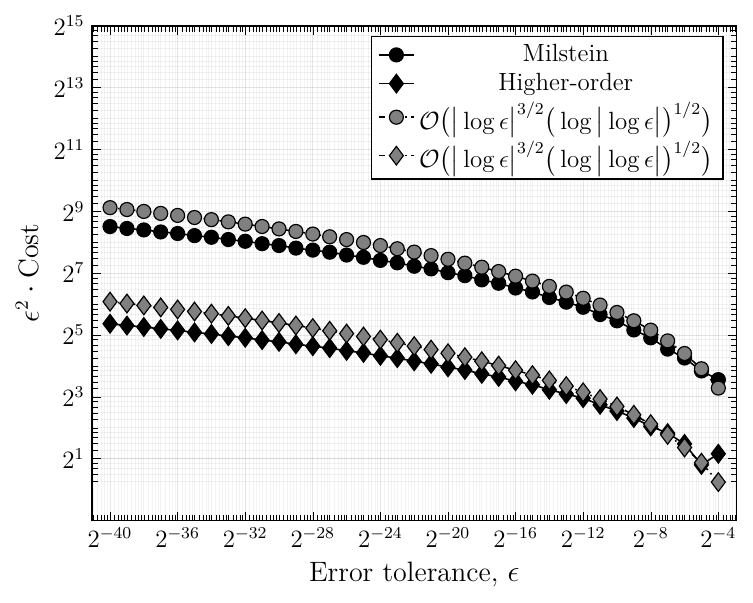}
        \includegraphics[width=0.49\linewidth]{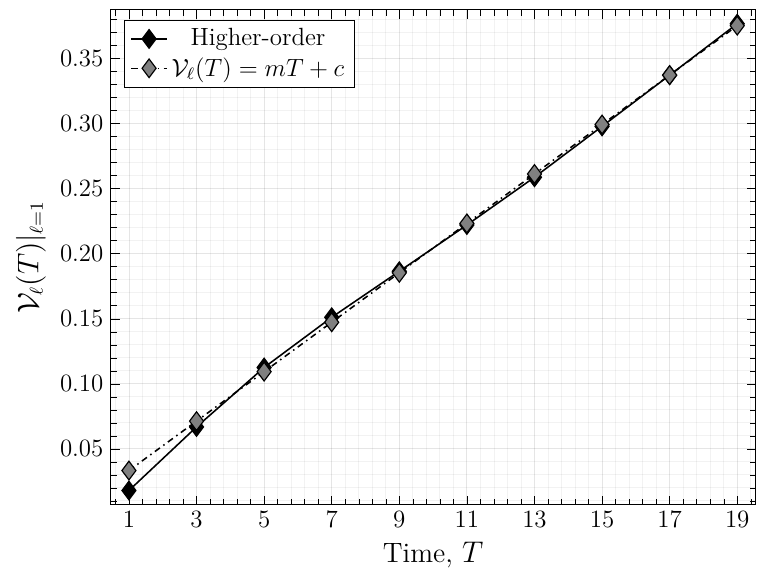}
        \caption{Change-of-measure MLMC scheme results for the problem considered in Section~\ref{subsec:3d_thomas_attractor_lipschitz}: computational cost (top left), mean squared error (top right), asymptotic computational cost (bottom left), variance growth rate of the MLMC estimator w.r.t terminal time on level $\ell = 1$ (bottom right).}
        \label{fig:thomas_attractor_mlmc}
    \end{figure}
    
    From Figure~\ref{fig:thomas_attractor_mlmc} we observe that for $h_{0} = 2^{-4}$, the variance grows linearly w.r.t. the terminal time $T$ which is consistent with the theoretical results in Theorem~\ref{thm:convergence_estimator}. The numerical experiments show that the computational cost of both the higher-order and the Milstein schemes are bounded from above by $\cO \big( \epsilon^{-2} \abs{\log \epsilon}^{3/2} (\log \abs{\log \epsilon})^{1/2} \big)$. This is because the majority of the computational cost is attributed to generating samples on level zero of the MLMC estimator for both numerical schemes. Although both numerical schemes have the same computational cost order, the computational cost constant associated with the higher-order numerical scheme is lower than the Milstein scheme to achieve an MSE of $\cO (\epsilon^{2})$. 
    
    \section{Conclusion and future scope}
    \label{sec:conclusion}
    In this paper, we developed a tractable higher-order numerical scheme to compute weak approximations of the invariant measure of an SDE using the MLMC method, that does not satisfy the contractivity condition. The fundamental idea is similar to one in \cite{fang2019multilevel}, where we introduce a spring term in the 1.5 strong It\^o--Taylor scheme that pulls the fine and the coarse trajectories of the multilevel estimator close to each other. This ensures that the fine and the coarse trajectories do not drift far away from each other while recovering the contractivity property of the SDE needed for the efficient implementation of MLMC method. We then use Girsanov's transformation to compute the weak approximations w.r.t. the new measure.

    We theoretically proved that the fine and the coarse trajectories of the multilevel estimator are $L^{p}(\Omega)$ stable. We proved that the pairwise coupled difference between the fine and the coarse trajectories of the multilevel estimator has $L^{p}(\Omega)$ convergence rate $\cO \big( \min ( p^{1/2} h^{1/2}, p h, p^{3/2} h^{3/2} ) \big)$ for all $p \geq 1$, and for all sufficiently small $h > 0$. We proved the $L^{p}(\Omega)$-norm convergence of the multilevel estimator to be $\cO \big( T^{1/2} p^{5/2} h^{3/2} \big)$ for a uniformly Lipschitz continuous payoff function. For a discontinuous payoff function defined by equation~\ref{eq:indicator function}, the $p$-th moment of the multilevel estimator was proved to be $\cO \Big( T^{p/2} p^{2p} h^{\frac{3}{2} - \xi} \Big)$ for all $p \geq 1$ and for all $\xi > 0$. Provided that the SDE under consideration satisfies Assumptions~\ref{asmp:lipschitz_a}, \ref{asmp:aD_lipschitz}, and \ref{asmp:dissipativity_a}, we proved that the computational cost involved in the implementation of the higher-order MLMC scheme is $\cO \big( \epsilon^{-2} \abs{\log \epsilon}^{3/2} (\log \abs{\log \epsilon})^{1/2} \big)$, and $\cO \big( \epsilon^{-2} \abs{\log \epsilon}^{5/3 + \xi} \big)$, where $\xi > 0$, for all uniformly Lipschitz continuous payoffs, and all discontinuous payoffs of the form given by equation~\eqref{eq:indicator function}, respectively.

    There are several ways in which the current work can be extended. Firstly, we can extend the analysis of the higher-order change-of-measure scheme to SDE with state-dependent diffusion coefficients. Another possible extension would be to use adaptive time-stepping based on the idea devised in \cite{wang2016_EM1, wang2017_EM2, wang2018}. By choosing time-steps adapted to the state of the process, and using the new higher-order change-of-measure scheme, we can extend the computation of weak approximations of invariant measure of SDE with non-uniformly Lipschitz continuous drift coefficients. It is important to note that the coupling between the coarse and the fine trajectories and the computation of the corresponding Radon--Nikodym derivatives for the higher-order change-of-measure scheme are not straightforward unlike the Milstein scheme due to the higher-order noise terms. Further analysis is needed to develop an optimal coupling and to compute the corresponding Radon--Nikodym derivatives. This poses an interesting problem for future research.
    
    The authors of \cite{hakon2023} construct a higher-order adaptive method for strong approximations of exit times of Itô stochastic differential equations (SDE). The fundamental idea is to use a smaller time-step size as the numerical solution gets closer to the boundary of the domain, and to use higher-order integration schemes for better approximation of the state of the Itô-diffusion. The authors prove that by using a higher-order order numerical integrator along with adaptive time-steps leads to a better strong error convergence rate of exit-time computations. Combining the ideas of adaptive timestepping for exit-time problems with the higher-order change-of-measure scheme may help in computing weak approximations of quantities of interest (QoI) that depend on the exit-time and the state of the process. This makes an interesting problem for future research.

    \section{Theoretical proofs}
    \label{sec:theoretical_proofs}

    \subsubsection*{Proof of Theorem~\ref{thm:stability}}
        We adopt a methodology similar to one in \cite{fang2019multilevel} to prove the stability of the numerical trajectories with the spring constant. To keep the notation concise, we denote the tuple of random variables $(\Delta W_{n}^{\PMeas}, \Delta Z_{n}^{\PMeas}) = (\Delta W_{n}, \Delta Z_{n})$ for all $n = 1, \ldots, N - 1$. We introduce the constants $C, \tildealpha, \tildebeta > 0$ that depend only on $S, d$, and the constants from Assumptions~\ref{asmp:lipschitz_a}, \ref{asmp:aD_lipschitz}, and \ref{asmp:dissipativity_a}. The constants $C, \tildealpha, \tildebeta$ may change from line to line in the proof argument.
        
        Computing the square of the Euclidean norm of the fine and the coarse trajectories at the odd time-step, we obtain
        \begin{equation*}
            \begin{split}
                \norm{\hatY_{2n + 1}^{f}}^{2} &= \norm{Sh \hatY_{2n}^{c} + (1 - Sh) \bigg\{ \hatY_{2n}^{f} + \frac{A_{2n}^{f}}{1 - Sh} \bigg\}}^{2}, \\
                \norm{\hatY_{2n + 1}^{c}}^{2} &= \norm{Sh \hatY_{2n}^{f} + (1 - Sh) \bigg\{ \hatY_{2n}^{c} + \frac{A_{2n}^{c}}{1 - Sh} \bigg\}}^{2},
            \end{split}
        \end{equation*}
        where \begin{equation*}
            \begin{split}
                A_{2n}^{f} \coloneqq h a(\hatY_{2n}^{f}) + \Delta W_{2n} + Da(\hatY_{2n}^{f}) \Delta Z_{2n} + \frac{h^{2}}{2} \cA a(\hatY_{2n}^{f}), \\
                A_{2n}^{c} \coloneqq h a(\hatY_{2n}^{c}) + \Delta W_{2n} + Da(\hatY_{2n}^{c}) \Delta Z_{2n} + \frac{h^{2}}{2} \cA a(\hatY_{2n}^{c}).
            \end{split}
        \end{equation*}
        Note that for the given spring constant $S > 0$, and for all sufficiently small $h > 0$, \[ \frac{1}{1 - Sh} \leq 2. \]
        Since the Euclidean norm squared is a convex function, from Jensen's inquality, we have
        \begin{align*}
            \norm{\theta x_{1} + (1 - \theta) x_{2}}^{2} \leq \theta \norm{x_{1}}^{2} + (1 - \theta) \norm{x_{2}}^{2}, \qquad \forall x_{1}, x_{2} \in \bR^{d}, \;\; \theta \in (0, 1).
        \end{align*}
        Using this inequality, we obtain
        \begin{equation}
            \label{eq:fine_even_norm_squared}
            \begin{split}
                \norm{\hatY_{2n + 1}^{f}}^{2} &\leq Sh \norm{\hatY_{2n}^{c}}^{2} + (1 - Sh) \norm{ \hatY_{2n}^{f} + \frac{A_{2n}^{f}}{1 - Sh} }^{2} \\
                &\leq Sh \norm{\hatY_{2n}^{c}}^{2} + (1 - Sh) \norm{\hatY_{2n}^{f}}^{2} + 2 \norm{A_{2n}^{f}}^{2} + 2 \inner{\hatY_{2n}^{f}, A_{2n}^{f}}.
            \end{split}
        \end{equation}
        Let $b_{i} \in \bR^{d}$ for $i = 1, 2, \ldots, m$. From H\"older's inequality, we then have
        \begin{equation*}
            \norm{\sum_{i = 1}^{m} b_{i}}^{p} \leq m^{p - 1} \Biggl( \sum_{i = 1}^{m} \norm{b_{i}}^{p} \Biggr), \qquad \forall p \geq 1.
        \end{equation*}
        Using the above inequality, we can bound $\norm{A_{2n}^{f}}^{2}$ as
        \begin{equation}
            \label{eq:norm_A2nf}
            \begin{split}
                \norm{A_{2n}^{f}}^{2} &\leq 4 \bigg\{ h^{2} \norm{a(\hatY_{2n}^{f})}^{2} + \norm{\Delta W_{2n}}^{2} + \norm{Da(\hatY_{2n}^{f}) \Delta Z_{2n}}^{2} + \frac{h^{4}}{4} \norm{\cA a(\hatY_{2n}^{f})}^{2} \bigg\} \\
                ~
                &\leq C \bigg\{ h^{2} \Big( 1 + \norm{\hatY_{2n}^{f}}^{2} \Big) + \norm{\Delta W_{2n}}^{2} + \norm{\Delta Z_{2n}}^{2} \bigg\}.
            \end{split}
        \end{equation}
        Using Young's inequality, Cauchy-Schwarz inequality, and Assumption~\ref{asmp:dissipativity_a}, the inner product term $\inner{\hatY_{2n}^{f}, A_{2n}^{f}}$ can be bounded as
        \begin{equation}
            \label{eq:inner_A2nf}
            \begin{split}
                \inner{\hatY_{2n}^{f}, A_{2n}^{f}} &= h \inner{\hatY_{2n}^{f}, a(\hatY_{2n}^{f})} + \inner{\hatY_{2n}^{f}, \Delta W_{2n}} + \inner{\hatY_{2n}^{f}, Da(\hatY_{2n}^{f}) \Delta Z_{2n}} \\
                &\quad + \frac{h^{2}}{2} \inner{\hatY_{2n}^{f}, \cA a(\hatY_{2n}^{f})} \\
                ~
                &\leq -\tildealpha h \norm{\hatY_{2n}^{f}}^{2} + \inner{\hatY_{2n}^{f}, \Delta W_{2n}} + \inner{\hatY_{2n}^{f}, Da(\hatY_{2n}^{f}) \Delta Z_{2n}} + \tildebeta h.
            \end{split}
        \end{equation}
        Combining equations~\eqref{eq:norm_A2nf} and \eqref{eq:inner_A2nf} with \eqref{eq:fine_even_norm_squared}, we obtain
        \begin{equation}
            \label{eq:fine_odd_bound}
            \begin{split}
                \norm{\hatY_{2n + 1}^{f}}^{2} &\leq Sh \norm{\hatY_{2n}^{c}}^{2} + (1 - Sh - 2 \tildealpha h) \norm{\hatY_{2n}^{f}}^{2} + C \Big\{ \norm{\Delta W_{2n}}^{2} + \norm{\Delta Z_{2n}}^{2} \Big\} \\
                &\quad + 2 \inner{\hatY_{2n}^{f}, \Delta W_{2n}} + 2 \inner{\hatY_{2n}^{f}, Da(\hatY_{2n}^{f}) \Delta Z_{2n}} + 2 \tildebeta h.
            \end{split}
        \end{equation}
        Similarly, we have the following upper bound for the coarse path at the odd time-step
        \begin{equation}
            \label{eq:coarse_odd_bound}
            \begin{split}
                \norm{\hatY_{2n + 1}^{c}}^{2} &\leq Sh \norm{\hatY_{2n}^{f}}^{2} + (1 - Sh - 2 \tildealpha h) \norm{\hatY_{2n}^{c}}^{2} + C \Big\{ \norm{\Delta W_{2n}}^{2} + \norm{\Delta Z_{2n}}^{2} \Big\} \\
                &\quad + 2 \inner{\hatY_{2n}^{c}, \Delta W_{2n}} + 2 \inner{\hatY_{2n}^{c}, Da(\hatY_{2n}^{c}) \Delta Z_{2n}} + 2 \tildebeta h.
            \end{split}
        \end{equation}
        Using the similar approach as above, we can write the square of the Euclidean norm for the fine and the coarse paths at the even time-step as
        \begin{equation*}
            \begin{split}
                \norm{\hatY_{2n + 2}^{f}}^{2} &= \norm{Sh \hatY_{2n + 1}^{c} + (1 - Sh) \bigg\{ \hatY_{2n + 1}^{f} + \frac{A_{2n + 1}^{f}}{1 - Sh} \bigg\}}^{2}, \\
                ~
                \norm{\hatY_{2n + 2}^{c}}^{2} &= \norm{2Sh \hatY_{2n}^{f} + (1 - 2Sh) \bigg\{ \hatY_{2n}^{c} + \frac{A_{2n + 1}^{c}}{1 - 2Sh} \bigg\}}^{2},
            \end{split}
        \end{equation*}
        where
        \begin{equation*}
            \begin{split}
                A_{2n + 1}^{f} &\coloneqq h a(\hatY_{2n + 1}^{f}) + \Delta W_{2n + 1} + Da(\hatY_{2n + 1}^{f}) \Delta Z_{2n + 1} + \frac{h^{2}}{2} \cA a(\hatY_{2n + 1}^{f}), \\
                ~
                A_{2n + 1}^{c} &\coloneqq 2h a(\hatY_{2n}^{c}) + (\Delta W_{2n} + \Delta W_{2n + 1}) \\
                &\quad + Da(\hatY_{2n}^{c}) (\Delta Z_{2n + 1} + \Delta Z_{2n} + h \Delta W_{2n}) + 2 h^{2} \cA a(\hatY_{2n}^{c}).
            \end{split}
        \end{equation*}
        Using Young's inequality, Cauchy-Schwarz inequality, and Assumption~\ref{asmp:dissipativity_a}, we obtain
        \begin{align}
            \label{eq:fine_even_bound_1}
            \begin{split}
                \norm{\hatY_{2n + 2}^{f}}^{2} &\leq Sh \norm{\hatY_{2n + 1}^{c}}^{2} + (1 - Sh - 2 \tildealpha h) \norm{\hatY_{2n + 1}^{f}}^{2} \\
                &\quad + C \Big\{ \norm{\Delta W_{2n + 1}}^{2} + \norm{\Delta Z_{2n + 1}}^{2} \Big\} + 2 \inner{\hatY_{2n + 1}^{f}, \Delta W_{2n + 1}} \\
                &\quad + 2 \inner{\hatY_{2n + 1}^{f}, Da(\hatY_{2n + 1}^{f}) \Delta Z_{2n + 1}} + 2 \tildebeta h, \\
            \end{split} \\
            \label{eq:coarse_even_bound}
            \begin{split}
                \norm{\hatY_{2n + 2}^{c}}^{2} &\leq 2Sh \norm{\hatY_{2n}^{f}}^{2} + (1 - 2Sh - 4 \tildealpha h) \norm{\hatY_{2n}^{c}}^{2} + C \Big\{ \norm{\Delta Z_{2n + 1}}^{2} + \norm{\Delta Z_{2n}}^{2} \\
                &\quad + \norm{\Delta W_{2n + 1}}^{2} + \norm{\Delta W_{2n}}^{2} \Big\} + 2 \inner{\hatY_{2n}^{c}, \Delta W_{2n}} + 2 \inner{\hatY_{2n}^{c}, \Delta W_{2n + 1}} \\
                &\quad + 2 \inner{\hatY_{2n}^{c}, Da(\hatY_{2n}^{c}) \Delta Z_{2n + 1}} + 2 \inner{\hatY_{2n}^{c}, Da(\hatY_{2n}^{c}) \Delta Z_{2n}} + 4 \tildebeta h.
            \end{split}
        \end{align}
        Substituting equations~\eqref{eq:fine_odd_bound} and \eqref{eq:coarse_odd_bound} into \eqref{eq:fine_even_bound_1} and simplifying the inequality, we obtain
        \begin{equation}
            \label{eq:fine_even_bound}
            \begin{split}
                \norm{\hatY_{2n + 2}^{f}}^{2} &\leq 2Sh (1 - Sh - 2 \tildealpha h) \norm{\hatY_{2n}^{c}}^{2} + (1 - 2Sh - 4 \tildealpha h + C h^{2}) \norm{\hatY_{2n}^{f}}^{2} + C \Big\{ \norm{\Delta W_{2n}}^{2} \\
                &\quad + \norm{\Delta Z_{2n}}^{2} + \norm{\Delta W_{2n + 1}}^{2} + \norm{\Delta Z_{2n + 1}}^{2} \Big\} \\
                &\quad + 2 \inner{Sh \hatY_{2n}^{c} + (1 - Sh - 2 \tildealpha h) \hatY_{2n}^{f}, \Delta W_{2n}} \\
                &\quad + 2 \inner{Sh (Da(\hatY_{2n}^{c}))^{\top} \hatY_{2n}^{c} + (1 - Sh - 2 \tildealpha h) (Da(\hatY_{2n}^{f}))^{\top} \hatY_{2n}^{f}, \Delta Z_{2n}} \\
                &\quad + 2 \inner{\hatY_{2n + 1}^{f}, \Delta W_{2n + 1}} + 2 \inner{(Da(\hatY_{2n + 1}^{f}))^{\top} \hatY_{2n + 1}^{f}, \Delta Z_{2n + 1}} + 4 \tildebeta h - 4 \tildealpha \tildebeta h^{2}.
            \end{split}
        \end{equation}
        Combining equations~\eqref{eq:fine_even_bound} and \eqref{eq:coarse_even_bound}, note that for any $\gamma \in (0, \tildealpha)$, and for all sufficiently small $h > 0$, the following inequality holds
        \begin{equation}
            \label{eq:combine_fc_even_1}
            \begin{split}
                \norm{\hatY_{2n + 2}^{f}}^{2} + \norm{\hatY_{2n + 2}^{c}}^{2} &\leq (1 - 4 \gamma h) \Bigg( \norm{\hatY_{2n}^{f}}^{2} + \norm{\hatY_{2n}^{f}}^{2} \Bigg) + C \Big\{ \norm{\Delta Z_{2n + 1}}^{2} + \norm{\Delta Z_{2n}}^{2} \\
                &\quad + \norm{\Delta W_{2n + 1}}^{2} + \norm{\Delta W_{2n}}^{2} \Big\} \\
                &\quad + 2 \inner{(1 + Sh) \hatY_{2n}^{c} + (1 - Sh - 2 \tildealpha h) \hatY_{2n}^{f}, \Delta W_{2n}} \\
                &\quad + 2 \inner{\hatY_{2n}^{c} + \hatY_{2n + 1}^{f}, \Delta W_{2n + 1}} \\
                &\quad + 2 \inner{(Da(\hatY_{2n}^{c}))^{\top} \hatY_{2n}^{c} + (Da(\hatY_{2n + 1}^{f}))^{\top} \hatY_{2n + 1}^{f}, \Delta Z_{2n + 1}} \\
                &\quad + 2 \inner{(1 + Sh) (Da(\hatY_{2n}^{c}))^{\top} \hatY_{2n}^{c}, \Delta Z_{2n}} \\
                &\quad + 2 \inner{(1 - Sh - 2 \tildealpha h) (Da(\hatY_{2n}^{f}))^{\top} \hatY_{2n}^{f}, \Delta Z_{2n}} + 8 \tildebeta h.
            \end{split}
        \end{equation}
        Note that for all sufficiently small $h > 0, \; (1 - 4 \gamma h) \leq \e^{-4 \gamma h}$, and $\e^{4 \gamma h} \leq 2$. Multiplying both sides of equation~\eqref{eq:combine_fc_even_1} by $\e^{2 \gamma t_{2n + 2}}$, we obtain
        \begin{equation*}
            \begin{split}
                \e^{2 \gamma t_{2n + 2}} \Bigg\{ \norm{\hatY_{2n + 2}^{f}}^{2} + \norm{\hatY_{2n + 2}^{c}}^{2} \Bigg\} &\leq \e^{2 \gamma t_{2n}} \Bigg\{ \norm{\hatY_{2n}^{f}}^{2} + \norm{\hatY_{2n}^{c}}^{2} \Bigg\} + C \e^{2 \gamma t_{2n}} \Bigg\{ \norm{\Delta Z_{2n + 1}}^{2} \\
                &\quad + \norm{\Delta Z_{2n}}^{2} + \norm{\Delta W_{2n + 1}}^{2} + \norm{\Delta W_{2n}}^{2} \Bigg\} \\
                &\quad + 2 \e^{2 \gamma t_{2n}} \inner{\varphi_{1, 2n}, \Delta W_{2n}} + 2 \e^{2 \gamma t_{2n}} \inner{\varphi_{1, 2n + 1}, \Delta W_{2n + 1}} \\
                &\quad + 2 \e^{2 \gamma t_{2n}} \inner{\varphi_{2, 2n}, \Delta Z_{2n}} + 2 \e^{2 \gamma t_{2n}} \inner{\varPhi_{2, 2n + 1}, \Delta Z_{2n + 1}} \\
                &\quad + 16 \tildebeta h \e^{2 \gamma t_{2n}},
            \end{split}
        \end{equation*}
        where
        \begin{equation*}
            \begin{split}
                \varphi_{1, 2n} &\coloneqq \e^{4 \gamma h} \Big\{ (1 + Sh) \hatY_{2n}^{c} + (1 - Sh - 2 \tildealpha h) \hatY_{2n}^{f} \Big\}, \\
                \varphi_{1, 2n + 1} &\coloneqq \e^{4 \gamma h} \Big\{ \hatY_{2n}^{c} + \hatY_{2n + 1}^{f} \Big\}, \\
                \varphi_{2, 2n} &\coloneqq \e^{4 \gamma h} \Big\{ (1 + Sh) (Da(\hatY_{2n}^{c}))^{\top} \hatY_{2n}^{c} + (1 - Sh - 2 \tildealpha h) (Da(\hatY_{2n}^{f}))^{\top} \hatY_{2n}^{f} \Big\}, \\
                \varphi_{2, 2n + 1} &\coloneqq \e^{4 \gamma h} \Big\{ (Da(\hatY_{2n}^{c}))^{\top} \hatY_{2n}^{c} + (Da(\hatY_{2n + 1}^{f}))^{\top} \hatY_{2n + 1}^{f} \Big\}.
            \end{split}
        \end{equation*}
        Summing the above inequality over multiple time-steps, we obtain
        \begin{equation}
            \label{eq:even_combine_sum}
            \begin{split}
                \e^{2 \gamma t_{2n}} \Bigg\{ \norm{\hatY_{2n}^{f}}^{2} + \norm{\hatY_{2n}^{c}}^{2} \Bigg\} &\leq 2 \norm{x_{0}}^{2} + C \Bigg\{ \sum_{k = 0}^{2n - 1} \e^{2 \gamma t_{k}} \norm{\Delta Z_{k}}^{2} + \sum_{k = 0}^{2n - 1} \e^{2 \gamma t_{k}} \norm{\Delta W_{k}}^{2} \Bigg\} \\
                &\quad + 2 \sum_{k = 0}^{2n - 1} \e^{2 \gamma t_{k}} \inner{\varphi_{1, k}, \Delta W_{k}} + 2 \sum_{k = 0}^{2n - 1} \e^{2 \gamma t_{k}} \inner{\varphi_{2, k}, \Delta Z_{k}} \\
                &\quad + \sum_{k = 0}^{n - 1} 16 \tildebeta h \e^{2 \gamma t_{2k}}.
            \end{split}
        \end{equation}
        Note that for all sufficiently small $h > 0 , \; \frac{1}{1 - 2 \tildealpha h} \leq 2$. Combining equations~\eqref{eq:fine_odd_bound}, \eqref{eq:coarse_odd_bound} and using Young's inequality, we have
        \begin{align}
            \notag
            \norm{\hatY_{2n + 1}^{f}}^{2} + \norm{\hatY_{2n + 1}^{c}}^{2} &\leq (1 - 2 \tildealpha h) C \Bigg\{  \norm{\hatY_{2n}^{f}}^{2} + \norm{\hatY_{2n}^{c}}^{2} + \norm{\Delta W_{2n}}^{2} + \norm{\Delta Z_{2n}}^{2} + 8 \tildebeta h \Bigg\}
        \end{align}
        Multiplying both sides of the above inequality by the $\e^{2 \gamma t_{2n + 1}}$ and substituting equation~\eqref{eq:even_combine_sum} into the above inequality, we obtain
        \begin{equation}
            \label{eq:odd_combine_sum}
            \begin{split}
                \e^{2 \gamma t_{2n + 1}} \Bigg\{ \norm{\hatY_{2n + 1}^{f}}^{2} + \norm{\hatY_{2n + 1}^{c}}^{2} \Bigg\} &\leq C \Bigg\{ 2 \norm{x_{0}}^{2} + \sum_{k = 0}^{2n} \e^{2 \gamma t_{k}} \norm{\Delta Z_{k}}^{2} + \sum_{k = 0}^{2n} \e^{2 \gamma t_{k}} \norm{\Delta W_{k}}^{2} \\
                &\quad + \sum_{k = 0}^{2n - 1} \e^{2 \gamma t_{k}} \inner{\varphi_{1, k}, \Delta W_{k}} + \sum_{k = 0}^{2n - 1} \e^{2 \gamma t_{k}} \inner{\varPhi_{2, k}, \Delta Z_{k}} \\
                &\quad + \sum_{k = 0}^{n - 1} 16 \tildebeta h \e^{2 \gamma t_{2k}} \Bigg\}.
            \end{split}
        \end{equation}
        Combining equations~\eqref{eq:even_combine_sum} and \eqref{eq:odd_combine_sum}, raising both sides of the inequality to the power $p / 2$, taking the supremum over $n \in [0, N]$ and taking the expectation on both sides of the inequality, we obtain
        \begin{align}
            \E \Bigg[ \sup_{0 \leq n \leq N} \e^{\gamma p t_{n}} \Bigg\{ \norm{\hatY_{t_{n}}^{f}}^{2} + \norm{\hatY_{t_{n}}^{c}}^{2} \Bigg\}^{p / 2} \Bigg] &\leq \widetilde{C} \Big( J_{1} + J_{2} + J_{3} + J_{4} + J_{5} + J_{6} \Big),
        \end{align}
        where $\widetilde{C} \coloneqq \max (6^{p/2 - 1} C^{p/2}, 1)$, and 
        \begin{align*}
            J_{1} &\coloneqq 2^{p/2} \norm{x_{0}}^{p}, & J_{4} &\coloneqq \E \Bigg[ \sup_{1 \leq n \leq N/2} \Bigg| \sum_{k = 0}^{2n - 1} \e^{2 \gamma t_{k}} \inner{\varphi_{1, k}, \Delta W_{k}} \Bigg|^{p/2} \Bigg], \\
            ~
            J_{2} &\coloneqq \E \Bigg[ \Bigg| \sum_{k = 0}^{N - 1} \e^{2 \gamma t_{k}} \norm{\Delta Z_{k}}^{2} \Bigg|^{p/2} \Bigg], & J_{5} &\coloneqq \E \Bigg[ \sup_{1 \leq n \leq N/2} \Bigg| \sum_{k = 0}^{2n - 1} \e^{2 \gamma t_{k}} \inner{\varphi_{2, k}, \Delta Z_{k}} \Bigg|^{p/2} \Bigg], \\
            ~
            J_{3} &\coloneqq \E \Bigg[ \Bigg| \sum_{k = 0}^{N - 1} \e^{2 \gamma t_{k}} \norm{\Delta W_{k}}^{2} \Bigg|^{p/2} \Bigg], & J_{6} &\coloneqq \abs{\sum_{k = 0}^{N - 1} 16 \tildebeta \e^{2 \gamma t_{k}} h}^{p/2}.
        \end{align*}
        Recall that \[ \Delta W_{k} = \Delta V_{1}, \qquad \Delta Z_{k} = \frac{h}{2} \bigg( \Delta V_{1} + \frac{\Delta V_{2}}{\sqrt{3}} \bigg). \]
        Note that the term $J_{1}$ is a constant and is bounded. To bound the terms $J_{2}, J_{3}, J_{4}$ and $J_{6}$, we make use of the fact that
        \begin{equation*}
            \begin{aligned}
                \E[\norm{\Delta W_{k}}^{p} / h^{p/2}] &= \E[\norm{\Delta V_{1, k}}^{p} / h^{p/2}] \leq d^{p/2} p!! \leq d^{p/2} p^{p/2}, \\
                \E[\norm{\Delta Z_{k}}^{p} / h^{3p/2}] &= \frac{1}{2^{p} h^{p/2}} \E \Bigg[ \norm{\Delta V_{1, k} + \frac{1}{\sqrt{3}} \Delta V_{2, k}}^{p} \Bigg] \\
                &\leq \frac{1}{2 h^{p/2}} \Bigg( \E \Big[ \norm{\Delta V_{1, k}}^{p} \Big] + \frac{1}{3^{p/2}} \E \Big[ \norm{\Delta V_{2, k}}^{p} \Big] \Bigg) \leq 2 d^{p/2} p^{p/2}.
            \end{aligned}
        \end{equation*}
        The terms $J_{3}$ and $J_{6}$ can be bounded in a similar way as in \cite{fang2019multilevel} yielding the bounds
        \begin{equation*}
            \begin{split}
                J_{3} &= \E \Bigg[ \Bigg| \sum_{k = 0}^{N - 1} \e^{2 \gamma t_{k}} \norm{\Delta W_{k}}^{2} \Bigg|^{p/2} \Bigg] \leq d^{p/2} p^{p/2} \frac{\e^{\gamma p T}}{(2 \gamma)^{p/2}}, \\
                J_{6} &= \abs{\sum_{k = 0}^{N - 1} 16 \tildebeta \e^{2 \gamma t_{k}} h}^{p/2} \leq 4^{p} \tildebeta^{p/2} \frac{\e^{\gamma p T}}{(2 \gamma)^{p/2}}.
            \end{split}
        \end{equation*}
        Similarly, the term $J_{2}$ yields the bounds
        \begin{align*}
            J_{2} = \E \Bigg[ \Bigg| \sum_{k = 0}^{N - 1} \e^{2 \gamma t_{k}} \norm{\Delta Z_{k}}^{2} \Bigg|^{p/2} \Bigg] &= \E \Bigg[ \Bigg| \sum_{k = 0}^{N - 1} \e^{2 \gamma t_{k}} h \frac{\norm{\Delta Z_{k}}^{2}}{h} \Bigg|^{p/2} \Bigg] \\
            &\leq \abs{\sum_{k = 0}^{N - 1} \e^{2 \gamma t_{k}} h}^{p/2 - 1} \E \Bigg[ \sum_{k = 0}^{N - 1} \e^{2 \gamma t_{k}} h^{1 + p} \frac{\norm{\Delta Z_{k}}^{p}}{h^{3p/2}} \Bigg] \\
            &\leq 2 d^{p/2} p^{p/2} h^{p} \frac{\e^{\gamma p T}}{(2 \gamma)^{p /2}}.f
        \end{align*}
        To bound the terms $J_{4}$ and $J_{5}$, we rewrite the summations in the terms $J_{4}$ and $J_{5}$ as It\^o integrals, and remind ourselves that
        \begin{equation*}
                dW_{s} = dV_{1}, \qquad dZ_{s} = \frac{h}{2} \bigg( dV_{1} + \frac{dV_{2}}{\sqrt{3}} \bigg).
        \end{equation*}
        where $dV_{1}$ and $dV_{2}$ are independent standard Wiener processes. Applying the Burkholder--Davis--Gundy (BDG) inequality as stated in \cite{BARLOW1982198}, we then have
        \begin{equation*}
            \begin{split}
                J_{4} &\coloneqq \E \Bigg[ \sup_{0 \leq t \leq T} \abs{\int_{0}^{t} \e^{2 \gamma \floor{s / h} h} \inner{\varphi_{1, \floor{s / h}}, \di W_{s}} }^{p / 2} \Bigg] \leq (C_{\text{BDG}} p)^{p/4} \E \Bigg[ \abs{\sum_{k = 0}^{N - 1} \e^{4 \gamma t_{k}} h \norm{\varphi_{1, k}}^{2}}^{p/4} \Bigg], \\
            \end{split}
        \end{equation*}
        where $C_{\text{BDG}} > 0$ is a constant that is independent of $p$. Using Young's inequality, for all sufficiently small $h > 0$, we have
        \begin{equation*}
            \norm{\varphi_{1, 2n}}^{2} \leq 8 \Big( \norm{\hatY_{2n}^{f}}^{2} + \norm{\hatY_{2n}^{c}}^{2} \Big), \qquad \norm{\varphi_{1, 2n + 1}}^{2} \leq 4 \Big( \norm{\hatY_{2n}^{f}}^{2} + \norm{\hatY_{2n + 1}^{c}}^{2} \Big).
        \end{equation*}
        Using the above inequalities, we can write
        \begin{equation*}
            \begin{split}
                J_{4} &\leq (12 C_{\text{BDG}} p)^{p/4} \abs{\sum_{k = 0}^{N - 1} \e^{2 \gamma t_{k}} h}^{p/4 - 1} \E \Bigg[ \sum_{k = 0}^{N - 1} \e^{2 \gamma t_{k}} h \e^{\gamma p t_{k} / 2} \Big( \norm{\hatY_{t_{k}}^{c}}^{2} + \norm{\hatY_{t_{k}}^{f}}^{2} \Big)^{p / 4} \Bigg] \\
                ~
                &\leq (12 C_{\text{BDG}} p)^{p/4} \Bigl( \frac{\e^{2 \gamma T}}{2 \gamma} \Bigr)^{p/4 - 1} \E \Bigg[ \sum_{k = 0}^{N - 1} \e^{2 \gamma t_{k}} h \cdot \sup_{0 \leq n \leq N} \e^{\gamma p t_{n} / 2} \Big( \norm{\hatY_{t_{n}}^{c}}^{2} + \norm{\hatY_{t_{n}}^{f}}^{2} \Big)^{p/4}  \Bigg] \\
                ~
                &\leq (12 C_{\text{BDG}} p)^{p/4} \E \Bigg[ \frac{\e^{\gamma p T / 2}}{(2 \gamma)^{p/4}} \cdot \sup_{0 \leq n \leq N} \e^{\gamma p t_{n} / 2} \Big( \norm{\hatY_{t_{n}}^{c}}^{2} + \norm{\hatY_{t_{n}}^{f}}^{2} \Big)^{p/4}  \Bigg].
            \end{split}
        \end{equation*}
        Using Young's inequality, and Jensen's inequality, for any $\zeta > 0$, we can write
        \begin{equation*}
            \begin{split}
                J_{4} \leq \frac{1}{4 \zeta} \E \Bigg[ \sup_{0 \leq n \leq N} \e^{\gamma p t_{n}} \Big( \norm{\hatY_{t_{n}}^{c}}^{2} + \norm{\hatY_{t_{n}}^{f}}^{2} \Big)^{p/2}  \Bigg] + \zeta \Big( \frac{6 C_{\text{BDG}}}{\gamma} \Big)^{p/2} p^{p/2} \e^{\gamma p T}.
            \end{split}
        \end{equation*}
        From Young's inequality we have
        \begin{equation*}
            \norm{\varphi_{2, 2n}}^{2} \leq 8 d K_{b} \Big( \norm{\hatY_{2n}^{f}}^{2} + \norm{\hatY_{2n}^{c}}^{2} \Big), \qquad \norm{\varphi_{2, 2n + 1}}^{2} \leq 4 d K_{b} \Big( \norm{\hatY_{2n}^{f}}^{2} + \norm{\hatY_{2n}^{c}}^{2} \Big).
        \end{equation*}
        Using the above inequality, and Burkholder--Davis--Gundy's inequality, we can bound $J_{5}$ as 
        \begin{equation*}
            \begin{split}
                J_{5} &\coloneqq \E \Bigg[ \sup_{0 \leq t \leq T} \abs{\int_{0}^{t} \e^{2 \gamma \floor{s / h} h} \inner{\varphi_{2, \floor{s / h}}, \di Z_{s}} }^{p / 2} \Bigg] \\
                ~
                &\leq (4 d K_{b} C_{\text{BDG}} p)^{p/4} \E \Bigg[ \abs{\sum_{k = 0}^{N - 1} \e^{4 \gamma t_{k}} h^{3} \norm{\varphi_{2, k}}^{2}}^{p/4} \Bigg] \\
                ~
                &\leq (4 d K_{b} C_{\text{BDG}} p)^{p/4} \abs{\sum_{k = 0}^{N - 1} \e^{2 \gamma t_{k}} h}^{p/4 - 1} \E \Bigg[ \sum_{k = 0}^{N - 1} \e^{2 \gamma t_{k}} h^{1 + \frac{p}{2}} \e^{\frac{\gamma p t_{k}}{2}} \Big( \norm{\hatY_{t_{k}}^{c}}^{2} + \norm{\hatY_{t_{k}}^{f}}^{2} \Big)^{\frac{p}{4}} \Bigg] \\
                ~
                &\leq (4 d K_{b} C_{\text{BDG}} p)^{p/4} \Bigl( \frac{\e^{2 \gamma T}}{2 \gamma} \Bigr)^{p/4 - 1} \E \Bigg[ \sum_{k = 0}^{N - 1} \e^{2 \gamma t_{k}} h^{1 + p/2} \cdot \sup_{0 \leq n \leq N} \e^{\gamma p t_{n} / 2} \Big( \norm{\hatY_{t_{n}}^{c}}^{2} + \norm{\hatY_{t_{n}}^{f}}^{2} \Big)^{p/4}  \Bigg].
            \end{split}
        \end{equation*}
        Using Young's inequality, and Jensen's inequality, for any $\zeta > 0$, we have
        \begin{equation*}
            \begin{split}
                J_{5} \leq \frac{1}{4 \zeta} \E \Bigg[ \sup_{0 \leq n \leq N} \e^{\gamma p t_{n}} \Big( \norm{\hatY_{t_{n}}^{c}}^{2} + \norm{\hatY_{t_{n}}^{f}}^{2} \Big)^{p/2}  \Bigg] + \zeta \bigg( \frac{2 d K_{b} C_{\text{BDG}}}{\gamma} \bigg)^{p/2} p^{p/2} h^{p} \e^{\gamma p T}.
            \end{split}
        \end{equation*}
        Combining all the terms together, choosing $\zeta = \widetilde{C}$, using Jensen's inequality, there exists a constant $\kappa_{2} > 0$ independent of $T$ and $p$, for all $T > 0$ and $p \geq 1$, such that
        \begin{equation*}
            \begin{split}
                \E \Bigg[ \sup_{0 \leq n \leq N} \e^{\gamma p t_{n}} \Bigg( \norm{\hatY_{t_{n}}^{f}}^{2} + \norm{\hatY_{t_{n}}^{c}}^{2} \Bigg)^{p/2} \Bigg] &\leq \kappa_{2}^{p} p^{p/2} \e^{\gamma p T}.
            \end{split}
        \end{equation*}
        Since the above inequality holds for all $T > 0$, there exists a constant $\kappa_{1} > 0$ such that for any $0 < h < \kappa_{1}$,
        \begin{equation}
            \sup_{0 \leq n \leq N} \E \Big[ \norm{\hatY_{t_{n}}^{f}}^{p} \Big] \leq \kappa_{2}^{p} p^{p/2}, \qquad \sup_{0 \leq n \leq N} \E \Big[ \norm{\hatY_{t_{n}}^{c}}^{p} \Big] \leq \kappa_{2}^{p} p^{p/2}.
        \end{equation}

    \subsubsection*{Proof of Theorem~\ref{thm:stability_wo_spring}}
        We obtain the bound only for the fine trajectory to keep the argument brief. The proof argument can be extended similarly for the coarse trajectory. We also introduce the constants $C, \tildealpha, \tildebeta > 0$ that depend only on $d$, and the constants from Assumptions~\ref{asmp:lipschitz_a}, \ref{asmp:aD_lipschitz}, and \ref{asmp:dissipativity_a}. The constants $C, \tildealpha, \tildebeta$ may change from line to line in the proof argument.

        From equations~\eqref{eq:fine_odd_wo_spring}, and \eqref{eq:fine_even_wo_spring}, computing the square of the Euclidean norm of the fine trajectory, we obtain
        \begin{equation}
            \label{eq:barX_fine_norm_squared}
            \begin{split}
                \norm{\barX_{n + 1}^{f}}^{2} &= \norm{X_{n}}^{2} + \norm{\hatI_{n}}^{2} + 2 \inner{\barX_{n}^{f}, \hatI_{n}},
            \end{split}
        \end{equation}
        where \[ \hatI_{n} \coloneqq ha(\barX_{n}^{f}) + \Delta W_{n} + Da(\barX_{n}^{f}) \Delta Z_{n} + \frac{h^{2}}{2} \cA a(\barX_{n}^{f}). \]
        Using H\"older's inequality and Assumptions~\ref{asmp:lipschitz_a}, and \ref{asmp:aD_lipschitz}, the term $\norm{\hatI_{n}}^{2}$ yields the bound
        \begin{equation}
            \label{eq:barX_In_norm_term}
            \begin{split}
                \norm{\hatI_{n}}^{2} &\leq 4 \bigg( h^{2} \norm{a(\barX_{n}^{f})}^{2} + \norm{\Delta W_{n}}^{2} + \norm{Da(\barX_{n}^{f}) \Delta Z_{n}}^{2} + \frac{h^{4}}{4} \norm{\cA a(\barX_{n}^{f})}^{2} \bigg) \\
                ~
                &\leq C \bigg( h^{2} \big( 1 + \norm{\barX_{n}^{f}}^{2} \big) + \norm{\Delta W_{n}}^{2} + \norm{\Delta Z_{n}}^{2} \bigg).
            \end{split}
        \end{equation}
        Using Assumption~\ref{asmp:dissipativity_a}, and Young's inequality, the term $\inner{\barX_{n}^{f}, \hatI_{n}}$ yields
        \begin{equation}
            \label{eq:barX_inner_prod_term}
            \begin{split}
                \inner{\barX_{n}^{f}, \hatI_{n}} &= h \inner{\barX_{n}^{f}, a(\barX_{n}^{f})} + \inner{\barX_{n}^{f}, \Delta W_{n}} + \inner{\barX_{n}^{f}, Da(\barX_{n}^{f}) \Delta Z_{n}} + \frac{h^{2}}{2} \inner{\barX_{n}^{f}, \cA a(\barX_{n}^{f})} \\
                ~
                &\leq -\alpha h \norm{\barX_{n}^{f}}^{2} + \beta h + \inner{\barX_{n}^{f}, \Delta W_{n}} + \inner{\barX_{n}^{f}, Da(\barX_{n}^{f}) \Delta Z_{n}} + C h^{2} \Big( 1 + \norm{\barX_{n}^{f}}^{2} \Big) \\
                ~
                &\leq -\tildealpha h \norm{\barX_{n}^{f}}^{2} + \inner{\barX_{n}^{f}, \Delta W_{n}} + \inner{\barX_{n}^{f}, Da(\barX_{n}^{f}) \Delta Z_{n}} + \tildebeta h.
            \end{split}
        \end{equation}
        Substituting equations~\eqref{eq:barX_In_norm_term} and \eqref{eq:barX_inner_prod_term} in \eqref{eq:barX_fine_norm_squared}, we obtain
        \begin{equation*}
            \begin{split}
                \norm{\barX_{n + 1}^{f}}^{2} &\leq (1 - 2 \tildealpha h + C h^{2}) \norm{\barX_{n}^{f}}^{2} + C (\norm{\Delta W_{n}}^{2} + \norm{\Delta Z_{n}}^{2}) + 2 \inner{\barX_{n}^{f}, \Delta W_{n}} \\
                &\quad + 2 \inner{\barX_{n}^{f}, Da(\barX_{n}^{f}) \Delta Z_{n}} + 2 \tildebeta h.
            \end{split}
        \end{equation*}
        Note that for any $\gamma \in (0, \tildealpha - Ch/2)$, and for all sufficiently small $h > 0$, the following inequality holds
        \begin{equation*}
            \begin{split}
                \norm{\barX_{n + 1}^{f}}^{2} &\leq (1 - 2 \gamma h) \norm{\barX_{n}^{f}}^{2} + C (\norm{\Delta W_{n}}^{2} + \norm{\Delta Z_{n}}^{2}) + 2 \inner{\barX_{n}^{f}, \Delta W_{n}} \\
                &\quad + 2 \inner{\barX_{n}^{f}, Da(\barX_{n}^{f}) \Delta Z_{n}} + 2 \tildebeta h.
            \end{split}
        \end{equation*}
        Multiplying both sides of the above inequality by $\e^{2 \gamma t_{n + 1}}$, noting that $\e^{2 \gamma h} \leq 2$ and $(1 - 2 \gamma h) \leq \e^{-2 \gamma h}$, for all sufficiently small $h > 0$, we obtain
        \begin{equation*}
            \begin{split}
                \e^{2 \gamma t_{n + 1}} \norm{\barX_{n + 1}^{f}}^{2} &\leq \e^{2 \gamma t_{n}} \norm{\barX_{n}^{f}}^{2} + C \e^{2 \gamma t_{n}} (\norm{\Delta W_{n}}^{2} + \norm{\Delta Z_{n}}^{2}) + 4 \e^{2 \gamma t_{n}} \inner{\barX_{n}^{f}, \Delta W_{n}} \\
                &\quad + 4 \e^{2 \gamma t_{n}} \inner{\barX_{n}^{f}, Da(\barX_{n}^{f}) \Delta Z_{n}} + 4 \tildebeta \e^{2 \gamma t_{n}} h.
            \end{split}
        \end{equation*}
        Recursively summing the above inequality over $n = 0, 1, \ldots, N - 1$, taking the power $p/2$ on both sides of the inequality, and taking the supremum over all $n$, and taking the expected value on both sides of the inequality, we obtain
        \begin{equation*}
            \begin{split}
                \E \Bigg[ \sup_{0 \leq n \leq N - 1} \e^{\gamma p t_{n + 1}} \norm{\barX_{n + 1}^{f}}^{p} \Bigg] &\leq \widetilde{C} \Big( J_{1} + J_{2} + J_{3} + J_{4} + J_{5} + J_{6} \Big),
            \end{split}
        \end{equation*}
        where $\widetilde{C} \coloneqq \max(6^{p/2 - 1} C, \; 1)$, and the terms $J_{1}$ through $J_{6}$ are defined and bounded below.

        The terms $J_{1}$ and $J_{6}$ yield the bounds
        \begin{equation*}
            \begin{split}
                J_{1} \coloneqq \norm{x_{0}}^{p} < \infty, \qquad J_{6} \coloneqq \abs{ \sum_{k = 0}^{N - 1} 4 \tildebeta \e^{2 \gamma t_{k}} h }^{p/2} \leq 2^{p} \tildebeta^{p/2} \frac{\e^{\gamma p T}}{(2 \gamma)^{p/2}}.
            \end{split}
        \end{equation*}
        The terms $J_{2}, J_{3}, J_{4}$, and $J_{5}$ can be bounded in a similar way as in Theorem~\ref{thm:stability} to yield the bounds
        \begin{equation*}
            \begin{split}
                J_{2} &\coloneqq \E \Bigg[ \abs{ \sum_{k = 0}^{N - 1} \e^{2 \gamma t_{k}} \norm{\Delta W_{k}}^{2} }^{p/2} \Bigg] \leq d^{p/2} p^{p/2} \frac{\e^{\gamma p T}}{(2 \gamma)^{p/2}}, \\
                ~
                J_{3} &\coloneqq \E \Bigg[ \abs{ \sum_{k = 0}^{N - 1} \e^{2 \gamma t_{k}} \norm{\Delta Z_{k}}^{2} }^{p/2} \Bigg] \leq 2 d^{p/2} p^{p/2} h^{p} \frac{\e^{\gamma p T}}{(2 \gamma)^{p/2}}, \\
                ~
                J_{4} &\coloneqq \E \Bigg[ \sup_{0 \leq t \leq T} \abs{\int_{0}^{t} \e^{2 \gamma \floor{s / h} h} \inner{\barX_{\floor{s / h}}^{f}, \di W_{s}}}^{p/2} \Bigg] \\
                &\leq \frac{1}{4 \zeta} \E \Bigg[ \sup_{0 \leq n \leq N} \e^{\gamma p t_{n}} \norm{\barX_{n}^{f}}^{p} \Bigg] + \zeta \bigg( \frac{C_{\text{BDG}}}{\gamma} \bigg)^{p/2} p^{p/2} \e^{\gamma p T}, \\
                ~
                J_{5} &\coloneqq \E \Bigg[ \sup_{0 \leq t \leq T} \abs{\int_{0}^{t} \e^{2 \gamma \floor{s / h} h} \inner{\barX_{\floor{s / h}}^{f}, Da(\barX_{\floor{s / h}}^{f}) \di Z_{s}}}^{p/2} \Bigg] \\
                &\leq \frac{1}{4 \zeta} \E \Bigg[ \sup_{0 \leq n \leq N} \e^{\gamma p t_{n}} \norm{\barX_{n}^{f}}^{p} \Bigg] + \zeta \bigg( \frac{d K_{b} C_{\text{BDG}}}{\gamma} \bigg)^{p/2} p^{p/2} h^{p} \e^{\gamma p T},
            \end{split}
        \end{equation*}
        where $C_{\text{BDG}} > 0$ is a constant independent of $p$.

        Combining all the terms together, choosing $\zeta = \widetilde{C}$, using Jensen's inequality, there exists a constant $C_{2} > 0$ independent of $T$ and $p$, for all $T > 0$ and $p \geq 1$, such that
        \begin{equation*}
            \begin{split}
                \E \Bigg[ \sup_{0 \leq n \leq N} \e^{\gamma p t_{n}} \norm{\barX_{t_{n}}^{f}}^{p} \Bigg] &\leq C_{2}^{p} p^{p/2} \e^{\gamma p T}.
            \end{split}
        \end{equation*}
        Since the above inequality holds for all $T > 0$, there exists a constant $C_{1} > 0$ such that for any $0 < h < C_{1}$,
        \begin{equation*}
            \sup_{0 \leq n \leq N} \E \Big[ \norm{\barX_{t_{n}}^{f}}^{p} \Big] \leq C_{2}^{p} p^{p/2}.
        \end{equation*}

    \subsubsection*{Proof of Theorem~\ref{thm:convergence}}
        Using equations~\eqref{eq:fine_odd} to \eqref{eq:coarse_even}, we can write the difference between the fine and the coarse trajectories at the even time-step as
        \begin{equation}
            \label{eq:fc_even_diff}
            \begin{split}
                \hatY_{2n + 2}^{f} - \hatY_{2n + 2}^{c} &= (1 - 4Sh + 2S^{2}h^{2}) (\hatY_{2n}^{f} - \hatY_{2n}^{c}) + (2 - Sh)h \big( a(\hatY_{2n}^{f}) - a(\hatY_{2n}^{c}) \big) \\
                &\quad + h \big( a(\hatY_{2n + 1}^{f}) - a(\hatY_{2n}^{f}) \big) + (1 - Sh) \big( Da(\hatY_{2n}^{f}) - Da(\hatY_{2n}^{c}) \big) \Delta Z_{2n} \\
                &\quad + \big( Da(\hatY_{2n + 1}^{f}) - Da(\hatY_{2n}^{f}) \big) \Delta Z_{2n + 1} + \big( Da(\hatY_{2n}^{f}) - Da(\hatY_{2n}^{c}) \big) \Delta Z_{2n + 1} \\
                &\quad - h Da(\hatY_{2n}^{c}) \Delta W_{2n} + \frac{h^{2}}{2} (2 - Sh) \big( \cA a(\hatY_{2n}^{f}) - \cA a(\hatY_{2n}^{c}) \big) \\
                &\quad + \frac{h^{2}}{2} \big( \cA a(\hatY_{2n + 1}^{f}) - \cA a(\hatY_{2n}^{f}) \big) - h^{2} \cA a(\hatY_{2n}^{c}).
            \end{split}
        \end{equation}
        As a first step, we aim to obtain a loose upper-bound on the difference between the fine and the coarse trajectories
        \begin{equation}
            \label{eq:method1_diff_fc_1}
            \begin{split}
                \hatY_{2n + 2}^{f} - \hatY_{2n + 2}^{c} &= (1 - 4Sh + 2S^{2}h^{2}) (\hatY_{2n}^{f} - \hatY_{2n}^{c}) + (2 - Sh)h \big( a(\hatY_{2n}^{f}) - a(\hatY_{2n}^{c}) \big) + \hatI_{2n + 1},
            \end{split}
        \end{equation}
        where
        \begin{equation*}
            \begin{split}
                \hatI_{2n + 1} &\coloneqq h \big( a(\hatY_{2n + 1}^{f}) - a(\hatY_{2n}^{f}) \big) + (1 - Sh) \big( Da(\hatY_{2n}^{f}) - Da(\hatY_{2n}^{c}) \big) \Delta Z_{2n} \\
                &\quad + \big( Da(\hatY_{2n + 1}^{f}) - Da(\hatY_{2n}^{f}) \big) \Delta Z_{2n + 1} + \big( Da(\hatY_{2n}^{f}) - Da(\hatY_{2n}^{c}) \big) \Delta Z_{2n + 1} \\
                &\quad - h Da(\hatY_{2n}^{c}) \Delta W_{2n} + \frac{h^{2}}{2} (2 - Sh) \big( \cA a(\hatY_{2n}^{f}) - \cA a(\hatY_{2n}^{c}) \big) \\
                &\quad + \frac{h^{2}}{2} \big( \cA a(\hatY_{2n + 1}^{f}) - \cA a(\hatY_{2n}^{f}) \big) - h^{2} \cA a(\hatY_{2n}^{c}).
            \end{split}
        \end{equation*}
        Taking the square of the Euclidean norm on both sides of equation~\eqref{eq:method1_diff_fc_1}, using the Cauchy--Schwarz inequality and the one-sided Lipschitz property of the drift coefficient, we can  for sufficiently small $h>0$ write
        \begin{equation*}
            \begin{split}
                \norm{\hatY_{2n + 2}^{f} - \hatY_{2n + 2}^{c}}^{2} &= \norm{(1 - 4Sh + 2S^{2}h^{2}) (\hatY_{2n}^{f} - \hatY_{2n}^{c}) + (2 - Sh)h \big( a(\hatY_{2n}^{f}) - a(\hatY_{2n}^{c}) \big) + \hatI_{2n + 1}}^{2} \\
                ~
                &\leq (1 - 4Sh + 2S^{2}h^{2})^{2} \norm{\hatY_{2n}^{f} - \hatY_{2n}^{c}}^{2} + 4 h^{2} \norm{a(\hatY_{2n}^{f}) - a(\hatY_{2n}^{c})}^{2} \\
                &\quad + \norm{\hatI_{2n + 1}}^{2} + 2h (1 - 4Sh + 2S^{2}h^{2}) (2 - Sh) \inner{\hatY_{2n}^{f} - \hatY_{2n}^{c}, a(\hatY_{2n}^{f}) - a(\hatY_{2n}^{c})} \\
                &\quad + 2 (1 + 2 \sqrt{d} K_{b} h) \norm{\hatY_{2n}^{f} - \hatY_{2n}^{c}}\norm{\hatI_{2n + 1}} \\
                ~
                &\leq (1 - 4 (2S - \lambda) h + Ch^{2}) \norm{\hatY_{2n}^{f} - \hatY_{2n}^{c}}^{2} + \norm{\hatI_{2n + 1}}^{2} \\
                &\quad + 2 (1 + 2 \sqrt{d} K_{b} h) \norm{\hatY_{2n}^{f} - \hatY_{2n}^{c}}\norm{\hatI_{2n + 1}},
            \end{split}
        \end{equation*}
        where $C > 0$ is a constant that depends only on $S, d$, and the constants from Assumptions~\ref{asmp:lipschitz_a}, \ref{asmp:aD_lipschitz} and  \ref{asmp:dissipativity_a}. The constant $C$ may change from line to line in the proof argument. From the above inequality, we observe that for all sufficiently small $h>0$, there exists a $\gamma \in (0, 2S - \lambda)$ such that the following holds
        \begin{equation*}
            \begin{split}
                \norm{\hatY_{2n + 2}^{f} - \hatY_{2n + 2}^{c}}^{2} &\leq (1 - 4 \gamma h) \norm{\hatY_{2n}^{f} - \hatY_{2n}^{c}}^{2} + 3 \norm{\hatY_{2n}^{f} - \hatY_{2n}^{c}}\norm{\hatI_{2n + 1}} + \norm{\hatI_{2n + 1}}^{2}.
            \end{split}
        \end{equation*}
        Further applying Young's inequality to
        \begin{equation*}
            \norm{\hatY_{2n}^{f} - \hatY_{2n}^{c}}\norm{\hatI_{2n + 1}} \leq \frac{2 \gamma h}{3} \norm{\hatY_{2n}^{f} - \hatY_{2n}^{c}}^{2} + \frac{3}{8 \gamma h} \norm{\hatI_{2n + 1}}^{2},
        \end{equation*}
        we obtain
        \begin{equation}
            \label{eq:method1_diff_fc_2}
            \begin{split}
                \norm{\hatY_{2n + 2}^{f} - \hatY_{2n + 2}^{c}}^{2} &\leq (1 - 2 \gamma h) \norm{\hatY_{2n}^{f} - \hatY_{2n}^{c}}^{2} + \frac{C}{h} \norm{\hatI_{2n + 1}}^{2}.
            \end{split}
        \end{equation}
        Using H\"older's inequality, and Assumptions~\ref{asmp:lipschitz_a}, \ref{asmp:dissipativity_a}, and \ref{asmp:aD_lipschitz}, the term $\hatI_{2n + 1}$ can be bounded as follows
        \begin{equation}
            \label{eq:hatI_2n_bound_1}
            \begin{split}
                \norm{\hatI_{2n + 1}}^{2} &\leq 8 C \Bigg\{ h^{2} \norm{\hatY_{2n + 1}^{f} - \hatY_{2n}^{f}}^{2} + h^{2} \norm{\Delta W_{2n}}^{2} + \norm{\Delta Z_{2n}}^{2} + \norm{\Delta Z_{2n + 1}}^{2} \\
                &\quad + h^{4} \Big( \norm{\hatY_{2n}^{f}}^{2} + \norm{\hatY_{2n}^{c}}^{2} \Big) + h^{4} \Bigg\}.
            \end{split}
        \end{equation}
        From equation~\eqref{eq:fine_odd}, we have
        \begin{equation}
            \begin{split}
                \label{eq:pth_power_fine_diff_intermediate}
                \norm{\hatY_{2n + 1}^{f} - \hatY_{2n}^{f}}^{p} &\leq 6^{p - 1} C \Bigg\{ h^{p} \bigg( 1 + \norm{\hatY_{2n}^{f}}^{p} + \norm{\hatY_{2n}^{c}}^{p} \bigg) + \norm{\Delta W_{2n}}^{p} + \norm{\Delta Z_{2n}}^{p} \Bigg\}.
            \end{split}
        \end{equation}
        Substituting equation~\eqref{eq:pth_power_fine_diff_intermediate} in \eqref{eq:hatI_2n_bound_1}, for $p = 2$, we obtain
        \begin{equation}
            \label{eq:hatI_2n_bound_2}
            \begin{split}
                \norm{\hatI_{2n + 1}}^{2} &\leq C \Bigg\{ h^{4} \bigg( 1 + \norm{\hatY_{2n}^{f}}^{2} + \norm{\hatY_{2n}^{c}}^{2} \bigg) +h^{2} \norm{\Delta W_{2n}}^{2} + \norm{\Delta Z_{2n}}^{2} + \norm{\Delta Z_{2n + 1}}^{2} \Bigg\}.
            \end{split}
        \end{equation}
        Substituting equation~\eqref{eq:hatI_2n_bound_2} in \eqref{eq:method1_diff_fc_2}, we obtain
        \begin{equation*}
            \begin{split}
                \norm{\hatY_{2n + 2}^{f} - \hatY_{2n + 2}^{c}}^{2} &\leq (1 - 2 \gamma h) \norm{\hatY_{2n}^{f} - \hatY_{2n}^{c}}^{2} + C \Bigg\{ h^{3} \bigg( \norm{\hatY_{2n}^{f}}^{2} + \norm{\hatY_{2n}^{c}}^{2} \bigg) + h \norm{\Delta W_{2n}}^{2} \\
                &\quad + \frac{1}{h} \Big( \norm{\Delta Z_{2n}}^{2} + \norm{\Delta Z_{2n + 1}}^{2} \Big) + h^{3} \Bigg\}.
            \end{split}
        \end{equation*}
        Multiplying both sides of the above inequality by $\e^{\gamma t_{2n}}$, noting that $1 - 2 \gamma h \leq e^{2 \gamma h} \leq 2$ holds for all sufficiently small $h > 0$, and from recursive summation of the inequality over multiple time-steps, we obtain
        \begin{equation*}
            \begin{split}
                \e^{\gamma t_{2n}} \norm{\hatY_{2n}^{f} - \hatY_{2n}^{c}}^{2} &\leq C \Bigg\{ h^{3} \sum_{k = 0}^{2n - 1} \e^{\gamma t_{k}} \bigg( \norm{\hatY_{k}^{f}}^{2} + \norm{\hatY_{k}^{c}}^{2} \bigg) + h \sum_{k = 0}^{2n - 1} \e^{\gamma t_{k}} \norm{\Delta W_{k}}^{2} \\
                &\quad + \frac{1}{h} \sum_{k = 0}^{2n - 1} \e^{\gamma t_{k}} \norm{\Delta Z_{k}}^{2} + \sum_{k = 0}^{2n - 1} \e^{\gamma t_{k}} h^{3} \Bigg\}.
            \end{split}
        \end{equation*}
        Raising both sides of the above inequality to the power $p/2$, taking the supremum over $n \in [0, N/2]$, and taking the expectation on both sides of the inequality, we obtain
        \begin{equation*}
            \begin{split}
                \E \Bigg[ \sup_{0 \leq n \leq N/2} \e^{\gamma p t_{2n} / 2} \norm{\hatY_{2n}^{f} - \hatY_{2n}^{c}}^{p} \Bigg] &\leq \widetilde{C} \Bigg\{ J_{1} + J_{2} + J_{3} + J_{4} \Bigg\},
            \end{split}
        \end{equation*}
        where $\widetilde{C} \coloneqq \max (4^{p/2 - 1} C^{p/2}, 1)$, and
        \begin{equation*}
            \begin{aligned}[c]
                J_{1} &= h^{3p/2} \E \Bigg[ \abs{\sum_{k = 0}^{N - 1} \e^{\gamma t_{k}} \Big( \norm{\hatY_{k}^{f}}^{2} + \norm{\hatY_{k}^{c}}^{2} \Big)}^{p/2} \Bigg], \\
                J_{3} &= h^{3p/2} \abs{\sum_{k = 0}^{N - 1} \e^{\gamma t_{k}} }^{p/2},
            \end{aligned}
            \qquad
            \begin{aligned}[c]
                J_{2} &= h^{p/2} \E \Bigg[ \abs{\sum_{k = 0}^{N - 1} \e^{\gamma t_{k}} \norm{\Delta W_{k}}^{2}}^{p/2} \Bigg], \\
                J_{4} &= h^{-p/2} \E \Bigg[ \abs{\sum_{k = 0}^{N - 1} \e^{\gamma t_{k}} \norm{\Delta Z_{k}}^{2}}^{p/2} \Bigg].
            \end{aligned}
        \end{equation*}
        To keep the argument concise, we introduce the constants $C_{J_{i}} > 0$ for $i = 1, 2, 3, 4$ that depend only on $S, d$, $\gamma$, and the constants from Assumptions~\ref{asmp:lipschitz_a}, \ref{asmp:aD_lipschitz} and \ref{asmp:dissipativity_a}. Using Theorem~\ref{thm:stability}, term $J_{1}$ yields
        \begin{equation*}
            \begin{split}
                J_{1} &= h^{p} \E \Bigg[ \abs{\sum_{k = 0}^{N - 1} \e^{\gamma t_{k}} h \Big( \norm{\hatY_{k}^{f}}^{2} + \norm{\hatY_{k}^{c}}^{2} \Big)}^{p/2} \Bigg] \\
                &\leq h^{p} \abs{\sum_{k = 0}^{N - 1} \e^{\gamma t_{k}} h}^{p/2 - 1} \E \Bigg[ \sum_{k = 0}^{N - 1} \e^{\gamma t_{k}} h \Big( \norm{\hatY_{k}^{f}}^{2} + \norm{\hatY_{k}^{c}}^{2} \Big)^{p/2} \Bigg] \\
                &\leq 2^{p/2 - 1} h^{p} \abs{\int_{0}^{T} \e^{\gamma t} \di t}^{p/2 - 1} \sum_{k = 0}^{N - 1} \e^{\gamma t_{k}} h \E \Big[ \norm{\hatY_{k}^{f}}^{p} + \norm{\hatY_{k}^{c}}^{p} \Big] \\
                &\leq C_{J_{1}}^{p} p^{p/2} h^{p} \e^{\gamma p T / 2}.
            \end{split}
        \end{equation*}
        The terms $J_{2}$, $J_{3}$, and $J_{4}$ can be bounded in a similar way as in Theorem~\ref{thm:stability} yielding the bounds
        \begin{equation*}
            J_{2} \leq C_{J_{2}}^{p} p^{p/2} h^{p/2} \e^{\gamma p T / 2}, \qquad
            J_{3} \leq C_{J_{3}}^{p} h^{p} \e^{\gamma p T / 2}, \qquad
            J_{4} \leq C_{J_{4}}^{p} p^{p/2} h^{p/2} \e^{\gamma p T / 2}.
        \end{equation*}
        Combining all the estimates together, there exists a constant $\widetilde{C}_{1} > 0$ that is independent of $T$ and $p$, such that
        \begin{equation*}
            \begin{split}
                \E \Big[ \sup_{0 \leq n \leq N/2} \e^{\gamma p t_{2n} / 2} \norm{\hatY_{2n}^{f} - \hatY_{2n}^{c}}^{p} \Big] &\leq \widetilde{C}_{(1)}^{p} p^{p/2} h^{p/2} \e^{\gamma p T / 2}.
            \end{split}
        \end{equation*}
        Since the above inequality holds for any arbitrary $T > 0$, there exists a constant $\kappa_{3} > 0$, such that for any $0 < h < \kappa_{3}$,
        \begin{equation}
            \label{eq:diff_bound_1}
            \sup_{0 \leq n \leq N/2} \E \Big[ \norm{\hatY_{2n}^{f} - \hatY_{2n}^{c}}^{p} \Big] \leq \widetilde{C}_{1}^{p} p^{p/2} h^{p/2}.
        \end{equation}
        Having obtained a loose upper bound on the difference between the fine and the coarse trajectories, we aim to sharpen the bound obtained in equation~\eqref{eq:diff_bound_1}. Taking the square of the Euclidean norm on both sides of equation~\eqref{eq:method1_diff_fc_1} and using the one-sided Lipschitz property of the drift coefficient, we can write
        \begin{equation*}
            \begin{split}
                \norm{\hatY_{2n + 2}^{f} - \hatY_{2n + 2}^{c}}^{2} &= \norm{(1 - 4Sh + 2S^{2}h^{2}) (\hatY_{2n}^{f} - \hatY_{2n}^{c}) + (2 - Sh)h \big( a(\hatY_{2n}^{f}) - a(\hatY_{2n}^{c}) \big) + \hatI_{2n + 1}}^{2} \\
                ~
                &\leq (1 - 4Sh + 2S^{2}h^{2})^{2} \norm{\hatY_{2n}^{f} - \hatY_{2n}^{c}}^{2} + 6 h^{2} \norm{a(\hatY_{2n}^{f}) - a(\hatY_{2n}^{c})}^{2} \\
                &\quad + 3 \norm{\hatI_{2n + 1}}^{2} + 2h (1 - 4Sh + 2S^{2}h^{2}) (2 - Sh) \inner{\hatY_{2n}^{f} - \hatY_{2n}^{c}, a(\hatY_{2n}^{f}) - a(\hatY_{2n}^{c})} \\
                &\quad + 2 (1 - 4Sh + 2S^{2}h^{2}) \inner{\hatY_{2n}^{f} - \hatY_{2n}^{c}, \hatI_{2n + 1}} \\
                ~
                &\leq (1 - 4 (2S - \lambda) h + Ch^{2}) \norm{\hatY_{2n}^{f} - \hatY_{2n}^{c}}^{2} + 3 \norm{\hatI_{2n + 1}}^{2} \\
                &\quad + 2 (1 - 4Sh + 2S^{2}h^{2}) \inner{\hatY_{2n}^{f} - \hatY_{2n}^{c}, \hatI_{2n + 1}}.
            \end{split}
        \end{equation*}
        As in the first approach, there exists a $\gamma \in (0, 2S - \lambda)$, for all sufficiently small $h > 0$, such that
        \begin{equation}
            \label{eq:method2_diff_fc_1}
            \begin{split}
                \norm{\hatY_{2n + 2}^{f} - \hatY_{2n + 2}^{c}}^{2} &\leq (1 - 4 \gamma h) \norm{\hatY_{2n}^{f} - \hatY_{2n}^{c}}^{2} + 3 \norm{\hatI_{2n + 1}}^{2} \\
                &\qquad + 2 (1 - 4Sh + 2S^{2}h^{2}) \inner{\hatY_{2n}^{f} - \hatY_{2n}^{c}, \hatI_{2n + 1}}.
            \end{split}
        \end{equation}
        Using the Mean value theorem, we can write
        \begin{equation}
            \label{eq:first_order_R1}
            \begin{split}
                a(\hatY_{2n + 1}^{f}) - a(\hatY_{2n}^{f}) &= Da(\hatY_{2n}^{f}) \Big( \hatY_{2n + 1}^{f} - \hatY_{2n}^{f} \Big) + R_{2},
            \end{split}
        \end{equation}
        where $R_{2}$ denotes the second-order remainder from the Mean value theorem. From Assumption~\ref{asmp:lipschitz_a}, it follows that $\norm{R_{2}} \leq d^{3/2} K_{b} \norm{\hatY_{2n + 1}^{f} - \hatY_{2n}^{f}}^{2}$.
        
        Substituting equation~\eqref{eq:first_order_R1} in \eqref{eq:method2_diff_fc_1}, using H\"older's inequality, Young's inequality, and equation~\eqref{eq:pth_power_fine_diff_intermediate}, we have
        \begin{equation*}
            \begin{split}
                \norm{\hatY_{2n + 2}^{f} - \hatY_{2n + 2}^{c}}^{2} &\leq (1 - 4 \gamma h) \norm{\hatY_{2n}^{f} - \hatY_{2n}^{c}}^{2} + C \Bigg\{ \Big( h^{2} \norm{\Delta W_{2n}}^{2} + \norm{\Delta Z_{2n}}^{2} \\
                &\quad + \norm{\Delta Z_{2n + 1}}^{2} \Big) \norm{\hatY_{2n}^{f} - \hatY_{2n}^{c}}^{2} + h^{2} \norm{\Delta Z_{2n}}^{2} + h^{4} \Big( 1 + \norm{\hatY_{2n}^{f}}^{2} \Big) \\
                &\quad + \Big( h^{4} + \norm{\Delta Z_{2n + 1}}^{2} \Big) \norm{\hatY_{2n + 1}^{f} - \hatY_{2n}^{f}}^{2} + h^{2} \norm{R_{2}}^{2} \Bigg\} \\
                &\quad
                + 2 (1 - 4Sh + 2S^{2}h^{2}) \Bigg\{ h \inner{\hatY_{2n}^{f} - \hatY_{2n}^{c}, R_{2}} \\
                &\quad + h \inner{\hatY_{2n}^{f} - \hatY_{2n}^{c}, \big( Da(\hatY_{2n}^{f}) - Da(\hatY_{2n}^{c}) \big) \Delta W_{2n}} \\
                &\quad + h \inner{\hatY_{2n}^{f} - \hatY_{2n}^{c}, Da(\hatY_{2n}^{f}) Da(\hatY_{2n}^{f}) \Delta Z_{2n}} \\
                &\quad + (1 - Sh) \inner{\hatY_{2n}^{f} - \hatY_{2n}^{c}, \big( Da(\hatY_{2n}^{f}) - Da(\hatY_{2n}^{c}) \big) \Delta Z_{2n}} \\
                &\quad + \inner{\hatY_{2n}^{f} - \hatY_{2n}^{c}, \big( Da(\hatY_{2n + 1}^{f}) - Da(\hatY_{2n}^{f}) \big) \Delta Z_{2n + 1}} \\
                &\quad + \inner{\hatY_{2n}^{f} - \hatY_{2n}^{c}, \big( Da(\hatY_{2n}^{f}) - Da(\hatY_{2n}^{c}) \big) \Delta Z_{2n + 1}} \\
                &\quad + \frac{h^{2}}{2} \inner{\hatY_{2n}^{f} - \hatY_{2n}^{c}, \cA a(\hatY_{2n + 1}^{f}) - \cA a(\hatY_{2n}^{f})} - \frac{h^{2}}{2} \inner{\hatY_{2n}^{f} - \hatY_{2n}^{c}, \Delta a(\hatY_{2n}^{f})} \Bigg\}.
            \end{split}
        \end{equation*}
        Multiplying both sides of the above inequality by $\e^{2 \gamma t_{2n}}$, noting that $\e^{4 \gamma h} \leq 2$ for all sufficiently small $h > 0$, summing over multiple time-steps, raising both sides of the inequality to the power $p/2$, taking supremum over $n \in [0, N/2]$, taking the expectation on both sides of the inequality, we obtain
        \begin{align*}
            \E \Big[ \sup_{0 \leq n \leq N/2} \e^{\gamma p t_{2n}} \norm{\hatY_{2n + 2}^{f} - \hatY_{2n + 2}^{c}}^{p} \Big] &\leq \widetilde{C} \sum_{m = 1}^{16} J_{m},
        \end{align*}
        where $\widetilde{C} \coloneqq \max ( 16^{p/2 - 1} C^{p/2}, 1 )$ and the terms $J_{m}$ for $m = 1, \ldots, 16$ are defined and bounded below.
        For the term $J_{1}$, we have
        \begin{equation*}
            \begin{split}
                J_{1} &\coloneqq \E \Bigg[ \abs{ \sum_{k = 0}^{N/2 - 1} \e^{2 \gamma t_{2k}} h^{2} \norm{\Delta W_{2k}}^{2} \norm{\hatY_{2k}^{f} - \hatY_{2k}^{c}}^{2} }^{p/2} \Bigg] \\
                &\leq \abs{ \sum_{k = 0}^{N/2 - 1} \e^{2 \gamma t_{2k}} h }^{p/2 - 1} \E \Bigg[ \sum_{k = 0}^{N/2 - 1} \e^{2 \gamma t_{2k}} h^{1 + p} \frac{\norm{\Delta W_{2k}}^{p}}{h^{p/2}} \norm{\hatY_{2k}^{f} - \hatY_{2k}^{c}}^{p} \Bigg] \\
                &\leq d^{p/2} p^{p/2} h^{p} \abs{\int_{0}^{T} \e^{2 \gamma t} \di t}^{p/2 - 1} \sum_{k = 0}^{N/2 - 1} \e^{2 \gamma t_{2k}} h \E \Big[ \norm{\hatY_{2k}^{f} - \hatY_{2k}^{c}}^{p} \Big].
            \end{split}
        \end{equation*}
        Using equation~\eqref{eq:diff_bound_1}, we can write
        \begin{equation*}
            \begin{split}
                J_{1} &\leq C_{J_{1}}^{p} p^{p} h^{3p/2} \e^{\gamma p T}.
            \end{split}
        \end{equation*}
        Similarly, the terms $J_{2}, J_{3}$, and $J_{4}$, yield the bounds
        \begin{equation*}
            \begin{split}
                J_{2} &\coloneqq \E \Bigg[ \abs{ \sum_{k = 0}^{N/2 - 1} \e^{2 \gamma t_{2k}} \norm{\Delta Z_{2k}}^{2} \norm{\hatY_{2k}^{f} - \hatY_{2k}^{c}}^{2} }^{p/2} \Bigg] \leq C_{J_{2}}^{p} p^{p} h^{3p/2} \e^{\gamma p T}, \\
                J_{3} &\coloneqq \E \Bigg[ \abs{ \sum_{k = 0}^{N/2 - 1} \e^{2 \gamma t_{2k}} \norm{\Delta Z_{2k + 1}}^{2} \norm{\hatY_{2k}^{f} - \hatY_{2k}^{c}}^{2} }^{p/2} \Bigg] \leq C_{J_{3}}^{p} p^{p} h^{3p/2} \e^{\gamma p T}, \\
                J_{4} &\coloneqq \E \Bigg[ \abs{ \sum_{k = 0}^{N/2 - 1} \e^{2 \gamma t_{2k}} h^{2} \norm{\Delta Z_{2k}}^{2} }^{p/2} \Bigg] \leq C_{J_{4}}^{p} p^{p/2} h^{2p} \e^{\gamma p T}.
            \end{split}
        \end{equation*}
        The term $J_{5}$, using the result from Theorem~\ref{thm:stability}, yields the bound
        \begin{equation*}
            \begin{split}
                J_{5} &\coloneqq \E \Bigg[ \abs{ \sum_{k =0}^{N/2 - 1} \e^{2 \gamma t_{2k}} h^{4} \Big( 1 + \norm{\hatY_{2k}^{f}}^{2} \Big) }^{p/2} \Bigg] \leq C_{J_{5}}^{p} p^{p/2} h^{3p/2} \e^{\gamma p T}.
            \end{split}
        \end{equation*}
        For the term $J_{6}$, using equation~\eqref{eq:pth_power_fine_diff}, we obtain
        \begin{equation*}
            \begin{split}
                J_{6} &\coloneqq h^{3p/2} \E \Bigg[ \abs{ \sum_{k =0}^{N/2 - 1} \e^{2 \gamma t_{2k}} h \norm{\hatY_{2k + 1}^{f} - \hatY_{2k}^{f}}^{2} }^{p/2} \Bigg] \leq C_{J_{6}}^{p} p^{p/2} h^{2p} \e^{\gamma p T},
            \end{split}
        \end{equation*}
        where, using the result from Theorem~\ref{thm:stability}, we have
        \begin{equation}
            \label{eq:pth_power_fine_diff}
            \E \bigg[ \norm{\hatY_{2n + 1}^{f} - \hatY_{2n}^{f}}^{p} \bigg] \leq C_{f}^{p} p^{p/2} h^{p/2},
        \end{equation}
        for all $p \geq 1$, where $C_{f} > 0$ is a constant that independent of $T$ and $p$.
        Similarly, using equation~\eqref{eq:first_order_R1} we have
        \begin{equation*}
            \begin{split}
                J_{7} &\coloneqq \E \Bigg[ \abs{ \sum_{k =0}^{N/2 - 1} \e^{2 \gamma t_{2k}} \norm{\Delta Z_{2k + 1}}^{2} \norm{\hatY_{2k + 1}^{f} - \hatY_{2k}^{f}}^{2} }^{p/2} \Bigg] \leq C_{J_{7}}^{p} p^{p} h^{3p/2} \e^{\gamma p T}, \\
                ~
                J_{8} &\coloneqq \E \Bigg[ \abs{ \sum_{k = 0}^{N/2 - 1} \e^{2 \gamma t_{2k}} h^{2} \norm{R_{2}}^{2} }^{p/2} \Bigg] \leq C_{J_{8}}^{p} p^{p} h^{3p/2} \e^{\gamma p T}.
            \end{split}
        \end{equation*}
        For the term $J_{9}$, using Cauchy--Schwarz's inequality, Jensen's inequality, and Young's inequality for any $\zeta > 0$, we have
        \begin{equation*}
            \begin{split}
                J_{9} &\coloneqq \E \Bigg[ \sup_{0 \leq n \leq N/2 - 1} \abs{ \sum_{k = 0}^{n} \e^{2 \gamma t_{2k}} h \inner{\hatY_{2k}^{f} - \hatY_{2k}^{c}, R_{2}} }^{p/2} \Bigg] \\
                &\leq \E \Bigg[ \abs{ \sup_{0 \leq n \leq N/2} \e^{\gamma t_{2n}} \norm{\hatY_{2n}^{f} - \hatY_{2n}^{c}} }^{p/2} \abs{ \sum_{k = 0}^{N/2 - 1} \e^{\gamma t_{2k}} h \norm{\hatY_{2k + 1}^{f} - \hatY_{2k}^{f}}^{2} }^{p/2} \Bigg] \\
                &\leq \frac{1}{8 \zeta} \E \Bigg[ \sup_{0 \leq n \leq N/2} \e^{\gamma p t_{2n}} \norm{\hatY_{2n}^{f} - \hatY_{2n}^{c}}^{p} \Bigg] + 2 \zeta \E \Bigg[ \abs{ \sum_{k = 0}^{N/2 - 1} \e^{\gamma t_{2k}} h \norm{\hatY_{2k + 1}^{f} - \hatY_{2k}^{f}}^{2} }^{p} \Bigg] \\
                &\leq \frac{1}{8 \zeta} \E \Bigg[ \sup_{0 \leq n \leq N/2} \e^{\gamma p t_{2n}} \norm{\hatY_{2n}^{f} - \hatY_{2n}^{c}}^{p} \Bigg] + \zeta C_{J_{9}}^{p} p^{p} h^{p} \e^{\gamma p T}.
            \end{split}
        \end{equation*}
        For the term $J_{10}$, using the Burkholder--Davis--Gundy inequality as stated in \cite{BARLOW1982198}, and equation~\eqref{eq:diff_bound_1}, we have
        \begin{equation*}
            \begin{split}
                J_{10} &\coloneqq \E \Bigg[ \sup_{0 \leq n \leq N/2 - 1} \abs{ \sum_{k = 0}^{n} \e^{2 \gamma t_{2k}} h \inner{\hatY_{2k}^{f} - \hatY_{2k}^{c}, \big( Da(\hatY_{2k}^{f}) - Da(\hatY_{2k}^{c}) \big) \Delta W_{2k}} }^{p/2} \Bigg] \\
                &= \E \Bigg[ \sup_{0 \leq t \leq T} \abs{ \int_{0}^{t} \e^{2 \gamma \floor{s / 2 h} h} h \inner{\hatY_{(2 \floor{s / 2h})}^{f} - \hatY_{(2 \floor{s / 2h})}^{c}, \big( Da(\hatY_{(2 \floor{s / 2h})}^{f}) - Da(\hatY_{(2 \floor{s / 2h})}^{c}) \big) \di W_{s}} }^{p/2} \Bigg] \\
                &\leq (C_{\text{BDG}} p)^{p/4} d^{3p/4} K_{b}^{p/2} \E \Bigg[ \abs{ \sum_{k = 0}^{N/2 - 1} \e^{4 \gamma t_{2k}} h^{3} \norm{\hatY_{2k}^{f} - \hatY_{2k}^{c}}^{4} }^{p/4} \Bigg] \\
                &\leq (C_{\text{BDG}} p)^{p/4} d^{3p/4} K_{b}^{p/2} \abs{ \sum_{k = 0}^{N/2 - 1} \e^{4 \gamma t_{2k}} h }^{p/4 - 1} \sum_{k = 0}^{N/2 - 1} \e^{4 \gamma t_{2k}} h^{1 + p/2} \E \Bigg[ \norm{\hatY_{2k}^{f} - \hatY_{2k}^{c}}^{p} \Bigg] \\
                &\leq C_{J_{10}}^{p} p^{3p/4} h^{p} \e^{\gamma p T}.
            \end{split}
        \end{equation*}
        The terms $J_{12}$ and $J_{14}$ can be bounded in a similar way as the term $J_{10}$ to yield the bounds
        \begin{equation*}
            \begin{split}
                J_{12} &\coloneqq \E \Bigg[ \sup_{0 \leq n \leq N/2 - 1} \abs{ \sum_{k = 0}^{n} \e^{2 \gamma t_{2k}} \inner{\hatY_{2k}^{f} - \hatY_{2k}^{c}, \big( Da(\hatY_{2k}^{f}) - Da(\hatY_{2k}^{c}) \big) \Delta Z_{2k}} }^{p/2} \Bigg] \\
                &= \E \Bigg[ \sup_{0 \leq t \leq T} \abs{ \int_{0}^{t} \e^{2 \gamma \floor{s / 2h} h} \inner{\hatY_{(2 \floor{s / 2h})}^{f} - \hatY_{(2 \floor{s / 2h})}^{c}, \big( Da(\hatY_{(2 \floor{s / 2h})}^{f}) - Da(\hatY_{(2 \floor{s / 2h})}^{c}) \big) \di Z_{s}} }^{p/2} \Bigg] \\
                &\leq C_{J_{12}}^{p} p^{3p/4} h^{p} \e^{\gamma p T}, \\
                ~
                J_{14} &\coloneqq \E \Bigg[ \sup_{0 \leq n \leq N/2 - 1} \abs{ \sum_{k = 0}^{n} \e^{2 \gamma t_{2k}} \inner{\hatY_{2k}^{f} - \hatY_{2k}^{c}, \big( Da(\hatY_{2k}^{f}) - Da(\hatY_{2k}^{c}) \big) \Delta Z_{2k + 1}} }^{p/2} \Bigg] \\
                &\leq C_{J_{14}}^{p} p^{3p/4} h^{p} \e^{\gamma p T}.
            \end{split}
        \end{equation*}
        The term $J_{11}$, using the Burkholder--Davis--Gundy inequality, Jensen's inquality, and Young's inquality for any $\zeta > 0$, can be bounded as
        \begin{equation*}
            \begin{split}
                J_{11} &\coloneqq \E \Bigg[ \sup_{0 \leq n \leq N/2 - 1} \abs{\sum_{k = 0}^{n} \e^{2 \gamma t_{2k}} h \inner{\hatY_{2k}^{f} - \hatY_{2k}^{c}, Da(\hatY_{2k}^{f}) Da(\hatY_{2k}^{f}) \Delta Z_{2k}}}^{p/2} \Bigg] \\
                &= \E \Bigg[ \sup_{0 \leq t \leq T} \abs{ \int_{0}^{t} \e^{2 \gamma \floor{s / 2h} h} h \inner{\hatY_{(2 \floor{s / 2h})}^{f} - \hatY_{(2 \floor{s / 2h})}^{c}, Da(\hatY_{(2 \floor{s / 2h})}^{f}) Da(\hatY_{(2 \floor{s / 2h})}^{f}) \di Z_{s}} }^{p/2} \Bigg] \\
                &\leq \frac{1}{8 \zeta} \E \Bigg[ \sup_{0 \leq n \leq N/2} \norm{\hatY_{2n}^{f} - \hatY_{2n}^{c}}^{p} \Bigg] + 2 \zeta (C_{\text{BDG}} p)^{p/2} d^{2p} K_{b}^{2p} \E \Bigg[ \abs{\sum_{k = 0}^{N/2 - 1} \e^{2 \gamma t_{2k}} h^{5}}^{p/2} \Bigg] \\
                &\leq \frac{1}{8 \zeta} \E \Bigg[ \sup_{0 \leq n \leq N/2} \norm{\hatY_{2n}^{f} - \hatY_{2n}^{c}}^{p} \Bigg] + \zeta C_{J_{11}}^{p} p^{p/2} h^{2p} \e^{\gamma p T}.
            \end{split}
        \end{equation*}
        The term $J_{13}$ yields the bound
        \begin{equation*}
            \begin{split}
                J_{13} &\coloneqq \E \Bigg[ \sup_{0 \leq n \leq N/2 - 1} \abs{\sum_{k = 0}^{n} \e^{2 \gamma t_{2k}} h \inner{\hatY_{2k}^{f} - \hatY_{2k}^{c}, (Da(\hatY_{2k + 1}^{f}) - Da(\hatY_{2k}^{f})) \Delta Z_{2k + 1}}}^{p/2} \Bigg] \\
                &\leq (C_{\text{BDG}} p)^{p/4} d^{3p/4} K_{b}^{p/2} \E \Bigg[ \abs{\sum_{k = 0}^{N/2 - 1} \e^{4 \gamma t_{2k}} h^{5} \norm{\hatY_{2k}^{f} - \hatY_{2k}^{c}}^{2} \norm{\hatY_{2k + 1}^{f} - \hatY_{2k}^{f}}^{2}}^{p/4} \Bigg] \\
                &\leq C_{J_{13}}^{p} p^{3p/4} h^{3p/2} \e^{\gamma p T}.
            \end{split}
        \end{equation*}
        The term $J_{15}$, using the Cauchy--Schwarz inequality, Jensen's inequality, and Young's inequality for any $\zeta > 0$, can be bounded as
        \begin{equation*}
            \begin{split}
                J_{15} &\coloneqq \E \Bigg[ \abs{ \sum_{k = 0}^{N/2 - 1} \e^{2 \gamma t_{2k}} h^{2} \norm{\hatY_{2k}^{f} - \hatY_{2k}^{c}} \norm{\cA a(\hatY_{2k + 1}^{f}) - \cA a(\hatY_{2k}^{f})} }^{p/2} \Bigg] \\
                &\leq \hatL^{p/2} \E \Bigg[ \abs{\sup_{0 \leq n \leq N/2} \e^{\gamma t_{2n}} \norm{\hatY_{2n}^{f} - \hatY_{2n}^{c}}}^{p/2} \abs{\sum_{k = 0}^{N/2 - 1} \e^{\gamma t_{2k}} h^{2} \norm{\hatY_{2k + 1}^{f} - \hatY_{2k}^{f}} }^{p/2} \Bigg] \\
                &\leq \frac{1}{8 \zeta} \E \Bigg[ \sup_{0 \leq n \leq N/2} \e^{\gamma p t_{2n}} \norm{\hatY_{2n}^{f} - \hatY_{2n}^{c}}^{p} \Bigg] + 2 \zeta \hatL^{p} \E \Bigg[ \abs{\sum_{k = 0}^{N/2 - 1} \e^{\gamma t_{2k}} h^{2} \norm{\hatY_{2k + 1}^{f} - \hatY_{2k}^{f}} }^{p} \Bigg] \\
                &\leq \frac{1}{8 \zeta} \E \Bigg[ \sup_{0 \leq n \leq N/2} \e^{\gamma p t_{2n}} \norm{\hatY_{2n}^{f} - \hatY_{2n}^{c}}^{p} \Bigg] + \zeta C_{J_{15}}^{p} p^{p/2} h^{3p/2} \e^{\gamma p T}.
            \end{split}
        \end{equation*}
        The term $J_{16}$, using the Cauchy--Schwarz inequality, Jensen's inequality, and Young's inequality for any $\zeta > 0$, yields the bound
        \begin{equation*}
            \begin{split}
                J_{16} &\coloneqq \E \Bigg[ \sup_{0 \leq n \leq N/2 - 1} \abs{ \sum_{k = 0}^{n} \e^{2 \gamma t_{2k}} h^{2} \inner{\hatY_{2k}^{f} - \hatY_{2k}^{c}, \Delta a(\hatY_{2k}^{f})} }^{p/2} \Bigg] \\
                &\leq d^{3p/4} K_{b}^{p/2} \E \Bigg[ \abs{\sup_{0 \leq n \leq N/2} \e^{\gamma t_{2n}} \norm{\hatY_{2n}^{f} - \hatY_{2n}^{c}}}^{p/2} \abs{ \sum_{k = 0}^{N/2 - 1} \e^{\gamma t_{2k}} h^{2} }^{p/2} \Bigg] \\
                &\leq \frac{1}{8 \zeta} \E \Bigg[ \sup_{0 \leq n \leq N/2} \e^{\gamma p t_{2n}} \norm{\hatY_{2n}^{f} - \hatY_{2n}^{c}}^{p} \Bigg] + 2 \zeta d^{3p/2} K_{b}^{p} \E \Bigg[ \abs{ \sum_{k = 0}^{N/2 - 1} \e^{\gamma t_{2k}} h^{2} }^{p} \Bigg] \\
                &\leq \frac{1}{8 \zeta} \E \Bigg[ \sup_{0 \leq n \leq N/2} \e^{\gamma p t_{2n}} \norm{\hatY_{2n}^{f} - \hatY_{2n}^{c}}^{p} \Bigg] + \zeta C_{J_{16}}^{p} h^{p} \e^{\gamma p T}.
            \end{split}
        \end{equation*}
        Combining all the above estimates, choosing $\zeta = \widetilde{C}$ there exists a constant $\widetilde{C}_{2} > 0$ that is independent of $T$ and $p$, such that
        \begin{equation*}
            \begin{split}
                \E \Bigg[ \sup_{0 \leq n \leq N/2} \e^{\gamma p t_{2n}} \Bigg( \norm{\hatY_{2n}^{f} - \hatY_{2n}^{c}}^{p} \Bigg) \Bigg] \leq \widetilde{C}_{2}^{p} p^{p} h^{p} \e^{\gamma p T}.
            \end{split}
        \end{equation*}
        Since the above inequality holds for all $T > 0$, we have for all $0 < h < \kappa_{3}$,
        \begin{equation}
            \label{eq:diff_bound_2}
            \sup_{0 \leq n \leq N/2} \E \Big[ \norm{\hatY_{2n}^{f} - \hatY_{2n}^{c}}^{p} \Big] \leq \widetilde{C}_{2}^{p} p^{p} h^{p}.
        \end{equation}
        Note that since we have used equation~\eqref{eq:diff_bound_1} to obtain equation~\eqref{eq:diff_bound_2}. The bound obtained in equation~\eqref{eq:diff_bound_2} can be further sharpened through \textit{argumentum in recursum}.
        
        Using the Mean value theorem, we can write
        \begin{equation}
            \label{eq:second_order_R2}
            a(\hatY_{2n + 1}^{f}) - a(\hatY_{2n}^{f}) = Da(\hatY_{2n}^{f}) \Big( \hatY_{2n + 1}^{f} - \hatY_{2n}^{f} \Big) + \frac{1}{2} D^{2}a(\hatY_{2n}^{f}) \Big( \hatY_{2n + 1}^{f} - \hatY_{2n}^{f} \Big)^{2} + R_{3}
        \end{equation}
        where $\norm{R_{3}} \leq d^{2} K_{b} \norm{\hatY_{2n + 1}^{f} - \hatY_{2n}^{f}}^{3}$. Substituting equation~\eqref{eq:second_order_R2} in \eqref{eq:method2_diff_fc_1}, using H\"older's inequality, Young's inequality, and equation~\eqref{eq:pth_power_fine_diff_intermediate}, we obtain
        \begin{align*}
            \norm{\hatY_{2n + 2}^{f} - \hatY_{2n + 2}^{c}}^{2} &\leq (1 - 4 \gamma h) \norm{\hatY_{2n}^{f} - \hatY_{2n}^{c}}^{2} + C \Bigg\{ \Big( h^{2} \norm{\Delta W_{2n}}^{2} + \norm{\Delta Z_{2n}}^{2} \\
            &\quad + \norm{\Delta Z_{2n + 1}}^{2} \Big) \norm{\hatY_{2n}^{f} - \hatY_{2n}^{c}}^{2} + h^{2} \norm{\Delta Z_{2n}}^{2} + h^{4} \Big( 1 + \norm{\hatY_{2n}^{f}}^{2} \Big) \\
            &\quad + \Big( h^{4} + \norm{\Delta Z_{2n + 1}}^{2} \Big) \norm{\hatY_{2n + 1}^{f} - \hatY_{2n}^{f}}^{2} + h^{2} \norm{R_{3}}^{2} \Bigg\} \\
            &\quad
            + 2 (1 - 4Sh + 2S^{2}h^{2}) \Bigg\{ h \inner{\hatY_{2n}^{f} - \hatY_{2n}^{c}, R_{3}} \\
            &\quad + h \inner{\hatY_{2n}^{f} - \hatY_{2n}^{c}, \big( Da(\hatY_{2n}^{f}) - Da(\hatY_{2n}^{c}) \big) \Delta W_{2n}} \\
            &\quad + h \inner{\hatY_{2n}^{f} - \hatY_{2n}^{c}, Da(\hatY_{2n}^{f}) Da(\hatY_{2n}^{f}) \Delta Z_{2n}} \\
            &\quad + (1 - Sh) \inner{\hatY_{2n}^{f} - \hatY_{2n}^{c}, \big( Da(\hatY_{2n}^{f}) - Da(\hatY_{2n}^{c}) \big) \Delta Z_{2n}} \\
            &\quad + \inner{\hatY_{2n}^{f} - \hatY_{2n}^{c}, \big( Da(\hatY_{2n + 1}^{f}) - Da(\hatY_{2n}^{f}) \big) \Delta Z_{2n + 1}} \\
            &\quad + \inner{\hatY_{2n}^{f} - \hatY_{2n}^{c}, \big( Da(\hatY_{2n}^{f}) - Da(\hatY_{2n}^{c}) \big) \Delta Z_{2n + 1}} \\
            &\quad + \frac{h^{2}}{2} \inner{\hatY_{2n}^{f} - \hatY_{2n}^{c}, \cA a(\hatY_{2n + 1}^{f}) - \cA a(\hatY_{2n}^{f})} + \\
            &\quad + \frac{h}{2} \inner{\hatY_{2n}^{f} - \hatY_{2n}^{c}, D^{2}a(\hatY_{2n}^{f}) \big( \hatY_{2n + 1}^{f} - \hatY_{2n}^{f} \big)^{2} - h \Delta a(\hatY_{2n}^{f})} \Bigg\}.
        \end{align*}
        Observe that the order of convergence of the terms $J_{9}, J_{10}, J_{12}, J_{14}$, and $J_{16}$ can be improved using equation~\eqref{eq:diff_bound_2} and \eqref{eq:second_order_R2}.
        The term $J_{9}$ using the Cauchy--Schwarz inequality, and Young's inequality for any $\zeta > 0$, then yields the bound
        \begin{equation*}
            \begin{split}
                J_{9} &\coloneqq \E \Bigg[ \sup_{0 \leq n \leq N/2 - 1} \abs{ \sum_{k = 0}^{n} \e^{2 \gamma t_{2k}} h \inner{\hatY_{2k}^{f} - \hatY_{2k}^{c}, R_{3}} }^{p/2} \Bigg] \\
                &\leq \E \Bigg[ \abs{ \sup_{0 \leq n \leq N/2} \e^{\gamma t_{2n}} \norm{\hatY_{2n}^{f} - \hatY_{2n}^{c}} }^{p/2} \abs{ \sum_{k = 0}^{N/2 - 1} \e^{\gamma t_{2k}} h \norm{\hatY_{2k + 1}^{f} - \hatY_{2k}^{f}}^{3} }^{p/2} \Bigg] \\
                &\leq \frac{1}{8 \zeta} \E \Bigg[ \sup_{0 \leq n \leq N/2} \e^{\gamma p t_{2n}} \norm{\hatY_{2n}^{f} - \hatY_{2n}^{c}}^{p} \Bigg] + \zeta C_{J_{9}}^{p} p^{3p/2} h^{3p/2} \e^{\gamma p T}.
            \end{split}
        \end{equation*}
        Using equation~\eqref{eq:diff_bound_2}, the bounds for the terms $J_{10}, J_{12}$ and $J_{14}$ can be improved to
        \begin{equation*}
            J_{10} \leq C_{J_{10}}^{p} p^{5p/4} h^{3p/2} \e^{\gamma p T}, \qquad J_{12} \leq C_{J_{12}}^{p} p^{5p/4} h^{3p/2} \e^{\gamma p T}, \quad J_{14} \leq C_{J_{14}}^{p} p^{5p/4} h^{3p/2} \e^{\gamma p T}.
        \end{equation*}
        The term $J_{16}$, using equation~\eqref{eq:second_order_R2}, can be redefined as 
        \begin{equation*}
            J_{16} \coloneqq \E \Bigg[ \sup_{0 \leq n \leq N/2 - 1} \abs{ \sum_{k = 0}^{n} \e^{2 \gamma t_{2k}} h \inner{\hatY_{2n}^{f} - \hatY_{2n}^{c}, D^{2}a(\hatY_{2n}^{f}) \Big( \hatY_{2n + 1}^{f} - \hatY_{2n}^{f} \Big)^{2} - h \Delta a(\hatY_{2n}^{f})} }^{p/2} \Bigg].
        \end{equation*}
        We can express the term $D^{2}a(\hatY_{2n}^{f}) \Big( \hatY_{2n + 1}^{f} - \hatY_{2n}^{f} \Big)^{2}$ as
        \begin{equation}
            \label{eq:2nd_order_derivative_decomp}
            \begin{split}
                D^{2}a(\hatY_{2n}^{f}) \Big( \hatY_{2n + 1}^{f} - \hatY_{2n}^{f} \Big)^{2} &= D^{2}a(\hatY_{2n}^{f}) \Delta W_{2n}^{2} + D^{2}a(\hatY_{2n}^{f}) Q_{2n}^{2} + 2 D^{2}a(\hatY_{2n}^{f}) Q_{2n} \Delta W_{2n},
            \end{split}
        \end{equation}
        where \[ Q_{2n} \coloneqq Sh (\hatY_{2n}^{c} - \hatY_{2n}^{f}) + h a(\hatY_{2n}^{f}) + Da(\hatY_{2n}^{f}) \Delta Z_{2n} + \frac{h^{2}}{2} \cA a(\hatY_{2n}^{f}), \]
        and 
        \begin{equation}
            \label{eq:fine_diff_Q} \E \Big[ \norm{Q_{2n}}^{p} \Big] \leq C_{Q}^{p} p^{p/2} h^{p}
        \end{equation}
        where $C_{Q} > 0$ is a constant that depends only on $S, d$, and the constants from Assumptions~\ref{asmp:lipschitz_a}, \ref{asmp:aD_lipschitz}, and \ref{asmp:dissipativity_a}. Using equation~\eqref{eq:2nd_order_derivative_decomp}, we can express $J_{16}$ as \[ J_{16} \leq 3^{p/2 - 1} 2^{p/2} (\hatJ_{1} + \hatJ_{2} + \hatJ_{3} ), \]
        where $\hatJ_{1}, \hatJ_{2}$ and $\hatJ_{3}$ are defined as
        \begin{equation*}
            \begin{split}
                \hatJ_{1} &\coloneqq \E \Bigg[ \sup_{0 \leq n \leq N/2 - 1} \abs{\sum_{k = 0}^{n} \e^{2 \gamma t_{2k}} h \inner{\hatY_{2k}^{f} - \hatY_{2k}^{c}, D^{2}a(\hatY_{2k}^{f}) \Big( \Delta W_{2k}^{2} - \E \Big[ \Delta W_{2k}^{2} \Big] \Big)}}^{p/2} \Bigg], \\
                \hatJ_{2} &\coloneqq \E \Bigg[ \abs{\sum_{k = 0}^{N/2 - 1} \e^{2 \gamma t_{2k}} h \norm{\hatY_{2k}^{f} - \hatY_{2k}^{c}} \norm{D^{2}a(\hatY_{2k}^{f}) Q_{2k}^{2}}}^{p/2} \Bigg], \\
                \hatJ_{3} &\coloneqq \E \Bigg[ \sup_{0 \leq n \leq N/2 - 1} \abs{\sum_{k = 0}^{n} \e^{2 \gamma t_{2k}} h \inner{\hatY_{2k}^{f} - \hatY_{2k}^{c}, D^{2}a(\hatY_{2k}^{f}) Q_{2k} \Delta W_{2k}}}^{p/2} \Bigg].
            \end{split}
        \end{equation*}
        To bound the term $\hatJ_{1}$, we define the It\^o process $\di P_{1}(s)$ in $\bR^{d \times d}$, for all $s \in [0, T]$, such that
        \begin{equation*}
            \begin{split}
                \di P_{1}^{ij}(s) &= \begin{cases}
                    2 \Big( V_{1}^{i}(s) - V_{1}^{i} \big( 2 \floor{s / 2h} h \big) \Big) \di V_{1}^{i}(s), & \text{if } i = j \\
                    \Big( V_{1}^{i}(s) - V_{1}^{i} \big( 2 \floor{s / 2h} h \big) \Big) \di V_{1}^{j}(s) + \Big( V_{1}^{j}(s) - V_{1}^{j} \big( 2 \floor{s / 2h} h \big) \Big) \di V_{1}^{i}(s), & \text{if } i \neq j
                \end{cases}, \\
                P_{1}^{ij}(0) &= 0, 
            \end{split}
        \end{equation*}
        for all $i, j = 1, 2, \ldots, d$. The process $P_{1}$ is a Martingale process \cite[Definition 4.75]{arguin2021first}. Using the Burkholder--Davis--Gundy inequality, Jensen's inequality, and Young's inequality for any $\zeta > 0$, we have
        \begin{equation*}
            \begin{split}
                \hatJ_{1} &= \E \Bigg[ \sup_{0 \leq t \leq T} \abs{\int_{0}^{t} \e^{2 \gamma \floor{s / 2h} h} h \inner{\hatY_{2 \floor{s / 2h} }^{f} - \hatY_{2 \floor{s / 2h} }^{c}, D^{2}a(\hatY_{2 \floor{s / 2h} }^{f}) \di P_{1} (s)}}^{p/2} \Bigg] \\
                &\leq (2 C_{\text{BDG}} p)^{p/4} d^{3p/4} K_{b}^{p/2}  \E \Bigg[ \abs{\sup_{0 \leq n \leq N/2} \e^{2 \gamma t_{2n}} \norm{\hatY_{2n}^{f} - \hatY_{2n}^{c}}^{2}}^{p/4} \abs{\sum_{k = 0}^{N/2 - 1} \e^{2 \gamma t_{2k}} h^{4} } ^{p/4} \Bigg] \\
                &\leq \frac{1}{8 \zeta} \E \Bigg[ \sup_{0 \leq n \leq N/2} \e^{\gamma p t_{2n}} \norm{\hatY_{2n}^{f} - \hatY_{2n}^{c}}^{p} \Bigg] + C_{\hatJ_{1}}^{p} p^{p/2} h^{3p/2} \e^{\gamma p T}.
            \end{split}
        \end{equation*}
        For the term $\hatJ_{2}$, we have
        \begin{equation*}
            \begin{split}
                \hatJ_{2} &\leq d^{3p/4} K_{b}^{p/2} \E \Bigg[ \abs{ \sum_{k = 0}^{N/2 - 1} \e^{2 \gamma t_{2k}} h \norm{\hatY_{2k}^{f} - \hatY_{2k}^{c}} \norm{Q_{2k}}^{2} }^{p/2} \Bigg] \\
                &\leq d^{3p/4} K_{b}^{p/2} \E \Bigg[ \abs{\sup_{0 \leq n \leq N/2} \e^{\gamma t_{2k}} \norm{\hatY_{2n}^{f} - \hatY_{2n}^{c}}}^{p/2} \abs{ \sum_{k = 0}^{N/2 - 1} \e^{\gamma t_{2k}} h \norm{Q_{2k}}^{2} }^{p/2} \Bigg] \\
                &\leq \frac{1}{8 \zeta} \E \Bigg[ \sup_{0 \leq n \leq N/2} \e^{\gamma p t_{2k}} \norm{\hatY_{2n}^{f} - \hatY_{2n}^{c}}^{p} \Bigg] + \zeta C_{\hatJ_{2}}^{p} p^{p} h^{2p} \e^{\gamma p T}.
            \end{split}
        \end{equation*}
        To bound the term $\hatJ_{3}$, we rewrite it as
        \begin{equation*}
            \begin{split}
                \hatJ_{3} \leq 2^{p/2 - 1} \Big( \widetilde{J}_{1} + \widetilde{J}_{2} \Big),
            \end{split}
        \end{equation*}
        where
        \begin{equation*}
            \begin{split}
                \widetilde{J}_{1} &\coloneqq \E \Bigg[ \sup_{0 \leq n \leq N/2 - 1} \abs{\sum_{k = 0}^{n} \e^{2 \gamma t_{2k}} h \inner{\hatY_{2k}^{f} - \hatY_{2k}^{c}, D^{2} a(\hatY_{2k}^{f}) \widetilde{Q}_{2k} \Delta W_{2k} }}^{p/2} \Bigg], \\
                \widetilde{J}_{2} &\coloneqq \E \Bigg[ \sup_{0 \leq n \leq N/2 - 1} \abs{\sum_{k = 0}^{n} \e^{2 \gamma t_{2k}} h \inner{\hatY_{2k}^{f} - \hatY_{2k}^{c}, D^{2} a(\hatY_{2k}^{f}) \big( Da(\hatY_{2k}^{f}) \Delta Z_{2k} \big) \Delta W_{2k} }}^{p/2} \Bigg],
            \end{split}
        \end{equation*}
        and
        \[ \widetilde{Q}_{2k} \coloneqq Sh (\hatY_{2k}^{c} - \hatY_{2k}^{f}) + h a(\hatY_{2k}^{f}) + \frac{h^{2}}{2} \cA a(\hatY_{2k}^{f}), \] where \[ \E \bigg[ \norm{\widetilde{Q}_{2k}}^{p} \bigg] \leq C_{Q}^{p} p^{p/2} h^{p}. \]
        For the term $\widetilde{J}_{1}$, using the Burkholder--Davis--Gundy inequality, Jensen's inequality, and Young's inequality for any $\zeta > 0$, we obtain
        \begin{equation*}
            \begin{split}
                \widetilde{J}_{1} &= \E \Bigg[ \sup_{0 \leq t \leq T} \abs{\int_{0}^{t} \e^{2 \gamma \floor{s / 2h} h} h \inner{\hatY_{2 \floor{s / 2h}}^{f} - \hatY_{2 \floor{s / 2h}}^{c}, D^{2} a(\hatY_{2 \floor{s / 2h}}^{f}) \widetilde{Q}_{2 \floor{s / 2h}} \di W_{s} }}^{p/2} \Bigg] \\
                &\leq (C_{\text{BDG}} p)^{p/4} \E \Bigg[ \abs{\sum_{k = 0}^{N/2 - 1} \e^{4 \gamma t_{2k}} h^{3} \norm{\hatY_{2k}^{f} - \hatY_{2k}^{c}}^{2} \norm{D^{2} a(\hatY_{2k}^{f})}^{2} \norm{ \widetilde{Q}_{2k} }^{2}}^{p/4} \Bigg] \\
                &\leq (C_{\text{BDG}} p)^{p/4} d^{3p/4} K_{b}^{p/2} \E \Bigg[ \abs{\sup_{0 \leq n \leq N/2} \e^{2 \gamma t_{2n}} \norm{\hatY_{2n}^{f} - \hatY_{2n}^{c}}^{2}}^{p/4} \abs{\sum_{k = 0}^{N/2 - 1} \e^{2 \gamma t_{2k}} h^{3} \norm{ \widetilde{Q}_{2k} }^{2}}^{p/4} \Bigg] \\
                &\leq \frac{1}{8 \zeta} \E \Bigg[ \sup_{0 \leq n \leq N/2} \e^{\gamma p t_{2n}} \norm{\hatY_{2n}^{f} - \hatY_{2n}^{c}}^{p} \Bigg] + \zeta C_{\widetilde{J}_{1}}^{p} p^{p} h^{2p} \e^{\gamma p T}.
            \end{split}
        \end{equation*}
        Using, the H\"older's inequality, Jensen's inequality, Young's inequality for any $\zeta > 0$, we obtain
        \begin{equation*}
            \begin{split}
                \widetilde{J}_{2} &\leq \E \Bigg[ \abs{\sum_{k = 0}^{N/2 - 1} \e^{2 \gamma t_{2k}} h \norm{\hatY_{2k}^{f} - \hatY_{2k}^{c}} \norm{D^{2} a(\hatY_{2k}^{f}) \big( Da(\hatY_{2k}^{f}) \Delta Z_{2k} \big) \Delta W_{2k}}}^{p/2} \Bigg] \\
                &\leq d^{5p/4} K_{b}^{p} \E \Bigg[ \abs{\sup_{0 \leq n \leq N/2} \e^{\gamma t_{2n}} \norm{\hatY_{2n}^{f} - \hatY_{2n}^{c}} }^{p/2} \abs{\sum_{k = 0}^{N/2 - 1} \e^{\gamma t_{2k}} h \norm{\Delta Z_{2k}} \norm{\Delta W_{2k}}}^{p/2} \Bigg] \\
                &\leq \frac{1}{8 \zeta} \E \Bigg[ \sup_{0 \leq n \leq N/2} \e^{\gamma p t_{2n}} \norm{\hatY_{2n}^{f} - \hatY_{2n}^{c}}^{p} \Bigg] + \zeta C_{\widetilde{J}_{2}}^{p} p^{p} h^{2p} \e^{\gamma p T}.
            \end{split}
        \end{equation*}
        Combining all the estimates together, there exists a constant $\widetilde{C}_{3} > 0$, independent of $T$ and $p$, such that
        \begin{equation*}
            \E \Bigg[ \sup_{0 \leq n \leq N/2} \e^{\gamma p t_{2n}} \norm{\hatY_{2n}^{f} - \hatY_{2n}^{c}}^{p} \Bigg] \leq \widetilde{C}_{3}^{p} p^{3p/2} h^{3p/2} \e^{\gamma p T}.
        \end{equation*}
        As the above inequality holds for any $T > 0$, we have
        \begin{equation}
            \label{eq:diff_bound_3}
            \sup_{0 \leq n \leq N/2} \E \Bigg[ \norm{\hatY_{2n}^{f} - \hatY_{2n}^{c}}^{p} \Bigg] \leq \widetilde{C}_{3}^{p} p^{3p/2} h^{3p/2}.
        \end{equation}
        Choosing a constant $\overline{C}_{1} \coloneqq \max \big( \widetilde{C}_{1}, \; \widetilde{C}_{2}, \; \widetilde{C}_{3} \big)$, for all sufficiently small $0 < h < \kappa_{1}$, we have
        \begin{equation}
            \label{eq:diff_bound_even}
            \sup_{0 \leq n \leq N/2} \E \Bigg[ \norm{\hatY_{2n}^{f} - \hatY_{2n}^{c}}^{p} \Bigg] \leq \overline{C}_{1}^{p} \min \Big( p^{p/2} h^{p/2}, \; p^{p} h^{p}, \; p^{3p/2} h^{3p/2} \Big).
        \end{equation}
        Similarly, using equations~\eqref{eq:fine_odd} and \eqref{eq:coarse_odd}, we can write the difference between the fine and the coarse trajectories at the odd time-step as
        \begin{equation*}
            \begin{split}
                \hatY_{2n + 1}^{f} - \hatY_{2n + 1}^{c} &= (1 - 2Sh) (\hatY_{2n}^{f} - \hatY_{2n}^{c}) + h (a(\hatY_{2n}^{f}) - a(\hatY_{2n}^{c})) + \hatI_{2n},
            \end{split}
        \end{equation*}
        where \[ \hatI_{2n} \coloneqq (Da(\hatY_{2n}^{f}) - Da(\hatY_{2n}^{c})) \Delta Z_{2n} + \frac{h^{2}}{2} ( \cA a(\hatY_{2n}^{f}) - \cA a(\hatY_{2n}^{c}) ). \]
        Taking the Euclidean norm of the difference and squaring it, using Young's inequality, and Assumptions~\ref{asmp:lipschitz_a}, \ref{asmp:aD_lipschitz}, \ref{asmp:dissipativity_a}, and for all sufficiently small $h > 0$, we have
        \begin{equation*}
            \begin{split}
                \norm{\hatY_{2n + 1}^{f} - \hatY_{2n + 1}^{c}}^{2} &\leq \Big\{ 2 - 6Sh + (2 \lambda + d K_{b}^{2} ) h + C h^{2} \Big\} \norm{\hatY_{2n}^{f} - \hatY_{2n}^{c}}^{2} + 3 \norm{\hatI_{2n}}^{2} \\
                ~
                &\leq C \Big( \norm{\hatY_{2n}^{f} - \hatY_{2n}^{c}}^{2} + \norm{\hatY_{2n}^{f} - \hatY_{2n}^{c}}^{2} \norm{\Delta Z_{2n}}^{2} \Big).
            \end{split}
        \end{equation*}
        Raising both sides of the above inequality to the power $p/2$, and taking the expectation, we obtain
        \begin{equation*}
            \begin{split}
                \E \Big[ \norm{\hatY_{2n + 1}^{f} - \hatY_{2n + 1}^{c}}^{p}  \Big] &\leq 2^{p/2 - 1} C^{p/2} \Big( \E \Big[ \norm{\hatY_{2n}^{f} - \hatY_{2n}^{c}}^{p} \Big] + \E \Big[ \norm{\hatY_{2n}^{f} - \hatY_{2n}^{c}}^{p} \norm{\Delta Z_{2n}}^{p} \Big] \Big).
            \end{split}
        \end{equation*}
        Taking the supremum on both sides of the above inequality, there exist constants $\kappa_{4}, \kappa_{5} > 0$ independent of $T$ and $p$, for all $0 < h < \min(\kappa_{3}, \kappa_{4} / \sqrt[3]{p})$,
        \begin{equation*}
            \sup_{0 \leq n \leq N} \E \Big[ \norm{\hatY_{t_{n}}^{f} - \hatY_{t_{n}}^{c}}^{p} \Big] \leq \kappa_{5} \min \Big( p^{p/2} h^{p/2}, \; p^{p} h^{p}, \; p^{3p/2} h^{3p/2} \Big).
        \end{equation*}

    \subsubsection*{Proof of Theorem~\ref{thm:radon_nikodym}}
        \allowdisplaybreaks
        \begingroup
            For the sake of brevity, we show the Radon--Nikodym derivative bound only for the fine trajectory. The proof argument can be extended similarly for the coarse trajectory. Recall that the $p-$th power of the  Radon--Nikodym derivative for the fine trajectory is given by
            \begin{equation*}
                \begin{split}
                    \E \bigg[ \abs{R_{T}^{f}}^{p} \bigg] &= \E \Bigg[ \exp \Bigg( - p S \sum_{n = 0}^{N - 1} \inner{\hatY_{n}^{f} - \hatY_{n}^{c}, \Delta V_{1, n}} + \sqrt{3} p S \sum_{n = 0}^{N - 1} \inner{\hatY_{n}^{f} - \hatY_{n}^{c}, \Delta V_{2, n}} \\
                    &\qquad - 2 p S^{2} h \sum_{n = 0}^{N - 1} \norm{\hatY_{n}^{f} - \hatY_{n}^{c}}^{2} \Bigg) \Bigg] \\
                    ~
                    &= \E \Bigg[ \exp \Bigg( - p S \sum_{n = 0}^{N - 1} \inner{\hatY_{n}^{f} - \hatY_{n}^{c}, \Delta V_{1, n}} + \sqrt{3} p S \sum_{n = 0}^{N - 1} \inner{\hatY_{n}^{f} - \hatY_{n}^{c}, \Delta V_{2, n}} \\
                    &\quad - 4 p^{2} S^{2} h \sum_{n = 0}^{N - 1} \norm{\hatY_{n}^{f} - \hatY_{n}^{c}}^{2} + 2 (2p - 1) p S^{2} h \sum_{n = 0}^{N - 1} \norm{\hatY_{n}^{f} - \hatY_{n}^{c}}^{2} \Bigg) \Bigg].
                \end{split}
            \end{equation*}
            Using H\"older's inequality, we can write
            \begin{equation*}
                \E \bigg[ \abs{R_{T}^{f}}^{p} \bigg] \leq J_{1}^{1/2} J_{2}^{1/2},
            \end{equation*}
            where
            \begin{equation*}
                \begin{split}
                    J_{1} &\coloneqq \E \Bigg[ \exp \Bigg( 4 (2p - 1) p S^{2} h \sum_{n = 0}^{N - 1} \norm{\hatY_{n}^{f} - \hatY_{n}^{c}}^{2} \Bigg) \Bigg], \\
                    J_{2} &\coloneqq \E \Bigg[ \exp \Bigg( - 2 p S \sum_{n = 0}^{N - 1} \inner{\hatY_{n}^{f} - \hatY_{n}^{c}, \Delta V_{1, n}} + 2 \sqrt{3} p S \sum_{n = 0}^{N - 1} \inner{\hatY_{n}^{f} - \hatY_{n}^{c}, \Delta V_{2, n}} \\
                    &\quad - 8 p^{2} S^{2} h \sum_{n = 0}^{N - 1}  \norm{\hatY_{n}^{f} - \hatY_{n}^{c}}^{2} \Bigg) \Bigg] \leq 1.
                \end{split}
            \end{equation*}
            since the term $J_{2}$ is a super-martingale. Using H\"older's inequality, the term $J_{1}$ can be bounded as
            \begin{equation*}
                \begin{split}
                    J_{1} &\leq \E \Bigg[ \exp \Bigg( 8 p^{2} S^{2} h \sum_{n = 0}^{N - 1} \norm{\hatY_{n}^{f} - \hatY_{n}^{c}}^{2} \Bigg) \Bigg] = \E \Bigg[ \prod_{n = 0}^{N - 1} \exp \Bigg( 8 p^{2} S^{2} h \norm{\hatY_{n}^{f} - \hatY_{n}^{c}}^{2} \Bigg) \Bigg] \\
                    ~
                    &\leq \prod_{n = 0}^{N - 1} \E \Bigg[ \exp \Bigg( 8 p^{2} S^{2} T \norm{\hatY_{n}^{f} - \hatY_{n}^{c}}^{2} \Bigg) \Bigg]^{1 / N} = \prod_{n = 0}^{N/2 - 1} (\hatJ_{2n + 2})^{1/N} \prod_{n = 0}^{N/2 - 1} (\hatJ_{2n + 1})^{1/N},
                \end{split}
            \end{equation*}
            where
            \begin{equation}
                \label{eq:hatJ_terms_odd_even}
                \begin{split}
                    \hatJ_{2n + 2} &\coloneqq \E \Bigg[ \exp \Bigg( 8 p^{2} S^{2} T \norm{\hatY_{2n + 2}^{f} - \hatY_{2n + 2}^{c}}^{2} \Bigg) \Bigg], \\
                    \hatJ_{2n + 1} &\coloneqq \E \Bigg[ \exp \Bigg( 8 p^{2} S^{2} T \norm{\hatY_{2n + 1}^{f} - \hatY_{2n + 1}^{c}}^{2} \Bigg) \Bigg].
                \end{split}
            \end{equation}
            To bound the term $\hatJ_{2n + 2}$, we define the set
            \begin{equation}
                \label{eq:set_theta}
                \Theta_{2n} \coloneqq \Bigg\{ \omega \in \Omega \; \bigg\vert \; \norm{\hatY_{2n}^{f} - \hatY_{2n}^{c}} \geq \nu_{1} \abs{\log h} \Bigg\},
            \end{equation}
            where $\nu_{1} > 0$ is a constant independent of $T$. Using Markov's inequality, and the result from Theorem~\ref{thm:convergence}, we obtain
            \begin{equation*}
                \begin{split}
                    \Prob \big( \Theta_{2n} \big) \leq \frac{\E \Bigg[ \norm{\hatY_{2n}^{f} - \hatY_{2n}^{c}}^{q} \Bigg]}{\nu_{1}^{q} \abs{\log h}^q} \leq \nu_{2}^{q} \nu_{1}^{-q} q^{q/2} h^{q/2} \abs{\log h}^{-q},
                \end{split}
            \end{equation*}
            where $\nu_{1}, \nu_{2} > 0$ are constants independent of $T$. Using the law of total expectation, we obtain
            \begin{equation*}
                \begin{split}
                    & \E \Bigg[ \exp \bigg( 8 p^{2} S^{2} T \norm{\hatY_{2n + 2}^{f} - \hatY_{2n + 2}^{c}}^{2} \bigg) \Bigg] \\
                    = & \underbrace{\E \Bigg[ \ind_{\Theta_{2n}} \exp \bigg( 8 p^{2} S^{2} T \norm{\hatY_{2n + 2}^{f} - \hatY_{2n + 2}^{c}}^{2} \bigg) \Bigg]}_{\eqqcolon \hatI_{1}} + \underbrace{\E \Bigg[ \ind_{\Theta_{2n}^{\comp}} \exp \bigg( 8 p^{2} S^{2} T \norm{\hatY_{2n + 2}^{f} - \hatY_{2n + 2}^{c}}^{2} \bigg) \Bigg]}_{\eqqcolon \hatI_{2}}.
                \end{split}
            \end{equation*}
            We begin the analysis by bounding the term $\hatI_{1}$. For all sufficiently small $h > 0$, using H\"older's inequality with the H\"older exponent $\bar{p} \coloneqq h^{-1} \abs{\log h} \implies \bar{q} \coloneqq \abs{\log h} / \big( \abs{\log h} - h \big)$, we obtain
            \begin{equation*}
                \begin{split}
                    \hatI_{1} &= \E \Bigg[ \ind_{\Theta_{2n}} \exp \bigg( 8 p^{2} S^{2} T \norm{\hatY_{2n + 2}^{f} - \hatY_{2n + 2}^{c}}^{2} \bigg) \Bigg] \\
                    &\leq \E \Big[ \ind_{\Theta_{2n}}^{\bar{p}} \Big]^{\frac{1}{\bar{p}}} \E \Bigg[ \exp \bigg( 8 p^{2} S^{2} T \bar{q} \norm{\hatY_{2n + 2}^{f} - \hatY_{2n + 2}^{c}}^{2} \bigg) \Bigg]^{\frac{1}{\bar{q}}} \\
                    &\leq \bigg( \nu_{2}^{q} \nu_{1}^{-q} q^{q/2} h^{q/2} \abs{\log h}^{-q} \bigg)^{\frac{h}{\abs{\log h}}} \E \Bigg[ \exp \bigg( 8 p^{2} S^{2} T \frac{\abs{\log h}}{\abs{\log h} - h} \norm{\hatY_{2n + 2}^{f} - \hatY_{2n + 2}^{c}}^{2} \bigg) \Bigg]^{1 - \frac{h}{\abs{\log h}}}.
                \end{split}
            \end{equation*}
            Using equation~\eqref{eq:fc_even_diff}, and Young's inequality, we have
            \begin{equation}
                \label{eq:fc_diff_radon_nikodym}
                \begin{split}
                    \norm{\hatY_{2n + 2}^{f} - \hatY_{2n + 2}^{c}}^{2} &\leq \Big( 1 - 4 \big( 2S - \lambda \big) h + C h \abs{\log h}^{-1} \Big) \norm{\hatY_{2n}^{f} - \hatY_{2n}^{c}}^{2} \\
                    & + C \Bigg\{ h^{3} \abs{\log h} \bigg( 1 + \norm{\hatY_{2n}^{f}}^{2} + \norm{\hatY_{2n}^{c}}^{2} \bigg) \\
                    & + h \abs{\log h} \norm{\Delta W_{2n}}^{2} + \norm{\Delta Z_{2n}}^{2} + \norm{\Delta Z_{2n + 1}}^{2} \Bigg\} \\
                    & + 2 (1 - Sh) ( 1 - 4Sh + 2 S^{2} h^{2}) \inner{\hatY_{2n}^{f} - \hatY_{2n}^{c}, (Da(\hatY_{2n}^{f}) - Da(\hatY_{2n}^{c})) \Delta Z_{2n}} \\
                    & + 2 ( 1 - 4Sh + 2 S^{2} h^{2}) \inner{\hatY_{2n}^{f} - \hatY_{2n}^{c}, (Da(\hatY_{2n + 1}^{f}) - Da(\hatY_{2n}^{f})) \Delta Z_{2n + 1}} \\
                    & + 2 ( 1 - 4Sh + 2 S^{2} h^{2}) \inner{\hatY_{2n}^{f} - \hatY_{2n}^{c}, (Da(\hatY_{2n}^{f}) - Da(\hatY_{2n}^{c})) \Delta Z_{2n + 1}} \\
                    & - 2 h ( 1 - 4Sh + 2 S^{2} h^{2}) \inner{\hatY_{2n}^{f} - \hatY_{2n}^{c}, Da(\hatY_{2n}^{c}) \Delta W_{2n}},
                \end{split}
            \end{equation}
            where $C > 0$ is a constant that depends only on $S, d$ and the constants from Assumptions \ref{asmp:lipschitz_a}, \ref{asmp:aD_lipschitz}, and \ref{asmp:dissipativity_a}. Note that the constant $C$ may change from line to line in the proof argument.

            Note that for all $0 < h < w(2) / 2, \; \; \frac{\abs{\log h}}{\abs{\log h} - h} \leq 2$, where $w(x)$ denotes the product logarithm function. Using equation~\eqref{eq:fc_diff_radon_nikodym}, the tower property of expectations, for all sufficiently small $h > 0$, there exists a $\gamma_{1} \in (0, 2S - \lambda)$, such that the following inequality holds
            \begin{equation*}
                \begin{split}
                    & \E \Bigg[ \exp \bigg( 8 p^{2} S^{2} T \frac{\abs{\log h}}{\abs{\log h} - h} \norm{\hatY_{2n + 2}^{f} - \hatY_{2n + 2}^{c}}^{2} \bigg) \Bigg] \\
                    ~
                    \leq & \E \Bigg[ \exp \Bigg( 8 p^{2} S^{2} T \bigg( (1 - 2 \gamma_{1} h) \norm{\hatY_{2n}^{f} - \hatY_{2n}^{c}}^{2} \\
                    &\quad + C \Bigg\{ h^{3} \abs{\log h} \bigg( 1 + \norm{\hatY_{2n}^{f}}^{2} + \norm{\hatY_{2n}^{c}}^{2} \bigg) + h \abs{\log h} \norm{\Delta W_{2n}}^{2} + \norm{\Delta Z_{2n}}^{2} \Bigg\} \\
                    &\quad + 2 \frac{\abs{\log h}}{\abs{\log h} - h} (1 - Sh) ( 1 - 4Sh + 2 S^{2} h^{2}) \inner{\hatY_{2n}^{f} - \hatY_{2n}^{c}, (Da(\hatY_{2n}^{f}) - Da(\hatY_{2n}^{c})) \Delta Z_{2n}} \\
                    &\quad - 2 h \frac{\abs{\log h}}{\abs{\log h} - h} ( 1 - 4Sh + 2 S^{2} h^{2}) \inner{\hatY_{2n}^{f} - \hatY_{2n}^{c}, Da(\hatY_{2n}^{c}) \Delta W_{2n}} \bigg) \Bigg) \times \tildeJ_{1} \Bigg],
                \end{split}
            \end{equation*}
            where $\tildeJ_{1}$ is defined as
            \begin{equation*}
                \begin{split}
                    \tildeJ_{1} &\coloneqq \E \Bigg[ \exp \Bigg( 8 p^{2} S^{2} T \frac{\abs{\log h}}{\abs{\log h} - h} \\
                    &\qquad \qquad \bigg( 2 ( 1 - 4Sh + 2 S^{2} h^{2}) \inner{\hatY_{2n}^{f} - \hatY_{2n}^{c}, (Da(\hatY_{2n + 1}^{f}) - Da(\hatY_{2n}^{f})) \Delta Z_{2n + 1}} \\
                    &\qquad \qquad + 2 ( 1 - 4Sh + 2 S^{2} h^{2}) \inner{\hatY_{2n}^{f} - \hatY_{2n}^{c}, (Da(\hatY_{2n}^{f}) - Da(\hatY_{2n}^{c})) \Delta Z_{2n + 1}} \\
                    &\qquad \qquad + C \norm{\Delta Z_{2n + 1}}^{2} \bigg) \Bigg) \; \bigg\vert \; \cF_{t_{2n + 1}} \Bigg].
                \end{split}
            \end{equation*}
            Using H\"older's inequality conditioned on the filtration $\cF_{t_{2n + 1}}$ generated upto time $t_{2n + 1}$, we get
            \begin{align*}
                \tildeJ_{1} &\leq \E \Bigg[ \exp \Bigg( 48 p^{2} S^{2} T \frac{\abs{\log h}}{\abs{\log h} - h} ( 1 - 4Sh + 2 S^{2} h^{2}) \\
                &\qquad \qquad \inner{\hatY_{2n}^{f} - \hatY_{2n}^{c}, (Da(\hatY_{2n + 1}^{f}) - Da(\hatY_{2n}^{f})) \Delta Z_{2n + 1}} \Bigg) \; \bigg\vert \; \cF_{t_{2n + 1}} \Bigg]^{1/3} \times \\
                &\quad \E \Bigg[ \exp \Bigg( 48 p^{2} S^{2} T \frac{\abs{\log h}}{\abs{\log h} - h} ( 1 - 4Sh + 2 S^{2} h^{2}) \\
                &\qquad \qquad \inner{\hatY_{2n}^{f} - \hatY_{2n}^{c}, (Da(\hatY_{2n}^{f}) - Da(\hatY_{2n}^{c})) \Delta Z_{2n + 1}} \Bigg) \; \bigg\vert \; \cF_{t_{2n + 1}} \Bigg]^{1/3} \times  \\
                &\quad \E \Bigg[ \exp \Bigg( 24 p^{2} S^{2} T \frac{\abs{\log h}}{\abs{\log h} - h} C \norm{\Delta Z_{2n + 1}}^{2} \Bigg) \; \bigg\vert \; \cF_{t_{2n + 1}} \Bigg]^{1/3}.
            \end{align*}
            We define the random variable $\chi_{n} \sim \mathcal{N}(0, \Id)$ for all $n = 0, \ldots, N$.
            Recall that the random variable $\xi \coloneqq \sum_{k = 1}^{d} (\chi_{n, k})^{2}$ is chi-squared distributed. The moment generating function for the chi-squared distribution is given by
            \begin{equation}
                \label{eq:mgf_chi_squared}
                \E \big[ \exp (\theta \xi) \big] = (1 - 2 \theta)^{-d / 2}, \qquad \forall \theta < \frac{1}{2}.
            \end{equation}
            Using equation~\eqref{eq:mgf_chi_squared}, for all sufficiently small $h > 0$, the following inequality holds
            \begin{equation*}
                \begin{split}
                    & \E \Bigg[ \exp \bigg( 24 p^{2} S^{2} T \frac{\abs{\log h}}{\abs{\log h} - h} C \norm{\Delta Z_{2n + 1}}^{2} \bigg) \; \bigg \vert \; \cF_{t_{2n + 1}} \Bigg]^{1/3} \\
                    =& \E \Bigg[ \exp \bigg( 16 p^{2} S^{2} T C h^{3} \norm{\chi_{2n + 1}}^{2} \bigg) \Bigg]^{1/3} \leq \frac{1}{\big( 1 - 32 p^{2} S^{2} T C h^{3} \big)^{d/6}} \leq \exp(C T h^{3}),
                \end{split}
            \end{equation*}
            for all $T \leq C_{0} h^{-3}$, where $C_{0}$ is a constant that is independent of $h$. Using Assumptions~\ref{asmp:lipschitz_a}, and \ref{asmp:aD_lipschitz}, and the moment generating function for the standard Normal distribution, we have
            \begin{equation*}
                \begin{split}
                    &\E \Bigg[ \exp \bigg( 48 p^{2} S^{2} T \frac{\abs{\log h}}{\abs{\log h} - h} (1 - 4Sh + 2S^{2}h^{2}) \\
                    &\qquad \qquad \inner{\hatY_{2n}^{f} - \hatY_{2n}^{c}, (Da(\hatY_{2n + 1}^{f}) - Da(\hatY_{2n}^{f})) \Delta Z_{2n + 1}} \bigg) \; \bigg \vert \; \cF_{t_{2n + 1}} \Bigg]^{1/3} \\
                    = & \E \Bigg[ \exp \bigg( \frac{48}{\sqrt{3}} p^{2} S^{2} T \frac{\abs{\log h}}{\abs{\log h} - h} h^{3/2} (1 - 4Sh + 2S^{2}h^{2}) \\
                    &\qquad \inner{\hatY_{2n}^{f} - \hatY_{2n}^{c}, (Da(\hatY_{2n + 1}^{f}) - Da(\hatY_{2n}^{f})) \chi_{2n + 1}} \bigg) \; \bigg \vert \; \cF_{t_{2n + 1}} \Bigg]^{1/3} \leq \exp \bigg( C T^{2} h^{3} \norm{\hatY_{2n}^{f} - \hatY_{2n}^{c}}^{2} \bigg).
                \end{split}
            \end{equation*}
            Similarly, we have
            \begin{align*}
                &\E \Bigg[ \exp \bigg( 48 p^{2} S^{2} T \frac{\abs{\log h}}{\abs{\log h} - h} (1 - 4Sh + 2S^{2}h^{2}) \\
                &\qquad \inner{\hatY_{2n}^{f} - \hatY_{2n}^{c}, (Da(\hatY_{2n}^{f}) - Da(
                \hatY_{2n}^{c})) \Delta Z_{2n + 1}} \bigg) \; \bigg \vert \; \cF_{t_{2n + 1}} \Bigg]^{1/3} \leq \exp \bigg( C T^{2} h^{3} \norm{\hatY_{2n}^{f} - \hatY_{2n}^{c}}^{2} \bigg).
            \end{align*}
            Combining the above estimates, we obtain
            \begin{align*}
                & \E \Bigg[ \exp \Bigg( 8 p^{2} S^{2} T \frac{\abs{\log h}}{\abs{\log h} - h} \norm{\hatY_{2n + 2}^{f} - \hatY_{2n + 2}^{c}}^{2} \Bigg) \Bigg] \\
                ~
                \leq & \E \Bigg[ \exp \Bigg( 8 p^{2} S^{2} T \bigg( (1 - 2 \gamma_{1} h) \norm{\hatY_{2n}^{f} - \hatY_{2n}^{c}}^{2} + C  h^{3} \abs{\log h} \bigg( 1 + \norm{\hatY_{2n}^{f}}^{2} + \norm{\hatY_{2n}^{c}}^{2} \bigg) \bigg) \Bigg) \times \\
                &\quad \exp \Bigg( C T h^{3} + C T^{2} h^{3} \norm{\hatY_{2n}^{f} - \hatY_{2n}^{c}}^{2} \Bigg) \times \tildeJ_{2} \Bigg],
            \end{align*}
            where $\tildeJ_{2}$ is defined as
            \begin{equation*}
                \begin{split}
                    \tildeJ_{2} &\coloneqq \E \Bigg[ \exp \Bigg( 8 p^{2} S^{2} T \frac{\abs{\log h}}{\abs{\log h} - h} \bigg( C h \abs{\log h} \norm{\Delta W_{2n}}^{2} + C \norm{\Delta Z_{2n}}^{2} \\
                    &\quad + 2 (1 - Sh) ( 1 - 4Sh + 2 S^{2} h^{2}) \inner{\hatY_{2n}^{f} - \hatY_{2n}^{c}, (Da(\hatY_{2n}^{f}) - Da(\hatY_{2n}^{c})) \Delta Z_{2n}} \\
                    &\quad - 2 h ( 1 - 4Sh + 2 S^{2} h^{2}) \inner{\hatY_{2n}^{f} - \hatY_{2n}^{c}, Da(\hatY_{2n}^{c}) \Delta W_{2n}} \bigg) \Bigg) \; \bigg\vert \; \cF_{t_{2n}} \Bigg].
                \end{split}
            \end{equation*}
            Using H\"older's inequality conditioned on the filtration $\cF_{t_{2n}}$ generated until time $t_{2n}$, we obtain
            \begin{align*}
                \tildeJ_{2} &\leq \E \bigg[ \exp \bigg( 64 p^{2} S^{2} T C h \abs{\log h} \norm{\Delta W_{2n}}^{2} \bigg) \; \bigg\vert \; \cF_{t_{2n}} \bigg]^{1/4} \times \E \bigg[ \exp \bigg( 64 p^{2} S^{2} T C \norm{\Delta Z_{2n}}^{2} \bigg) \; \bigg\vert \; \cF_{t_{2n}} \bigg]^{1/4} \\
                & \times \E \bigg[ \exp \bigg( 64 p^{2} S^{2} T \frac{\abs{\log h}}{\abs{\log h} - h} (1 - Sh) (1 - 4Sh + 2S^{2}h^{2}) \\
                &\qquad \qquad \inner{\hatY_{2n}^{f} - \hatY_{2n}^{c}, (Da(\hatY_{2n}^{f}) - Da(\hatY_{2n}^{c})) \Delta Z_{2n}} \bigg) \; \bigg\vert \; \cF_{t_{2n}} \bigg]^{1/4} \times \\
                & \E \bigg[ \exp \bigg(- 64 p^{2} S^{2} T \frac{\abs{\log h}}{\abs{\log h} - h} h (1 - 4Sh + 2S^{2}h^{2}) \inner{\hatY_{2n}^{f} - \hatY_{2n}^{c}, Da(\hatY_{2n}^{c}) \Delta W_{2n}} \bigg) \; \bigg\vert \; \cF_{t_{2n}} \bigg]^{1/4}.
            \end{align*}
            Using a similar argument as before and combining all the estimates, we obtain
            \begin{equation*}
                \begin{split}
                    & \E \Bigg[ \exp \Bigg( 8 p^{2} S^{2} T \frac{\abs{\log h}}{\abs{\log h} - h} \norm{\hatY_{2n + 2}^{f} - \hatY_{2n + 2}^{c}}^{2} \Bigg) \Bigg] \\
                    ~
                    \leq & \E \Bigg[ \exp \Bigg( 8 p^{2} S^{2} T \bigg( (1 - 2 \gamma_{1} h) \norm{\hatY_{2n}^{f} - \hatY_{2n}^{c}}^{2} + C h^{3} \abs{\log h} \bigg( \norm{\hatY_{2n}^{f}}^{2} + \norm{\hatY_{2n}^{c}}^{2} \bigg) \bigg) \Bigg) \times \\
                    &\qquad \exp \Bigg( C T h^{2} \abs{\log h} + C T^{2} h^{3} \norm{\hatY_{2n}^{f} - \hatY_{2n}^{c}}^{2} \Bigg) \Bigg].
                \end{split}
            \end{equation*}
            Note that for all $1 \leq T \leq \gamma_{1} h^{-2} / \big( 2 \max \big( 1, \; C \big) \big)$, and for all sufficiently small $h > 0$, there exists a $\hatgamma_{1} > 0$ such that the following inequality holds
            \begin{equation*}
                \begin{split}
                    & \E \Bigg[ \exp \Bigg( 8 p^{2} S^{2} T \frac{\abs{\log h}}{\abs{\log h} - h} \norm{\hatY_{2n + 2}^{f} - \hatY_{2n + 2}^{c}}^{2} \Bigg) \Bigg] \\
                    ~
                    \leq & \exp \big( C T h^{2} \abs{\log h} \big) \\
                    &\qquad \qquad \E \Bigg[ \exp \Bigg( 8 p^{2} S^{2} T \bigg( (1 - 2 \hatgamma_{1} h) \norm{\hatY_{2n}^{f} - \hatY_{2n}^{c}}^{2} + C h^{3} \abs{\log h} \bigg( \norm{\hatY_{2n}^{f}}^{2} + \norm{\hatY_{2n}^{c}}^{2} \bigg) \bigg) \Bigg) \Bigg].
                \end{split}
            \end{equation*}
            Using H\"older's inequality once again, we obtain
            \begin{align*}
                & \E \Bigg[ \exp \Bigg( 8 p^{2} S^{2} T \frac{\abs{\log h}}{\abs{\log h} - h} \norm{\hatY_{2n + 2}^{f} - \hatY_{2n + 2}^{c}}^{2} \Bigg) \Bigg] \\
                \leq & \exp \big( C T h^{2} \abs{\log h} \big) \E \Bigg[ \exp \bigg( 8 p^{2} S^{2} T \norm{\hatY_{2n}^{f} - \hatY_{2n}^{c}}^{2} \bigg) \Bigg]^{1 - 2 \hatgamma_{1} h} \\
                &\qquad \qquad \E \Bigg[ \exp \bigg( C T h^{2} \abs{\log h} \bigg( \norm{\hatY_{2n}^{f}}^{2} + \norm{\hatY_{2n}^{c}}^{2} \bigg) \bigg) \Bigg]^{2 \hatgamma_{1} h}.
            \end{align*}
            Using Theorem~\ref{thm:stability}, H\"older's inequality, Taylor series expansion of the exponential function, Stirling's approximation, and Fatou's lemma, the following inequality holds
            \begin{align*}
                &\E \Bigg[ \exp \Bigg( C T h^{2} \abs{\log h} \bigg( \norm{\hatY_{2n}^{f}}^{2} + \norm{\hatY_{2n}^{c}}^{2} \bigg) \Bigg) \Bigg]^{2 \hatgamma_{1} h} \\
                ~
                \leq & \E \Bigg[ \exp \Bigg( C T h^{2} \abs{\log h} \norm{\hatY_{2n}^{f}}^{2} \Bigg) \Bigg]^{\hatgamma_{1} h} \E \Bigg[ \exp \Bigg( C T h^{2} \abs{\log h} \norm{\hatY_{2n}^{c}}^{2} \Bigg) \Bigg]^{\hatgamma_{1} h} \\
                ~
                = & \left( 1 + \sum_{k = 1}^{\infty} \frac{(C T h^{2} \abs{\log h})^{k} \E \bigg[ \norm{\hatY_{2n}^{f}}^{2k} \bigg]}{k!} \right)^{\hatgamma_{1} h} \left( 1 + \sum_{k = 1}^{\infty} \frac{(C T h^{2} \abs{\log h})^{k} \E \bigg[ \norm{\hatY_{2n}^{c}}^{2k} \bigg]}{k!} \right)^{\hatgamma_{1} h} \\
                ~
                \leq & \left( 1 + \sum_{k = 1}^{\infty} \frac{(C T \e h^{2} \abs{\log h})^{k} }{\sqrt{2 \pi k}} \right)^{\hatgamma_{1} h} \left( 1 + \sum_{k = 1}^{\infty} \frac{(C T \e h^{2} \abs{\log h})^{k}}{\sqrt{2 \pi k}} \right)^{\hatgamma_{1} h} \leq 2^{2 \hatgamma_{1} h},
            \end{align*}
            for all $T \leq C_{0} h^{-2} \abs{\log h}^{-1}$. Combining all the above estimates, $\hatI_{1}$ yields the upper bound
            \begin{equation*}
                \begin{split}
                    \hatI_{1} &\leq \Big( \nu_{2}^{q} \nu_{1}^{-q} q^{q/2} h^{q/2} \abs{\log h}^{-q} \Big)^{\frac{h}{\abs{\log h}}} \E \Bigg[ \exp \bigg( 8 p^{2} S^{2} T \frac{\abs{\log h}}{\abs{\log h} - h} \norm{\hatY_{2n + 2}^{f} - \hatY_{2n + 2}^{c}}^{2} \bigg) \Bigg]^{1 - \frac{h}{\abs{\log h}}} \\
                    ~
                    &\leq \Big( 2^{2 \hatgamma_{1} h} \exp(CTh^{2} \abs{\log h}) \Big)^{1 - \frac{h}{\abs{\log h}}} \Big( \nu_{2}^{q} \nu_{1}^{-q} q^{q/2} h^{q/2} \abs{\log h}^{-q} \Big)^{\frac{h}{\abs{\log h}}} \times \\
                    &\qquad \qquad \E \Bigg[ \exp \Bigg( 8 p^{2} S^{2} T \norm{\hatY_{2n}^{f} - \hatY_{2n}^{c}}^{2} \Bigg) \Bigg]^{(1 - 2 \hatgamma_{1} h) \Big(1 - \frac{h}{\abs{\log h}} \Big)}.
                \end{split}
            \end{equation*}
            Note that $\exp(x^2) \geq 1, \; \forall x \in \bR$. For $q = h^{-1} \abs{\log h}^{2}$, for all sufficiently large $\nu_{1} > \nu_{2} \e$, for all $T \leq C_{0} h^{-2} \abs{\log h}^{-1}$, and for all sufficiently small $h > 0$, we have
            \begin{equation*}
                \begin{split}
                    \hatI_{1} &\leq \Big( 2^{2 \hatgamma_{1} h} \exp(C) \Big)^{1 - \frac{h}{\abs{\log h}}} \Bigg( \frac{\nu_{2}}{\nu_{1}} \Bigg)^{\abs{\log h}} \E \Bigg[ \exp \Bigg( 8 p^{2} S^{2} T \norm{\hatY_{2n}^{f} - \hatY_{2n}^{c}}^{2} \Bigg) \Bigg]^{1 - 2 \hatgamma_{1} h} \\
                    ~
                    &\leq \hatC{1} h \Big( \hatJ_{2n} \Big)^{1 - 2 \hatgamma_{1} h},
                \end{split}
            \end{equation*}
            where $\hatC{1} \in (0, 1)$ is a constant independent of $T$ and $h$.
            
            We now bound the term $\hatI_{2}$ similarly. Using Assumptions~\ref{asmp:lipschitz_a} and \ref{asmp:aD_lipschitz}, and the mean value theorem, we have
            \begin{equation*}
                \begin{split}
                    a(\hatY_{2n + 1}^{f}) - a(\hatY_{2n}^{f}) = Da(\hatY_{2n}^{f}) (\hatY_{2n + 1}^{f} - \hatY_{2n}^{f}) + R_{2},
                \end{split}
            \end{equation*}
            where $\norm{R_{2}} \leq d^{3/2} K_{b} \norm{\hatY_{2n + 1}^{f} - \hatY_{2n}^{f}}^{2}$.
            
            Using equation~\eqref{eq:method1_diff_fc_1}, H\"older's inequality, and Young's inequality, there exists $\gamma_{2} \in (0, 2S - \lambda)$, for all sufficiently small $h > 0$, such that
            \begin{align*}
                \allowdisplaybreaks
                \norm{\hatY_{2n + 2}^{f} - \hatY_{2n + 2}^{c}}^{2} &\leq (1 - 4 (2S - \lambda) h + Ch \abs{\log h}^{-1} ) \norm{\hatY_{2n}^{f} - \hatY_{2n}^{c}}^{2} + C \bigg\{ h^{3} \abs{\log h} \bigg( 1 + \norm{\hatY_{2n}^{f}}^{2} \\
                &\quad + \norm{\hatY_{2n}^{c}}^{2} \bigg) + h^{2} \norm{\Delta W_{2n}}^{2} + \norm{\Delta Z_{2n + 1}}^{2} + \norm{\Delta Z_{2n}}^{2} \bigg\} \\
                &\quad + 2 (1 - Sh) (1 - 4Sh + 2S^{2}h^{2}) \inner{\hatY_{2n}^{f} - \hatY_{2n}^{c}, (Da(\hatY_{2n}^{f}) - Da(\hatY_{2n}^{c})) \Delta Z_{2n}} \\
                &\quad + 2 h (1 - 4Sh + 2S^{2}h^{2}) \inner{\hatY_{2n}^{f} - \hatY_{2n}^{c}, R_{2}} \\
                &\quad - h^{2} (1 - 4Sh + 2S^{2}h^{2}) \inner{\hatY_{2n}^{f} - \hatY_{2n}^{c}, \Delta a(\hatY_{2n}^{f})} \\
                &\quad + 2 (1 - 4Sh + 2S^{2}h^{2}) \inner{\hatY_{2n}^{f} - \hatY_{2n}^{c}, (Da(\hatY_{2n + 1}^{f}) - Da(
                \hatY_{2n}^{f})) \Delta Z_{2n + 1}} \\
                &\quad + 2 (1 - 4Sh + 2S^{2}h^{2}) \inner{\hatY_{2n}^{f} - \hatY_{2n}^{c}, (Da(\hatY_{2n}^{f}) - Da(
                \hatY_{2n}^{c})) \Delta Z_{2n + 1}} \\
                &\quad + 2 h (1 - 4Sh + 2S^{2}h^{2}) \inner{\hatY_{2n}^{f} - \hatY_{2n}^{c}, (Da(\hatY_{2n}^{f}) - Da(\hatY_{2n}^{c})) \Delta W_{2n}} \\
                ~
                &\leq (1 - 4 \gamma_{2} h) \norm{\hatY_{2n}^{f} - \hatY_{2n}^{c}}^{2} + C \bigg\{ h^{3} \abs{\log h} \bigg( 1 + \norm{\hatY_{2n}^{f}}^{2} + \norm{\hatY_{2n}^{c}}^{2} \bigg) \\
                &\quad + h^{2} \norm{\Delta W_{2n}}^{2} + \norm{\Delta Z_{2n + 1}}^{2} + \norm{\Delta Z_{2n}}^{2} \bigg\} \\
                &\quad + 2 (1 - Sh) (1 - 4Sh + 2S^{2}h^{2}) \inner{\hatY_{2n}^{f} - \hatY_{2n}^{c}, (Da(\hatY_{2n}^{f}) - Da(\hatY_{2n}^{c})) \Delta Z_{2n}} \\
                &\quad + C h \norm{\hatY_{2n}^{f} - \hatY_{2n}^{c}} \bigg( h^{2} \Big( 1 + \norm{\hatY_{2n}^{f}}^{2} + \norm{\hatY_{2n}^{c}}^{2} \Big) + \norm{\Delta W_{2n}}^{2} + \norm{\Delta Z_{2n}}^{2} \bigg) \\
                &\quad + 2 (1 - 4Sh + 2S^{2}h^{2}) \inner{\hatY_{2n}^{f} - \hatY_{2n}^{c}, (Da(\hatY_{2n + 1}^{f}) - Da(\hatY_{2n}^{f})) \Delta Z_{2n + 1}} \\
                &\quad + 2 (1 - 4Sh + 2S^{2}h^{2}) \inner{\hatY_{2n}^{f} - \hatY_{2n}^{c}, (Da(\hatY_{2n}^{f}) - Da(\hatY_{2n}^{c})) \Delta Z_{2n + 1}} \\
                &\quad + 2 h (1 - 4Sh + 2S^{2}h^{2}) \inner{\hatY_{2n}^{f} - \hatY_{2n}^{c}, (Da(\hatY_{2n}^{f}) - Da(\hatY_{2n}^{c})) \Delta W_{2n}}.
            \end{align*}
            Using the above inequality, equation~\eqref{eq:set_theta}, and Young's inequality, the term $\hatI_{2}$ yields
            \begin{align*}
                \hatI_{2} &= \E \Bigg[ \ind_{\Theta_{2n}^{\comp}} \exp \Bigg( 8 p^{2} S^{2} T \norm{\hatY_{2n + 2}^{f} - \hatY_{2n + 2}^{c}}^{2} \Bigg) \Bigg] \\
                ~
                &\leq \E \Bigg[ \big( \ind_{\Theta_{2n}^{\comp}} \big)^{2} \exp \Bigg( 8 p^{2} S^{2} T \bigg( (1 - 4 \gamma_{2} h) \norm{\hatY_{2n}^{f} - \hatY_{2n}^{c}}^{2} \\
                &\quad + C \bigg\{ h^{3} \abs{\log h} \big( 1 + \norm{\hatY_{2n}^{f}}^{2} + \norm{\hatY_{2n}^{c}}^{2} \big) + h^{2} \norm{\Delta W_{2n}}^{2} + \norm{\Delta Z_{2n + 1}}^{2} + \norm{\Delta Z_{2n}}^{2} \bigg\} \\
                &\quad + 2 (1 - Sh) (1 - 4Sh + 2S^{2}h^{2}) \inner{\hatY_{2n}^{f} - \hatY_{2n}^{c}, (Da(\hatY_{2n}^{f}) - Da(\hatY_{2n}^{c})) \Delta Z_{2n}} \\
                &\quad + C h \norm{\hatY_{2n}^{f} - \hatY_{2n}^{c}} \bigg\{ h^{2} \Big( 1 + \norm{\hatY_{2n}^{f}}^{2} + \norm{\hatY_{2n}^{c}}^{2} \Big) + \norm{\Delta W_{2n}}^{2} + \norm{\Delta Z_{2n}}^{2} \bigg\} \\
                &\quad + 2 (1 - 4Sh + 2S^{2}h^{2}) \inner{\hatY_{2n}^{f} - \hatY_{2n}^{c}, (Da(\hatY_{2n + 1}^{f}) - Da(
                \hatY_{2n}^{f})) \Delta Z_{2n + 1}} \\
                &\quad + 2 (1 - 4Sh + 2S^{2}h^{2}) \inner{\hatY_{2n}^{f} - \hatY_{2n}^{c}, (Da(\hatY_{2n}^{f}) - Da(
                \hatY_{2n}^{c})) \Delta Z_{2n + 1}} \\
                &\quad + 2 h (1 - 4Sh + 2S^{2}h^{2}) \inner{\hatY_{2n}^{f} - \hatY_{2n}^{c}, (Da(\hatY_{2n}^{f}) - Da(\hatY_{2n}^{c})) \Delta W_{2n}} \bigg) \Bigg) \Bigg].
            \end{align*}
            Following a similar argument as before, using H\"older's inequality conditioned on the filtration $\cF_{t_{2n + 1}}$ generated until time $t_{2n + 1}$, using equation~\eqref{eq:mgf_chi_squared}, and for all sufficiently small $h > 0$, we obtain
            \begin{align*}
                \hatI_{2} &\leq \E \Bigg[ \exp \Bigg( 8 p^{2} S^{2} T \bigg( (1 - 4 \gamma_{2} h) \norm{\hatY_{2n}^{f} - \hatY_{2n}^{c}}^{2} \\
                &\quad + C \bigg\{ h^{3} \abs{\log h} \Big( 1 + \norm{\hatY_{2n}^{f}}^{2} + \norm{\hatY_{2n}^{c}}^{2} \Big) + h^{2} \norm{\Delta W_{2n}}^{2} + \norm{\Delta Z_{2n}}^{2} \bigg\} \\
                &\quad + 2 (1 - Sh) (1 - 4Sh + 2S^{2}h^{2}) \inner{\hatY_{2n}^{f} - \hatY_{2n}^{c}, (Da(\hatY_{2n}^{f}) - Da(\hatY_{2n}^{c})) \Delta Z_{2n}} \\
                &\quad + 2 h (1 - 4Sh + 2S^{2}h^{2}) \inner{\hatY_{2n}^{f} - \hatY_{2n}^{c}, (Da(\hatY_{2n}^{f}) - Da(\hatY_{2n}^{c})) \Delta W_{2n}} \bigg) \Bigg) \times \\
                &\qquad \exp \Bigg( C \nu_{1} T h^{3} \abs{\log h} \Big(1 + \norm{\hatY_{2n}^{f}}^{2} + \norm{\hatY_{2n}^{c}}^{2} \Big) + C \nu_{1} T h \abs{\log h} \norm{\Delta Z_{2n}}^{2} \\
                &\qquad \qquad + C T h \norm{\hatY_{2n}^{f} - \hatY_{2n}^{c}} \norm{\Delta W_{2n}}^{2} \Bigg) \times \exp \Bigg( C T h^{3} + C T^{2} h^{3} \norm{\hatY_{2n}^{f} - \hatY_{2n}^{c}}^{2} \Bigg) \Bigg]. \\
            \end{align*}
            Similarly, using the tower property, taking the expectation conditioned on the filtration $\cF_{t_{2n}}$ generated upto time $t_{2n}$, we obtain
            \begin{align*}
                \hatI_{2} &\leq \E \Bigg[ \exp \Bigg( 8 p^{2} S^{2} T \bigg( (1 - 4 \gamma_{2} h) \norm{\hatY_{2n}^{f} - \hatY_{2n}^{c}}^{2} + C h^{3} \abs{\log h} \Big( 1 + \norm{\hatY_{2n}^{f}}^{2} + \norm{\hatY_{2n}^{c}}^{2} \Big) \bigg) \Bigg) \\
                &\qquad \qquad \times \exp \Bigg( C \nu_{1} T h^{3} \abs{\log h} \Big(1 + \norm{\hatY_{2n}^{f}}^{2} + \norm{\hatY_{2n}^{c}}^{2} \Big) \Bigg) \\
                &\qquad \qquad \times \exp \Bigg( C T h^{3} + C T^{2} h^{3} \norm{\hatY_{2n}^{f} - \hatY_{2n}^{c}}^{2} \Bigg) \times \tildeJ_{3} \Bigg],
            \end{align*}
            where $\tildeJ_{3}$ is defined as
            \begin{equation*}
                \begin{split}
                    \tildeJ_{3} &\coloneqq \E \Bigg[ \exp \Bigg( C \nu_{1} T h \abs{\log h} \norm{\Delta Z_{2n}}^{2} + C T h \norm{\hatY_{2n}^{f} - \hatY_{2n}^{c}} \norm{\Delta W_{2n}}^{2} \Bigg) \\
                    &\quad \times \exp \Bigg( 8 p^{2} S^{2} T \bigg( C \bigg\{ h^{2} \norm{\Delta W_{2n}}^{2} + \norm{\Delta Z_{2n}}^{2} \bigg\} \\
                    &\quad + 2 (1 - Sh) (1 - 4Sh + 2S^{2}h^{2}) \inner{\hatY_{2n}^{f} - \hatY_{2n}^{c}, (Da(\hatY_{2n}^{f}) - Da(\hatY_{2n}^{c})) \Delta Z_{2n}} \\
                    &\quad + 2 h (1 - 4Sh + 2S^{2}h^{2}) \inner{\hatY_{2n}^{f} - \hatY_{2n}^{c}, (Da(\hatY_{2n}^{f}) - Da(\hatY_{2n}^{c})) \Delta W_{2n}} \bigg) \Bigg) \; \bigg\vert \; \cF_{t_{2n}} \Bigg].
                \end{split}
            \end{equation*}
            Using H\"older's inequality, we obtain
            \begin{align*}
                    \tildeJ_{3} &\leq \E \Bigg[ \exp \Bigg( C \nu_{1} T h \abs{\log h} \norm{\Delta Z_{2n}}^{2} \Bigg) \; \bigg\vert \; \cF_{t_{2n}} \Bigg]^{1/6} \\
                    &\quad \times \underbrace{\E \Bigg[ \exp \Bigg( C T h \norm{\hatY_{2n}^{f} - \hatY_{2n}^{c}} \norm{\Delta W_{2n}}^{2} \Bigg) \; \bigg\vert \; \cF_{t_{2n}} \Bigg]^{1/6}}_{\tildeJ_{4}} \\
                    &\quad \times \E \Bigg[ \exp \Bigg( C T h^{2} \norm{\Delta W_{2n}}^{2} \Bigg) \; \bigg\vert \; \cF_{t_{2n}} \Bigg]^{1/6} \times \E \Bigg[ \exp \Bigg( C T \norm{\Delta Z_{2n}}^{2} \Bigg) \; \bigg\vert \; \cF_{t_{2n}} \Bigg]^{1/6} \\
                    &\quad \E \Bigg[ \exp \Bigg( 96 p^{2} S^{2} T (1 - Sh) (1 - 4Sh + 2S^{2}h^{2}) \\
                    &\qquad \qquad \inner{\hatY_{2n}^{f} - \hatY_{2n}^{c}, (Da(\hatY_{2n}^{f}) - Da(\hatY_{2n}^{c})) \Delta Z_{2n}} \Bigg) \; \bigg\vert \; \cF_{t_{2n}} \Bigg]^{1/6} \\
                    &\quad \E \Bigg[ \exp \Bigg( 96 p^{2} S^{2} T h (1 - 4Sh + 2S^{2}h^{2}) \\
                    &\qquad \qquad \inner{\hatY_{2n}^{f} - \hatY_{2n}^{c}, (Da(\hatY_{2n}^{f}) - Da(\hatY_{2n}^{c})) \Delta W_{2n}} \bigg) \Bigg) \; \bigg\vert \; \cF_{t_{2n}} \Bigg]^{1/6}.
            \end{align*}
            Note that for all $T \leq C_{0} h^{-2} \abs{\log h}^{-1}$, and for all sufficiently small $h > 0$, using Young's inequality, we obtain
            \begin{equation*}
                \begin{split}
                    \tildeJ_{4} &\leq \exp \Bigg( C T h^{2} \norm{\hatY_{2n}^{f} - \hatY_{2n}^{c}} \Bigg) \leq \exp \Bigg( C T^{2} h^{3} \norm{\hatY_{2n}^{f} - \hatY_{2n}^{c}}^{2} + C h \Bigg).
                \end{split}
            \end{equation*}
            Combining all the above estimates, using equation~\eqref{eq:mgf_chi_squared}, we obtain
            \begin{equation*}
                \begin{split}
                    \tildeJ_{3} \leq \exp \Bigg( C \nu_{1} T h^{4} \abs{\log h} + C T h^{3} + C h + C T^{2} h^{3} \norm{\hatY_{2n}^{f} - \hatY_{2n}^{c}}^{2} \Bigg).
                \end{split}
            \end{equation*}
            Note that for all $1 \leq T \leq \gamma_{2} h^{-2} / (2 \max(1, \; C))$, and all sufficiently small $h > 0$, there exists a $\hatgamma_{2} > 0$ such that the following inequality holds
            \begin{equation*}
                \begin{split}
                    \hatI_{2} &\leq \E \Bigg[ \exp \Bigg( 8 p^{2} S^{2} T \bigg( (1 - 4 \gamma_{2} h) \norm{\hatY_{2n}^{f} - \hatY_{2n}^{c}}^{2} + C h^{3} \abs{\log h} \Big( 1 + \norm{\hatY_{2n}^{f}}^{2} + \norm{\hatY_{2n}^{c}}^{2} \Big) \bigg) \Bigg) \\
                    &\quad \exp \Bigg( C \nu_{1} T h^{4} \abs{\log h} + C T h^{3} + C h \Bigg) \exp \Bigg( C \nu_{1} T h^{3} \abs{\log h} \Big(1 + \norm{\hatY_{2n}^{f}}^{2} + \norm{\hatY_{2n}^{c}}^{2} \Big) \Bigg) \Bigg] \\
                    ~
                    &\leq \exp \Big( C \nu_{1} \big( T h^{3} \abs{\log h} + h \big) \Big) \E \Bigg[ \exp \Bigg( 8 p^{2} S^{2} T (1 - 4 \hatgamma_{2} h) \norm{\hatY_{2n}^{f} - \hatY_{2n}^{c}}^{2} \\
                    &\qquad \qquad \qquad \qquad \qquad + C \nu_{1} T h^{3} \abs{\log h} \Big( \norm{\hatY_{2n}^{f}}^{2} + \norm{\hatY_{2n}^{c}}^{2} \Big) \Bigg) \Bigg].
                \end{split}
            \end{equation*}
            For all $T \leq C_{0} h^{-2} \abs{\log h}^{-1}$, and for all sufficiently small $h > 0$, using H\"older's inequality, we have
            \begin{equation*}
                \begin{split}
                    \hatI_{2} &\leq \exp \big( C \nu_{1} h \big) \E \Bigg[ \exp \Bigg( 8 p^{2} S^{2} T \norm{\hatY_{2n}^{f} - \hatY_{2n}^{c}}^{2} \Bigg) \Bigg]^{1 - 4 \hatgamma_{2} h} \times \\
                    & \qquad \qquad \E \Bigg[ \exp \Bigg( C \nu_{1} T h^{2} \abs{\log h} \Big( \norm{\hatY_{2n}^{f}}^{2} + \norm{\hatY_{2n}^{c}}^{2} \Big) \Bigg) \Bigg]^{4 \hatgamma_{2} h} \\
                    ~
                    &\leq 2^{4 \hatgamma_{2} h} \exp ( \hatC{2} \nu_{1} h ) \Big( \hatJ_{2n} \Big)^{1 - 4 \hatgamma_{2} h},
                \end{split}
            \end{equation*}
            where $\hatC{2} > 0$ is a constant that is independent of $T$ and $h$. Combining the terms $\hatI_{1}$ and $\hatI_{2}$, we obtain
            \begin{equation*}
                \begin{split}
                    \hatJ_{2n + 2} \leq \hatC{1} h \Big( \hatJ_{2n} \Big)^{1 - 2 \hatgamma_{1} h} + 2^{4 \hatgamma_{2} h} \exp \big( \hatC{2} \nu_{1} h \big) \Big( \hatJ_{2n} \Big)^{1 - 4 \hatgamma_{2} h}.
                \end{split}
            \end{equation*}
            Note that for $\tildegamma \coloneqq \min \big( 2 \hatgamma_{1}, 4 \hatgamma_{2} \big)$, the following inequality holds
            \begin{equation*}
                \begin{split}
                    \hatJ_{2n + 2} \leq \Big( \hatC{1} h + 2^{4 \hatgamma_{2} h} \exp \big( \hatC{2} \nu_{1} h \big) \Big) \Big( \hatJ_{2n} \Big)^{1 - \tildegamma h},
                \end{split}
            \end{equation*}
            where $\hatJ_{2n}$ is given by equation~\eqref{eq:hatJ_terms_odd_even}. By recursively bounding the term $\hatJ_{2n}$ using a similar argument as above, by defining the sets $\Theta_{2k}$ for all $k = 0, \ldots, n - 1$, we obtain
            \begin{equation*}
                \begin{split}
                    \hatJ_{2n + 2} &\leq \Big( \hatC{1} h + 2^{4 \hatgamma_{2} h} \exp \big( \hatC{2} \nu_{1} h \big) \Big)^{\sum_{k = 0}^{n} (1 - \tildegamma h)^{k}} \leq \Big( \hatC{1} h + 2^{4 \hatgamma_{2} h} \exp \big( \hatC{2} \nu_{1} h \big) \Big)^{\frac{1}{\tildegamma  h}} \\
                    ~
                    &\leq 2^{4 \hatgamma_{2} / \tildegamma} \exp \big( \hatC{2} \nu_{1} / \tildegamma \big) \Big( 2^{- 4 \hatgamma_{2} h} \hatC{1} h \exp \big( - \hatC{2} \nu_{1} h \big) + 1 \Big)^{\frac{1}{\tildegamma  h}}.
                \end{split}
            \end{equation*}
            Note that $\e^{-cx} \leq 2^{-cx} \leq 1 - cx/2$ for all $x \in [0, 1]$. Using this, we obtain
            \begin{equation*}
                \begin{split}
                    \hatJ_{2n + 2} &\leq 2^{4 \hatgamma_{2} / \tildegamma} \exp \big( \hatC{2} \nu_{1} / \tildegamma \big) \Big( 2^{- 4 \hatgamma_{2} h} \hatC{1} h \exp \big( - \hatC{2} \nu_{1} h \big) + 1 \Big)^{\frac{1}{\tildegamma  h}} \\
                    ~
                    &\leq 2^{4 \hatgamma_{2} / \tildegamma} \exp \big( \hatC{2} \nu_{1} / \tildegamma \big) \Big( \hatC{1} h \Big( 1 - 2 \hatgamma_{2} h \Big) \Big( 1 - \frac{\hatC{2} \nu_{1}}{2} h \Big) + 1 \Big)^{\frac{1}{\tildegamma  h}} \\
                    ~
                    &\leq 2^{4 \hatgamma_{2} / \tildegamma} \exp \big( \hatC{2} \nu_{1} / \tildegamma \big) \Big( 1 + \hatC{3} h \Big)^{\frac{1}{\tildegamma  h}} \\
                    ~
                    &\leq 2^{4 \hatgamma_{2} / \tildegamma} \exp \big( \hatC{2} \nu_{1} / \tildegamma \big) \exp \big( \hatC{3} h \Big)^{\frac{1}{\tildegamma  h}} \leq \eta_{1},
                \end{split}
            \end{equation*}
            where $\eta_{1} > 0$ is a constant that is independent of $T$ and $h$, and $\hatC{3} > 0$ is a constant that is independent of $h$.

            Using Assumptions~\ref{asmp:lipschitz_a}, \ref{asmp:dissipativity_a}, \ref{asmp:aD_lipschitz}, and equations~\eqref{eq:fine_odd} and \eqref{eq:coarse_odd}, and Young's inequality, there exists a $\gamma \in (0, 2S - \lambda)$ such that the following inequality holds
            \begin{equation*}
                \begin{split}
                    \hatJ_{2n + 1} &= \E \Bigg[ \exp \Bigg( 8 p^{2} S^{2} T \norm{\hatY_{2n + 1}^{f} - \hatY_{2n + 1}^{c}}^{2} \Bigg) \Bigg] \\
                    ~
                    &\leq \E \Bigg[ \exp \Bigg( 8 p^{2} S^{2} T \bigg( (1 - 2 \gamma h) \norm{\hatY_{2n}^{f} - \hatY_{2n}^{c}}^{2} \bigg) + C \norm{\Delta Z_{2n}}^{2} \\
                    &\qquad + 2 (1 - 2Sh) \inner{\hatY_{2n}^{f} - \hatY_{2n}^{c}, (Da(\hatY_{2n}^{f}) - Da(\hatY_{2n}^{c})) \Delta Z_{2n}} \Bigg) \Bigg].
                \end{split}
            \end{equation*}
            Using H\"older's inequality conditioned on the filtration $\cF_{t_{2n}}$ generated until time $t_{2n}$, and using a similar argument as before, the following inequality holds
            \begin{equation*}
                \begin{split}
                    \hatJ_{2n + 1} &\leq \E \Bigg[ \exp \Bigg( 8 p^{2} S^{2} T (1 - 2 \gamma h) \norm{\hatY_{2n}^{f} - \hatY_{2n}^{c}}^{2} \Bigg) \exp \Bigg( C T h^{3} + C T^{2} h^{3} \norm{\hatY_{2n}^{f} - \hatY_{2n}^{c}}^{2} \Bigg) \Bigg) \Bigg].
                \end{split}
            \end{equation*}
            For all $1 \leq T \leq \gamma h^{-2} / (2 \max (1, C) )$, and all sufficiently small $h > 0$, there exists $\hatgamma > 0$ such that, using H\"older's inequality, the following inequality holds
            \begin{equation*}
                \begin{split}
                    \hatJ_{2n + 1} &\leq \E \Bigg[ \exp \Bigg( 8 p^{2} S^{2} T (1 - 2 \hatgamma h) \norm{\hatY_{2n}^{f} - \hatY_{2n}^{c}}^{2} + C T h^{3} \Bigg) \Bigg] \\
                    &\leq \exp ( C T h^{2} ) \E \Bigg[ \exp \Bigg( 8 p^{2} S^{2} T \norm{\hatY_{2n}^{f} - \hatY_{2n}^{c}}^{2} \Bigg) \Bigg]^{1 - 2 \hatgamma h}.
                \end{split}
            \end{equation*}
            For all $T \leq C_{0} h^{-2} \abs{\log h}^{-1}$, and all sufficiently small $h > 0$, we then obtain
            \begin{equation*}
                \begin{split}
                    \hatJ_{2n + 1} &\leq \exp ( C T h^{2} ) \eta_{1}^{1 - 2 \hatgamma h} \leq \eta_{2},
                \end{split}
            \end{equation*}
            where $\eta_{2} > 0$ is a constant that is independent of $T$ and $h$.
        \endgroup

    \subsubsection*{Proof of Theorem~\ref{thm:convergence_estimator}}
        \allowdisplaybreaks
        \begingroup
            Following a similar argument as in \cite{fang2019multilevel}, using H\"older's inequality, and Jensen's inequality, we have
            \begin{equation*}
                \begin{split}
                    &\E \Bigg[ \abs{\varPhi(\hatY_{T}^{f}) R_{T}^{f} - \varPhi(\hatY_{T}^{c}) R_{T}^{c}}^{p} \Bigg] \\
                    &\quad = \E \Bigg[ \Bigg\vert \varPhi(\hatY_{T}^{f}) R_{T}^{f} - \varPhi(\hatY_{T}^{f}) + \varPhi(\hatY_{T}^{f}) - \varPhi(\hatY_{T}^{c}) + \varPhi(\hatY_{T}^{c}) - \varPhi(\hatY_{T}^{c}) R_{T}^{c} \Bigg\vert^{p} \Bigg] \\
                    ~
                    &\quad \leq 3^{p - 1} \Bigg\{ \underbrace{\E \bigg[ \abs{\varPhi(\hatY_{T}^{f}) - \varPhi(\hatY_{T}^{c})}^{p} \bigg]}_{\tildeJ_{1}} + \underbrace{\E \bigg[ \abs{\varPhi(\hatY_{T}^{f})}^{2p} \bigg]^{1/2} \E \bigg[ \abs{1 - R_{T}^{f}}^{2p} \bigg]^{1/2}}_{\tildeJ_{2}} \\
                    &\qquad \qquad \qquad + \underbrace{\E \bigg[ \abs{\varPhi(\hatY_{T}^{c})}^{2p} \bigg]^{1/2} \E \bigg[ \abs{1 - R_{T}^{c}}^{2p} \bigg]^{1/2}}_{\tildeJ_{3}} \Bigg\}.
                \end{split}
            \end{equation*}

            From Assumption~\ref{asmp:lipschitz_payoff}, and using the results from Theorem~\ref{thm:convergence}, the term $\tildeJ_{1}$ yields the bound
            \begin{equation*}
                \begin{split}
                    \tildeJ_{1} = \E \bigg[ \abs{\varPhi(\hatY_{T}^{f}) - \varPhi(\hatY_{T}^{c})}^{p} \bigg] \leq K^{p} \E \bigg[ \norm{\hatY_{T}^{f} - \hatY_{T}^{c}}^{p} \bigg] \leq C^{p} p^{3p/2} h^{3p/2}.
                \end{split}
            \end{equation*}
            Similarly, we have
            \begin{equation*}
                \begin{split}
                    \E \bigg[ \abs{\varPhi(\hatY_{T}^{f})}^{2p} \bigg]^{1/2} \leq C^{p} p^{p/2}, \qquad \E \bigg[ \abs{\varPhi(\hatY_{T}^{c})}^{2p} \bigg]^{1/2} \leq C^{p} p^{p/2}.
                \end{split}
            \end{equation*}
            
            To bound the term $\tildeJ_{2}$, we define $E^{f}$ as
            \begin{equation*}
                E^{f} \coloneqq - \sum_{n = 0}^{N - 1} \inner{\hatSf_{n}, \Delta V_{1, n}} + \sqrt{3} \sum_{n = 0}^{N - 1} \inner{\hatSf_{n}, \Delta V_{2, n}} - 2 h \sum_{n = 0}^{N - 1} \norm{\hatSf_{n}}^{2},
            \end{equation*}
            where $\hatSf_{n}$ is as defined in equation~\eqref{eq:hatSf}. Using the mean value theorem, there exists a $\xi = \theta x + (1 - \theta)y$ for all $x, y \in \bR$, and any $\theta \in (0, 1)$, such that
            \begin{equation*}
                R_{T}^{f} = \exp (E^{f}) = 1 + \exp(\xi) E^{f},
            \end{equation*}
            where $\abs{\xi} \leq \abs{x}$ for all $\theta \in (0, 1)$. Using H\"older's inequality, we obtain
            \begin{equation*}
                \begin{split}
                    \E \bigg[ \abs{1 - R_{T}^{f}}^{2p} \bigg] &= \E \bigg[ \big( \exp(\xi) \abs{E^{f}} \big)^{2p} \bigg] \leq \E \bigg[ \exp(4 p \xi) \bigg]^{1/2} \E \bigg[ \abs{E^{f}}^{4p} \bigg]^{1/2}.
                \end{split}
            \end{equation*}
            Using the result from Theorem~\ref{thm:radon_nikodym}, we have
            \[ \E \bigg[ \exp(4 p \xi) \bigg] \leq \E \bigg[ \max ( \exp(4 p E^{f}), 1 ) \bigg] \leq \E[ \exp(4 p E^{f}) ] + 1 \leq \eta_{1}, \] for all $T \leq C_{0} h^{-2} \abs{\log h}^{-1}$, where $C_{0}$ is a constant that is independent of $h$.

            From Jensen's inequality, and H\"older's inequality, we obtain
            \begin{equation*}
                \begin{split}
                    \E \bigg[ \abs{E^{f}}^{4p} \bigg] \leq 2^{4p - 1} S^{4p} \E \Bigg[ \abs{ \sum_{n = 0}^{N - 1} \inner{\hatY_{n}^{f} - \hatY_{n}^{c}, \Delta V_{3, n}} }^{4p} \Bigg] +  2^{8p - 1} S^{8p} h^{4p} \E \Bigg[ \abs{ \sum_{n = 0}^{N - 1} \norm{\hatY_{n}^{f} - \hatY_{n}^{c}}^{2} }^{4p} \Bigg],
                \end{split}
            \end{equation*}
            where $\Delta V_{3, n} \coloneqq \sqrt{3} \Delta V_{2, n} - \Delta V_{1, n}$. Using Theorem~\ref{thm:convergence}, and Burkholder--Davis--Gundy's inequality, there exists a constant $C_{\text{BDG}} > 0$ independent of $T$ and $p$, such that
            \begin{equation*}
                \begin{split}
                    \E \bigg[ \abs{E^{f}}^{4p} \bigg] &\leq 2^{12p - 1} S^{4p} C_{\text{BDG}}^{2p} p^{2p} h^{2p} \E \Bigg[ \abs{ \sum_{n = 0}^{N - 1} \norm{\hatY_{n}^{f} - \hatY_{n}^{c}}^{2}}^{2p} \Bigg] \\
                    &\qquad +  2^{8p - 1} S^{8p} h^{4p} \E \Bigg[ \abs{ \sum_{n = 0}^{N - 1} \norm{\hatY_{n}^{f} - \hatY_{n}^{c}}^{2} }^{4p} \Bigg] \\
                    ~
                    &\leq C_{3}^{6p} T^{2p} p^{8p} h^{6p} + C_{3}^{8p} T^{4p} p^{12p} h^{12p} \\
                    ~
                    &\leq C_{3}^{8p} T^{2p} p^{8p} h^{6p} \big( 1 + T^{2p} p^{4p} h^{6p} \big).
                \end{split}
            \end{equation*}
            Then, for all $T \leq C_{0} h^{-3}$, and for all sufficiently small $h > 0$, there exists a constant $C_{5} > 0$ independent of $T$ and $p$, such that
            \begin{equation*}
                \E \bigg[ \abs{E^{f}}^{4p} \bigg] \leq C_{5}^{8p} T^{2p} p^{8p} h^{6p},
            \end{equation*}
            which then yields
            \begin{equation*}
                \E \bigg[ \abs{1 - R_{T}^{f}}^{2p} \bigg] \leq \sqrt{3} C_{5}^{4p} T^{p} p^{4p} h^{3p}.
            \end{equation*}
            A similar result can obtained for the term $\E \bigg[ \abs{1 - R_{T}^{c}}^{2p} \bigg]$ to yield a bound for the term $\tildeJ_{3}$.
            
            Combining the estimates for the terms $\tildeJ_{1}$, $\tildeJ_{2}$, and $\tildeJ_{3}$, for all $T \leq C_{0} h^{-2} \abs{\log h}^{-1}$, $p \geq 1$, there exists a constant $\kappa_{7} > 0$, independent of $T$, $p$, for all sufficiently small $h > 0$, we obtain
            \begin{equation}
                \begin{split}
                    \E \Bigg[ \abs{ \varPhi(\hatY_{T}^{f}) R_{T}^{f} - \varPhi(\hatY_{T}^{c}) R_{T}^{c} }^{p} \Bigg] \leq \kappa_{7}^{p} T^{p/2} p^{\frac{5}{2}p} h^{\frac{3}{2}p}.
                \end{split}
            \end{equation}
        \endgroup

    \subsubsection*{Proof of Theorem~\ref{thm:convergence_estimator_discont}}
        \allowdisplaybreaks
        \begingroup
            Following the same steps as in Theorem~\ref{thm:convergence_estimator}, using H\"older's inequality, and Jensen's inequality, we have
            \begin{equation*}
                \begin{split}
                    &\E \Bigg[ \abs{\varPhi(\hatY_{T}^{f}) R_{T}^{f} - \varPhi(\hatY_{T}^{c}) R_{T}^{c}}^{p} \Bigg] \\
                    &\quad = \E \Bigg[ \Bigg\vert \varPhi(\hatY_{T}^{f}) R_{T}^{f} - \varPhi(\hatY_{T}^{f}) + \varPhi(\hatY_{T}^{f}) - \varPhi(\hatY_{T}^{c}) + \varPhi(\hatY_{T}^{c}) - \varPhi(\hatY_{T}^{c}) R_{T}^{c} \Bigg\vert^{p} \Bigg] \\
                    ~
                    &\quad \leq 3^{p - 1} \Bigg\{ \underbrace{\E \bigg[ \abs{\varPhi(\hatY_{T}^{f}) - \varPhi(\hatY_{T}^{c})}^{p} \bigg]}_{\tildeJ_{1}} + \underbrace{\E \bigg[ \abs{\varPhi(\hatY_{T}^{f})}^{2p} \bigg]^{1/2} \E \bigg[ \abs{1 - R_{T}^{f}}^{2p} \bigg]^{1/2}}_{\tildeJ_{2}} \\
                    &\qquad \qquad \qquad + \underbrace{\E \bigg[ \abs{\varPhi(\hatY_{T}^{c})}^{2p} \bigg]^{1/2} \E \bigg[ \abs{1 - R_{T}^{c}}^{2p} \bigg]^{1/2}}_{\tildeJ_{3}} \Bigg\}.
                \end{split}
            \end{equation*}
            The terms $\tildeJ_{2}$ and $\tildeJ_{3}$ can be bounded in a similar way as in Theorem~\ref{thm:convergence_estimator}. To bound $\tildeJ_{1}$, we begin the analysis by bound the term $\abs{\ind_{ \hatY_{T}^{f} \in G } - \ind_{ \hatY_{T}^{c} \in G }}$. Following a similar methodology as in \cite{giles2022multilevel}, we obtain
            \begin{equation*}
                \begin{split}
                    \abs{\ind_{ \hatY_{T}^{f} \in G } - \ind_{ \hatY_{T}^{c} \in G }} &\leq \ind_{ \norm{\hatY_{T}^{f} - \hatY_{T}^{c}} > d_{\partial G}(\hatY_{T}^{f}) } \\
                    &= \ind_{ \norm{\hatY_{T}^{f} - \hatY_{T}^{c}} > d_{\partial G}(\hatY_{T}^{f}) } \bigg( \ind_{ d_{\partial G}(\hatY_{T}^{f}) \leq \delta } + \ind_{ d_{\partial G}(\hatY_{T}^{f}) > \delta } \bigg) \\
                    &\leq \ind_{ d_{\partial G}(\hatY_{T}^{f}) \leq \delta } + \ind_{ \norm{\hatY_{T}^{f} - \hatY_{T}^{c}} > \delta }.
                \end{split}
            \end{equation*}
            Using H\"older's inequality, taking the expected value on both sides of the inequality, and using Markov's inequality, we obtain
            \begin{equation*}
                \begin{split}
                    \E \bigg[ \abs{\ind_{ \hatY_{T}^{f} \in G } - \ind_{ \hatY_{T}^{c} \in G }}^{p} \bigg] &\leq 2^{p - 1} \bigg( \Prob \Big( d_{\partial G}(\hatY_{T}^{f}) \leq \delta \Big) + \Prob \Big( \norm{\hatY_{T}^{f} - \hatY_{T}^{c}} > \delta \Big) \bigg) \\
                    ~
                    &\leq 2^{p - 1} \bigg( C_{b} \delta + \delta^{-q} \E \Big[ \norm{\hatY_{T}^{f} - \hatY_{T}^{c}}^{q} \Big] \bigg).
                \end{split}
            \end{equation*}
            Using the result from Theorem~\ref{thm:convergence}, for the choice of $\delta = q^{\frac{3q + 2}{2(q + 1)}} h^{\frac{3}{2} \frac{q}{q + 1}}$, $\tildeJ_{1}$ yields the bound
            \begin{equation*}
                \begin{split}
                    \tildeJ_{1} = \E \bigg[ \abs{\varPhi(\hatY_{T}^{f}) - \varPhi(\hatY_{T}^{c})}^{p} \bigg] &= \E \bigg[\abs{\ind_{ \hatY_{T}^{f} \in G } - \ind_{ \hatY_{T}^{c} \in G }}^{p} \bigg] \\
                    ~
                    &\leq C^{p} \bigg( \delta + \delta^{-q} q^{3q/2} h^{3q/2} \bigg) \\
                    ~
                    &\leq C^{p} (1 + q) q^{\frac{1}{2} \frac{q}{q + 1}} h^{\frac{3}{2} \frac{q}{q + 1}}.
                \end{split}
            \end{equation*}
            Note that for all $h \in (0, 1)$, and for all $\xi \geq \frac{3}{2} \frac{1}{q + 1}$, $q \geq 1$, the following inequality holds true
            \[ \tildeJ \leq C^{p} (1 + q) q^{\frac{1}{2} \frac{q}{q + 1}} h^{\frac{3}{2} \frac{q}{q + 1}} \leq C^{p} (1 + q) q^{\frac{1}{2} \frac{q}{q + 1}} h^{\frac{3}{2} - \xi}. \]
            
            Combining the estimates for the terms $\tildeJ_{1}$, $\tildeJ_{2}$, and $\tildeJ_{3}$, for all $T > 1$, $p, q \geq 1$, and for all $\xi \geq \frac{3}{2} \frac{1}{q + 1}$ there exists a constant $\kappa_{8} > 0$, independent of $T$, $p$, and $q$, for all sufficiently small $h > 0$, we obtain
            \begin{equation}
                \E \bigg[ \abs{ \varPhi(\hatY_{T}^{f}) R_{T}^{f} - \varPhi(\hatY_{T}^{c}) R_{T}^{c} }^{p} \bigg] \leq \kappa_{8}^{p} T^{p/2} p^{2p} (1 + q) q^{\frac{1}{2} \frac{q}{q + 1}} h^{\frac{3}{2} - \xi}.
            \end{equation}
        \endgroup

    \subsubsection*{Proof of Theorem~\ref{thm:mlmc_cost}}
        \allowdisplaybreaks
        \begingroup
            The MSE, using H\"older's inequality, can be split into three components based on the source of error
            \begin{equation*}
                \begin{split}
                     \E \Big[ \big( \hatphi - \E[\varPhi(X_{\infty})] \big)^{2} \Big] &= \E \Big[ \big( ( \hatphi - \E[\hatphi] ) + ( \E[\hatphi] - \E[\varPhi(X_{T})] ) + ( \E[\varPhi(X_{T})] - \E[\varPhi(X_{\infty})] ) \big)^{2} \Big] \\
                     ~
                     &\leq \underbrace{\Var(\hatphi)}_{\tildeJ_{1}} + \underbrace{2 \big(\E[\hatphi] - \E[\varPhi(X_{T})]\big)^2}_{\tildeJ_{2}} + \underbrace{2 \big(\E[\varPhi(X_{T})] - \E[\varPhi(X_{\infty})]\big)^{2}}_{\tildeJ_{3}},
                \end{split}
            \end{equation*}
            where $\Var(\hatphi)$ denotes the variance of the MLMC estimator. Recall that we want the MSE to be bounded from above by $\epsilon^{2}$. This can be achieved by splitting the error contribution equally among the three error components, i.e. \[ \tildeJ_{k} \leq \frac{\epsilon^{2}}{3}, \qquad \forall k \in \{ 1, 2, 3 \}. \]
            From equation~\eqref{eq:ergodic_geomtric}, we know that 
            \begin{equation*}
                \begin{split}
                    \abs{\E[\varPhi(X_{T})] - \E[\varPhi(X_{\infty})]} \leq \mu^{\star} \e^{- \lambda^{\star} T}.
                \end{split}
            \end{equation*}
            Then for the choice of
            \begin{equation*}
                T = \ceil{\frac{1}{\lambda^{\star}} \log (\epsilon^{-1}) + \frac{1}{\lambda^{\star}} \log(\sqrt{6} \mu^{\star})},
            \end{equation*}
            we achieve the bound $\tildeJ_{3} \leq \epsilon^{2} / 3$.

            We consider two different cases based on the regularity of the payoff function \\[0.2cm]
            \textbf{Lipschitz continuous payoff function:} 
            Using Jensen's inequality and the results from Theorem~\ref{thm:convergence_estimator}, we have \[ \abs{\E \big[ \varPhi(\hatY_{T}^{f, \ell}) R_{T}^{f, \ell} - \varPhi(\hatY_{T}^{c, \ell}) R_{T}^{c, \ell} \big]} \leq C_{1} \sqrt{T} ( 2^{-\ell} h_{0} )^{\frac{3}{2}}, \qquad \forall \ell \geq 1 \] where $C_{1} > 0$ is a constant. From Theorem~\ref{thm:radon_nikodym}, for the choice of $h_{0} = C_{2} / \sqrt{T \log T }$, we obtain
            \begin{equation}
                \begin{split}
                    \abs{\E[\hatphi] - \E[\varPhi(X_{T})]} &= \abs{ \sum_{\ell = L + 1} ^{\infty} \E \Big[ \varPhi(\hatY_{T}^{f, \ell}) R_{T}^{f, \ell} - \varPhi(\hatY_{T}^{c, \ell}) R_{T}^{c, \ell} \Big] } \\
                    ~
                    &\leq C_{1} C_{2}^{3/2} T^{-1/4} \big( \log T \big)^{-3/4} \sum_{\ell = L + 1}^{\infty} 2^{-\frac{3}{2} \ell}.
                \end{split}
            \end{equation}
            where $C_{2} > 0$ is a constant independent of $T$. Then for the choice of $L$ given by
            \begin{equation*}
                L = \ceil{\frac{2}{3} \log_{2} \Big( \epsilon^{-1} T^{-1/4} \big( \log T \big)^{-3/4} \Big) + \frac{2}{3} \log_{2} (\sqrt{6} \widetilde{C}_{1})}, 
            \end{equation*}
            we achieve the bound $\tildeJ_{2} \leq \epsilon^{2} / 3$, where $\widetilde{C}_{1} \coloneqq C_{1} C_{2}^{3/2} / \big( 2^{3/2} - 1 \big)$, and $L$ is total number of levels of the multilevel estimator.

            Let $\cV_{\ell}$ denote the variance of the MLMC estimator under the change-of-measure on level $\ell$. Using the results from Theorem~\ref{thm:convergence_estimator}, there exists a constant $c_{1} > 0$ such that
            \begin{equation}
                \label{eq:variance_lipschitz}
                \begin{split}
                    \cV_{\ell} \coloneqq \Var \Big( \varPhi(\hatY_{T}^{f, \ell}) R_{T}^{f, \ell} - \varPhi(\hatY_{T}^{c, \ell}) R_{T}^{c, \ell} \Big) \leq c_{1} T (2^{-\ell} h_{0})^{3}, \qquad \forall \ell \ge 1,
                \end{split}
            \end{equation}
            Let $\cC_{\ell}$ denote the computational cost needed to generate numerical realizations of both the fine and the coarse trajectories on each level $\ell$ \[ \cC_{\ell} \leq c_{2} \frac{T}{2^{-\ell} h_{0}}, \] where $c_{2} > 0$ is a constant.

            To bound the first error component $\tildeJ_{1}$, we choose the number of samples on each level of the multilevel estimator in such a way that that the total cost is minimized under the constraint that the total variance of the estimator is $\epsilon^{2} / 3$. From the standard MLMC theory \cite[Theorem 2.1]{giles2015multilevel}, for the choice of $N_{\ell}$ given by
            \begin{equation*}
                N_{\ell} = \ceil{3 \epsilon^{-2} \Bigg( \sum_{\ell = 0}^{L} \sqrt{ \cV_{\ell} \cC_{\ell} } \Bigg) \sqrt{\frac{\cV_{\ell}}{\cC_{\ell}}} },
            \end{equation*}
            we achieve the bound $\tildeJ_{1} \leq \epsilon^{2} / 3$. The total computational cost involved in the implementation of the higher-order change-of-measure scheme is
            \[ \Cost = 3 \epsilon^{-2} \Bigg( \sum_{\ell = 0}^{L} \sqrt{\cV_{\ell} \cC_{\ell}} \Bigg)^{2} + \sum_{\ell = 0}^{L} \cC_{\ell}, \]
            where the total computational cost relates to the total number of operations required to numerically sample the fine and the coarse trajectories, and to compute the Radon--Nikodym derivatives over all levels, for all the Monte Carlo samples.

            Note that $\cV_{0} = \mathcal{O}(1)$ and $\cC_{0} = \cO (T h_{0}^{-1})$. Then for the choice of $h_{0} = C_{2} / \sqrt{T \log T}$, we obtain
            \begin{equation*}
                \begin{split}
                    \Cost &\leq 6 \epsilon^{-2} \Bigg( \cV_{0} \cC_{0} + \bigg( \sum_{\ell = 1}^{L} \sqrt{\cV_{\ell} \cC_{\ell}} \bigg)^{2} \Bigg) + \sum_{\ell = 0}^{L} \cC_{\ell} \\
                    &= \cO \Big( \epsilon^{-2} \abs{\log \epsilon}^{3/2} (\log \abs{\log \epsilon})^{1/2} \Big).
                \end{split}
            \end{equation*}
            
            \textbf{Discontinuous payoff function:}
            From Theorem~\ref{thm:convergence_estimator_discont} and for all $\xi > 0$, we have \[ \abs{\E \big[ \varPhi(\hatY_{T}^{f, \ell}) R_{T}^{f, \ell} - \varPhi(\hatY_{T}^{c, \ell}) R_{T}^{c, \ell} \big]} \leq C_{1} \sqrt{T} ( 2^{-\ell} h_{0} )^{\frac{3}{2} - \xi}, \qquad \forall \ell \geq 1 \] where $C_{1} > 0$ is a constant. From Theorem~\ref{thm:radon_nikodym}, for the choice of $h_{0} = C_{2} T^{-\frac{1}{\frac{3}{2} - \xi}}$, we obtain
            \begin{equation}
                \begin{split}
                    \abs{\E[\hatphi] - \E[\varPhi(X_{T})]} &= \abs{ \sum_{\ell = L + 1} ^{\infty} \E \Big[ \varPhi(\hatY_{T}^{f, \ell}) R_{T}^{f, \ell} - \varPhi(\hatY_{T}^{c, \ell}) R_{T}^{c, \ell} \Big] } \\
                    ~
                    &\leq C_{1} C_{2}^{3/2 - \xi} T^{-1/2} \sum_{\ell = L + 1}^{\infty} 2^{-(\frac{3}{2} - \xi) \ell}.
                \end{split}
            \end{equation}
            Then for the choice of $L$ given by
            \[ L = \ceil{\frac{1}{\frac{3}{2} - \xi} \log_{2} \Big( \epsilon^{-1} T^{-1/2} \Big) + \frac{1}{\frac{3}{2} - \xi} \log_{2}(\sqrt{6} \widetilde{C}_{1})},\]
            we obtain the bound $\tildeJ_{2} \leq \epsilon^{2} / 3$, where $\widetilde{C}_{1} \coloneqq C_{1} C_{2}^{3/2 - \xi} / (2^{3/2 - \xi} - 1)$. 

            From Theorem~\ref{thm:convergence_estimator_discont}, we know that there exists a constant $c_{1} > 0$ for all $\xi > 0$ such that the following holds
            \begin{equation}
                \label{eq:variance_discont}
                \begin{split}
                    \cV_{\ell} \coloneqq \Var \Big( \varPhi(\hatY_{T}^{f, \ell}) R_{T}^{f, \ell} - \varPhi(\hatY_{T}^{c, \ell}) R_{T}^{c, \ell} \Big) \leq c_{1} T (2^{-\ell} h_{0})^{\frac{3}{2} - \xi}, \qquad \forall \ell \ge 1,
                \end{split}
            \end{equation}
            Using a similar argument as for the uniformly Lipschitz continuous payoff function, the expected total computational cost for a discontinuous payoff function is
            \begin{equation*}
                \begin{split}
                    \Cost &\leq 6 \epsilon^{-2} \Bigg( \cV_{0} \cC_{0} + \bigg( \sum_{\ell = 1}^{L} \sqrt{\cV_{\ell} \cC_{\ell}} \bigg)^{2} \Bigg) + \sum_{\ell = 0}^{L} \cC_{\ell} \\
                    &= \cO \Big( \epsilon^{-2} \abs{\log \epsilon}^{5/3 + \xi} \Big)
                \end{split}
            \end{equation*}
            for all $\xi > 0$.
        \endgroup

    \textit{Acknowledgements.} We would like to thank Mike Giles for his suggestions for improving the manuscript. 
    
    \bibliography{./references.bib}
\end{document}